\documentclass{amsart}
\pdfoutput=1
\usepackage{float}
\restylefloat{figure}
\usepackage{mathtools}
\usepackage{amssymb}
\usepackage{mathrsfs}
\usepackage{bm}
\usepackage[all]{xy}
\UseComputerModernTips
\CompileMatrices
\usepackage{hyperref}
\usepackage{ifthen}
\usepackage{latexsym}
\usepackage{tikz}
\usepackage{tikz-3dplot}
\usetikzlibrary{backgrounds}
\allowdisplaybreaks
\numberwithin{equation}{section}
\newtheorem{theorem}[equation]{Theorem}
\newtheorem{lemma}[equation]{Lemma}
\newtheorem{proposition}[equation]{Proposition}
\newtheorem{corollary}[equation]{Corollary}
\newtheorem*{thm2}{Theorem A}
\theoremstyle{definition}
\newtheorem{definition}[equation]{Definition}
\newtheorem{example}[equation]{Example}
\newtheorem{examples}[equation]{Examples}

\theoremstyle{remark}
\newtheorem{remark}[equation]{Remark}

\newtheorem{scholium}[equation]{Scholium}
\renewcommand{\phi}{\varphi}
\newcommand{\bndry}{\partial}
\DeclareMathSymbol{\boxprod}{\mathbin}{AMSa}{"03} 
\DeclareMathSymbol{\mixprod}{\mathbin}{AMSa}{"4F} 

\newcommand{\dirsum}{\oplus}
\newcommand{\Dirsum}{\bigoplus}
\newcommand{\disjunion}{\sqcup}
\newcommand{\Disjunion}{\coprod}
\newcommand{\dual}{^\vee}

\newcommand{\hmtpc}{\simeq}
\newcommand{\homeo}{\approx}

\newcommand{\includesin}{\hookrightarrow}

\newcommand{\iso}{\cong}
\newcommand{\Mackey}[1]{{\underline {#1}}}

\newcommand{\onto}{\twoheadrightarrow}

\newcommand{\smsh}{\wedge}

\newcommand{\susp}{\Sigma}
\newcommand{\tensor}{\otimes}
\newcommand{\union}{\cup}

\newcommand{\Wedge}{\bigvee}

\newcommand{\C}{{\mathbb C}}

\newcommand{\F}{{\mathcal F}}

\newcommand{\NN}{{\mathbb N}}

\newcommand{\Oscr}{{\mathscr O}}

\newcommand{\PP}{\mathbb{P}}
\newcommand{\R}{{\mathbb R}}

\newcommand{\Z}{\mathbb{Z}}
\newcommand{\ZZ}{\mathbb{Z}}
\newcommand{\HS}{\Mackey H^{RO(G)}_G(S^0)}

\newcommand{\tE}{\tilde E}

\newcommand{\cw}{c_{\omega}}
\newcommand{\cxw}{c_{\chiw}}
\newcommand{\cwt}{\zeta_1}
\newcommand{\cxwt}{\zeta_0}
\newcommand{\cwd}{\widehat{c}_\omega}
\newcommand{\cxwd}{\widehat{c}_{\chiw}}
\newcommand{\Cpq}[2]{\C^{#1+#2\sigma}}
\newcommand{\Cq}[1]{\C^{#1\sigma}}
\newcommand{\Xpq}[2]{\PP(\Cpq{#1}{#2})}
\newcommand{\Xp}[1]{\PP(\C^{#1})}
\newcommand{\Xq}[1]{\PP(\Cq{#1})}
\newcommand{\chiw}{\chi\omega}

\newcommand{\I}{{\mathrm{I}}}
\newcommand{\II}{{\mathrm{II}}}
\newcommand{\III}{{\mathrm{III}}}
\newcommand{\IV}{{\mathrm{IV}}}

\newcommand{\orb}[1]{{\mathscr{O}_{#1}}}
\newcommand{\sorb}[1]{{\widehat{\mathscr{O}}_{#1}}}
\newcommand{\conc}[1]{\langle #1 \rangle}
\DeclareMathOperator*{\colim}{colim}
\DeclareMathOperator{\Hom}{Hom}

\DeclareMathOperator{\Pic}{Pic}
\DeclareMathOperator{\Inf}{Inf}
\newcommand{\tensorS}{\tensor_{\Mackey H_{C_2}^{RO(C_2)}(S^0;\Mackey A)}}
\begin{document}
\title{The $C_2$-equivariant cohomology of complex projective spaces}

\author{Steven R. Costenoble}
\address{Steven R. Costenoble, Department of Mathematics\\Hofstra University\\
   Hempstead, NY 11549}
\email{Steven.R.Costenoble@Hofstra.edu}
\author{Thomas Hudson}
\address{Thomas Hudson, Fachgruppe Mathematik und Informatik, Bergische Universit\"at Wuppertal, 42119 Wuppertal, Germany}
\email{hudson@math.uni-wuppertal.de}
\author{Sean Tilson}
\address{Sean Tilson, Fachgruppe Mathematik und Informatik, Bergische Universit\"at Wuppertal, 42119 Wuppertal, Germany}
\email{tilson@math.uni-wuppertal.de}
\begin{abstract}
We compute the equivariant cohomology of complex projective spaces associated to finite-dimensional representations of $C_2$,
using ordinary cohomology graded on representations of the fundamental groupoid, with coefficients in the Burnside ring Mackey functor.
This extension of the $RO(C_2)$-graded theory allows for the definition of Euler classes, which are used as generators
of the cohomology of the projective spaces.
As an application, we give an equivariant version of Bezout's theorem.
\end{abstract}
\makeatletter
\@namedef{subjclassname@2020}{%
  \textup{2020} Mathematics Subject Classification}
\makeatother
\maketitle

\tableofcontents



\section{Introduction}

In recent years there has been significant progress in equivariant homotopy theory.
Both its use in the solution of the Kervaire invariant one conjecture, see \cite{HHRKervaire}, and its relationship to motivic homotopy theory, see for example \cite{DuggerIsaksenEquivandReal}, have led to a renewed interest in the subject, particularly in computations.
In relation to our work, the works of Dugger,  Hogle, and Hazel are of particular interest.

In \cite{DuggerGrass}, Dugger was able to establish an equivariant version of the classical calculation of the cohomology of infinite real Grassmannians.
He provided a Borel style presentation of the cohomology of $Gr_k(\mathscr{U})$, where $\mathscr{U}$ is a complete $C_2$ universe. His methods are computational and involve investigating the cellular spectral sequence;
they are specific to the infinite case and they do not have an obvious geometric interpretation. The work of Hogle in \cite{Hogle} extends these computations to include some finite Grassmannians.
He was able to compute the equivariant cohomology of $Gr_k(\R^n\oplus \R^{\sigma})$ and $Gr_2(\R^n\oplus \R^{2\sigma})$, where $\R$ denotes the trivial representation and $\R^{\sigma}$ denotes the sign representation. 
In \cite{HazelFundamental}, Hazel introduces a definition of fundamental classes in $RO(C_2)$-graded cohomology and
shows that these classes satisfy familiar relations and generate the cohomology of surfaces, but do not have all the properties
one would wish for.
The relationship between Hazel's notion of fundamental class and the one we use here, introduced in \cite{CostenobleWanerBook},
still needs to be explored.
Note that all of the results of Dugger, Hogle, and Hazel are with coefficients in the constant Mackey functor $\Mackey{\mathbb{F}_{2}}$.

Other more classical $RO(G)$-graded computations were done by Lewis. In \cite{LewisCP}, Lewis computed the equivariant cohomology of complex projective spaces associated to complex representations of $C_p$. However, his computations do not provide a simple description of the product structure for all representations. Our present computation remedies this for $C_2$ by using an extension of the classical equivariant theory, this extension being the natural home for Euler classes.

In fact, we are using the $RO(\Pi B)$-graded cohomology, which was developed by the first author and Waner in \cite{CostenobleWanerBook}. This theory mixes both equivariant and parametrized homotopy theory and was previously used by the first author to compute the cohomology of complex projective spaces of complete universes for groups of prime order, see \cite{Co:BGU1preprint} for $p=2$ and \cite{Co:BGU1odd} for odd primes. Unlike in the classical situation, we use the infinite case in order to obtain information about the finite case: both the generators and the last two relations come from the infinite case via restriction.

More explicitly, we obtain the following description.

\begin{thm2}
Let $0 \leq p < \infty$ and $0 \leq q < \infty$ with $p+q > 0$.
As an algebra over $\Mackey H_{C_2}^{RO(G)}(S^0)$, we have that $\Mackey H_{C_2}^{RO(\Pi B)}(\PP(\C^p\oplus \C^{q\sigma})_+)$ is generated by the elements
$\cw$, $\cxw$, $\cwt$, and $\cxwt$, with $\cw^p$ infinitely divisible by $\cxwt$ and $\cxw^q$ infinitely divisible by $\cwt$.
The generators satisfy the following relations:
\begin{align*}
	\cw^p \cxw^q &= 0, \\
	\cwt \cxw - (1-\kappa)\cxwt \cw &= e^2 \qquad\text{and} \\
	\cxwt \cwt &= \xi.
\end{align*}
\end{thm2}

Here, $\C$ denotes the trivial complex $C_2$-representation and $\C^{\sigma}$ is $\C$ with the sign action, while the $C_2$-space $B$ is the infinite projective space $\PP(\C^{\infty}\oplus \C^{\infty\sigma})$. We write $\omega$ for the tautological line bundle over $\PP(\C^p\oplus \C^{q\sigma})$ and $\chi\omega$ for $\omega\tensor \C^{\sigma}$.
The elements $\cw$ and $\cxw$ denote their Euler classes, while $\cwt$ and $\cxwt$ denote related classes, whereas $\kappa$, $e^2$, and $\xi$ are elements of $\HS$. All cohomology is taken with coefficients in the Burnside ring Mackey functor $\Mackey A$. For more details about this cohomology, see  Section~\ref{subsec:cohompoint}.

In future works, we would like to compute the fully graded cohomology of other finite Grassmannians, that is $\Mackey H_G^{RO(\Pi X)}(Gr_m(\C^p\oplus \C^{q\sigma})_+;\Mackey A)$. As one of our aims is to generalize classical results of Schubert calculus and enumerative geometry to the equivariant setting, the second part of this paper is devoted to the presentation of an example of such a generalization, discussed below. We hope that this application will help convince the reader of the added benefits that can be reaped by considering the extended grading.  

\subsection{Equivariant Bezout's theorem}
In its classical 0-dimensional formulation, Bezout's theorem provides a formula which computes the number of points of intersection of $n$ algebraic hypersurfaces $H_1,\dots,H_n$ in $\C\PP^n$. If we denote by $O(1)$ the dual of the tautological bundle $\omega$, then each of these hypersurfaces can be viewed as the vanishing locus of an algebraic section of a line bundle $O(d_i)\rightarrow \C\PP^n$, where $O(m)$ is the $m$-th tensor power of $O(1)$.
One usually refers to the integer $d_i$ as the \textit{degree} of $H_i$. Bezout's theorem states that, if the intersection of the hypersurfaces is 0-dimensional, then its cardinality is given by $d$, the product of the degrees, provided that the points are counted with the appropriate notion of multiplicity.

This result follows easily once one knows the structure of the cohomology ring of $\C\PP^n$. The fundamental cohomology class $[H_i]^*$ associated to each hypersurface is given by the Euler class $e(O(d_i))$, and, in view of the hypothesis on the intersection, one can then write 
$$\Big[\bigcap_{i=1}^n H_i\Big]^*=\prod_{i=1}^n[H_i]^*=\prod_{i=1}^n e(O(d_n))
=\prod_{i=1}^n (d_i\cdot c)=d \cdot c^n,$$
where $c=e(O(1))$ is, up to a sign, the usual multiplicative generator of the cohomology of $\C\PP^n$; $c^n$ should be interpreted as the fundamental cohomology class of a point. 

In the equivariant case, one replaces the hypersurfaces with equivariant hypersurfaces and, as a consequence, the intersection will now inherit a $C_2$-action, giving a $C_2$-set, hence an element of the Burnside ring $A(C_2)$. 
If we consider $\Xpq{p}{q}$ with $n=p+q-1$, a computation similar in spirit to the one that we just outlined allows us to predict how many of the $d$ points are fixed by the action and how many pairs of points are swapped. In the case in which the equivariant hypersurfaces intersect transversely we obtain 
$$\Big[\bigcap_{i=1}^n H_i\Big]^*=\frac{d-A-B}{2}[C_2]^*+A[\Xp{}]^*+B[\Xq{}]^*,$$
where $A$ and $B$ are non-negative integers. Here $[C_2]^*$ is the fundamental cohomology class of a free orbit, whereas $[\Xp{}]^*$ and $[\Xq{}]^*$ are the fundamental classes of fixed points: the first point lies in the fixed component $\Xp{p}$, while the second belongs to $\Xq{q}$. The values of $A$ and $B$ depend on the exact composition of the family of hypersurfaces. In fact, it turns out that it becomes necessary to subdivide the line bundles on $\Xpq{p}{q}$ according to parity of the degree and the lack or presence of the tensor factor $\C^\sigma$. For the precise statement we refer the reader to Corollary \ref{cor:genericallyTransverse} and Theorem \ref{Theorem BezoutMult}, the latter of which addresses the more general case in which we do not assume transversality and it becomes necessary to count points with (equivariant) multiplicities.

\subsection{Motivation for extended grading}
The ordinary cohomology theory we use is an extension of the more familiar $RO(G)$-graded Bredon cohomology with coefficients in a Mackey functor.
The grading is extended to include all representations of the equivariant fundamental groupoid of the space in question, not simply the constant representations (the elements of $RO(G)$).
This extension results in a cohomology theory with several useful features: Thom isomorphisms, Poincar\'e duality, Euler classes, and a good theory of fundamental classes. 
Let us address the Euler classes and fundamental classes in particular. 

By a good theory of fundamental classes we mean, among other things, that the fundamental homology class of an equivariant manifold restricts to the fundamental class of the fixed submanifold under the fixed point functors with which equivariant homology is naturally packaged.
Similarly, a good notion of Euler classes should have the property that the Euler class of an equivariant vector bundle $V$ restricts to $e(V^{C_2})\in H^*(X^{C_2})$. 
As the fixed points of the base space may be disconnected,  the dimension of the fixed bundle can be nonconstant.
This indeed happens for $\Xpq{\infty}{\infty}$, the classifying space for $C_2$-equivariant complex line bundles. This failure of the dimension function of the fixed bundle to be constant means that the relevant Euler class may be nonhomogeneous. 
As homogeneous elements of the $RO(C_2)$-graded cohomology restrict to homogeneous elements under the fixed point functors, $RO(C_2)$-graded cohomology can not have a theory of Euler classes which restricts appropriately.
While nonhomogeneous elements are uncommon, we will see later how this allows us to interpret certain products in cohomology geometrically.
This behavior of the Euler class is fundamental and one of the key features of this extended cohomology.

Euler classes can also be used to explain why $RO(G)$-graded equivariant cohomology cannot provide the full picture. In \cite{MilnorStasheff}, Milnor and Stasheff gave a list of axioms that Chern classes satisfy and it is natural to look for an equivariant generalization of this list. 
One of the key axioms is that $c_1(\omega)\in H^2(\C\PP^{\infty};\ZZ)$ restricts along the natural map to a generator in the cohomology of $\C\PP^1$.
Equivariantly, there are two different spaces that play the role of $\C\PP^1$, they are the two representation spheres $S^{2\sigma} = \Xpq{1}{}$ 
and $S^2 = \Xp{2} = \Xq{2}$.
Their $RO(C_2)$-graded cohomology has generators in different gradings, $2\sigma$ and $2$, respectively (via the suspension isomorphism).
Any theory of Chern classes based solely on $RO(C_2)$-graded Bredon cohomology would not be able to satisfy the reasonable condition
that the first Chern class restrict to generators of the cohomology of both $S^{2\sigma}$ and $S^2$.
The extended grading developed in \cite{CostenobleWanerBook} precisely addresses this.
The Euler class of $\omega$ in \cite{CostenobleShortC2} (morally, the first Chern class since $\omega$ is a line bundle) 
lives outside of the $RO(G)$-grading and restricts as described above.
Indeed, the grading of $e(\omega)$ is the representation of the fundamental groupoid induced by $\omega$, as described in Section~\ref{sub ROPi}. 


\subsection{Outline}

We begin in Section \ref{sec: background} with some background material.
This includes recalling facts about $RO(C_2)$-graded cohomology and introducing the extended $RO(\Pi B)$-graded theory and some of its more relevant properties.
Specifically, we mention various functorialities that this theory enjoys and which we will need.

In Section \ref{sec: Statement} we give the statement of our main Theorem as well as the proof of some initial cases. The general case is then handled in Section \ref{sec: Proof} by induction and the computation of certain push-forward maps. A key ingredient is the fact that, in order to verify that a given set of elements constitutes a basis for the equivariant cohomology, it is sufficient to check that it produces a basis for both the underlying cohomology and that of the fixed points, a line of argument suggested to us by a referee.
We first use this to establish the additive structure and then use push-forward maps to understand the aspects of the multiplicative structure that are not inherited from the infinite case $\Xpq{\infty}{\infty}$.
In Section \ref{sec: leftovers}, we consider the cases in which one of $p$ or $q$ is infinite, the comparison of our work to Lewis' computation in \cite{LewisCP},
and the analogues of our main theorem with other coefficient systems.

In Section \ref{sec: Line bundles}, we study the Picard group of $\Xpq{p}{q}$ as well as the Euler classes of all line bundles and their sums. 
Finally, in Section \ref{sec: Application}, we make use of the computation in the previous section to obtain an equivariant refinement of Bezout's theorem for $0$-dimensional intersections.
We also obtain a completely numerical necessary condition for a sum of line bundles to have an equivariant section transverse to the zero section.

We have included two appendices to help with the reading of this paper.
The first is a glossary of terms and notations.
The second is a table presenting the behaviour of elements belonging either to the equivariant cohomology of a point or of $\Xpq{\infty}{\infty}$ under the different restriction maps.
We recommend having these tables in hand when beginning to look at our computations.

\subsection*{Acknowledgements}
The authors would like to express their gratitude to the anonymous referee whose comments significantly improved the paper in many ways, in particular by suggesting a way to simplify the proof of the main result and encouraging 
us to look into the possibility of extending Bezout's theorem as a possible application.   
The first author wishes to thank the other two authors and the Bergische Universit\"{a}t Wuppertal for their gracious hospitality during his visit in summer 2018.
The second and third author would like to thank Jens Hornbostel for encouraging the development of this joint project and fostering an inquisitive scientific atmosphere.
The second and third author were partially supported by the DFG through the SPP 1786: \textit{``Homotopy theory and Algebraic Geometry''}, Project number 405468058: $C_2$-equivariant Schubert calculus of homogeneous spaces.
The research was conducted in the framework of the research training group
\emph{GRK 2240: Algebro-Geometric Methods in Algebra, Arithmetic and Topology},
which is funded by the DFG.



\section{Background}
\label{sec: background}

\subsection{$RO(C_2)$-graded cohomology}\label{subsec:cohompoint}

The construction of a $G$-equivariant ordinary cohomology theory graded on the representation ring $RO(G)$
was announced in \cite{LMM:roghomology} with the details first appearing in print
in \cite{Alaska}. It takes as coefficients a Mackey functor, which we think of as
a contravariant functor on the stable orbit category $\sorb G$. For $G = C_2$, $\sorb{C_2}$
takes the following form:
\[
 \xymatrix{
  C_2/C_2 \ar@(ur,ul)[]_{A(C_2)} \ar@/^/[d]^{\tau} \\
  C_2/e \ar@/^/[u]^{\rho} \ar@(dl,dr)[]_{\Z[C_2]} \\
 }
\]
Here, $\rho$ is the projection and we mean that the group of maps $C_2/e \to C_2/C_2$ is
the free abelian group on $\rho$.
Similarly, $\tau$ is the transfer map and the group of maps $C_2/C_2 \to C_2/e$
is free abelian on $\tau$.
$A(C_2)$ is the Burnside
ring of $C_2$, 
the Grothendieck group of finite $C_2$-sets under disjoint union, with the product (composition)
induced by the product of $C_2$-sets;
this is isomorphic to the ring of stable self maps of $C_2/C_2$. 
Explicitly, we have
\[
 A(C_2) \iso \Z[g]/\langle g^2 - 2g \rangle,
\]
where $g = [C_2]$ is the class of the orbit space $C_2/e$. It will be convenient to write
\[
 \kappa = 2 - [C_2].
\]
We have $\kappa^2 = 2\kappa$ and
\[
 A(C_2) \iso \Z[\kappa] / \langle \kappa^2 - 2\kappa \rangle.
\]
Finally, the ring of self-maps of $C_2/e$ is the group ring
\[
 \Z[C_2] \iso \Z[t]/\langle t^2 - 1 \rangle.
\]

The coefficient system we will use throughout most of this paper is
the Burnside ring Mackey functor $\Mackey A$, which can
be pictured as
\[
 \xymatrix{
		A(C_2) \ar@/_/[d]_{\epsilon} \\
		\Z \ar@/_/[u]_{\cdot [C_2]} \ar@(dl,dr)[]_{1}
 }
\]
Here, we mean that $\Mackey A(C_2/C_2) = A(C_2)$ and $\Mackey A(C_2/e) = \Z$.
We have $\rho^* = \epsilon$, where $\epsilon$ counts the number of points
in a finite $C_2$-set, so $\epsilon(1) = 1$, $\epsilon([C_2]) = 2$, and $\epsilon(\kappa) = 0$.
Similarly, $\tau^* = \cdot[C_2]$, which is the map that takes $n$ to $n[C_2] = 2n - n\kappa$.
Note also that $t^*$ acts as the identity on $\Mackey A(C_2/e) = \Z$.

$RO(G)$-graded ordinary cohomology is constructed in \cite{Alaska} using
chain complexes coming from $G$-CW structures on equivariant spectra,
but the result can be stated in its represented form as follows: 
There is a unique (up to uniquely determined equivalence)
Eilenberg-Mac\,Lane $C_2$-spectrum $H\Mackey A$ with the property that
\[
 \Mackey \pi^{C_2}_n (H\Mackey A) \iso
  \begin{cases}
  	\Mackey A & \text{if $n = 0$ and} \\
	0 & \text{if $n\neq 0$}
  \end{cases}
\]
for $n\in \Z$.
Here, $\Mackey \pi^{C_2}_n (E)$ denotes the Mackey functor defined by
\[
 \Mackey \pi_n^{C_2}(E)(C_2/J) = [(C_2/J)_+\smsh S^n, E]^{C_2}
  \iso [S^n, E]^J
\]
for $J = C_2$ or $e$.

We grade the cohomology theory represented by $H\Mackey A$
on the representation ring $RO(C_2)$, which is the free abelian
group on two generators, the trivial representation $\R$ and the nontrivial representation $\R^\sigma$.
We write elements of $RO(C_2)$ as 
\[
 a + b\sigma = \R^{a + b\sigma} = \R^a\dirsum (\R^\sigma)^b.
\]
We shall also consider the cohomology theory to take values in Mackey functors.
Explicitly, for a $C_2$-spectrum $X$, we write
\[
 \Mackey H_{C_2}^{a+b\sigma}(X;\Mackey A)(C_2/J)
  = [(C_2/J)_+\smsh X, \susp^{a+b\sigma}H\Mackey A]^{C_2}
  \iso [X, \susp^{a+b\sigma}H\Mackey A]^J.
\]
Note that the value at level $C_2/e$ is just the nonequivariant cohomology of $X$ with $\Z$ coefficients,
while the value at $C_2/C_2$ is what is usually thought of as the equivariant cohomology, which,
when we need it, we shall write as
\[
 H_{C_2}^{a+b\sigma}(X;\Mackey A) = \Mackey H_{C_2}^{a+b\sigma}(X;\Mackey A)(C_2/C_2).
\]
We will usually abbreviate to $\Mackey H_{C_2}^{a+b\sigma}(X)$
with the coefficient system understood to be $\Mackey A$.
We write
\[
 \Mackey H_{C_2}^{RO(C_2)}(X) = \Dirsum_{\alpha\in RO(C_2)} \Mackey H_{C_2}^{\alpha}(X)
\]
for the whole $RO(C_2)$-graded Mackey-valued ordinary cohomology of $X$.
When $X$ is a based $C_2$-space, we shall write
\[
 \Mackey H_{C_2}^{RO(C_2)}(X) = \Mackey H_{C_2}^{RO(C_2)}(\susp^\infty_{C_2} X).
\]

The cohomology $\Mackey H_{C_2}^{RO(C_2)}(X)$ is a module
over the ring $\Mackey H_{C_2}^{RO(C_2)}(S^0)$. In $\Z$-grading this ring is just $\Mackey A$
in degree 0 and 0 elsewhere, but in gradings outside of $\Z$ it is much more interesting.
The calculation was first done by Stong, in an unpublished manuscript, and
was first published by Lewis in \cite{LewisCP}.
To describe the result we first need to give names to the several Mackey functors that
appear:
\[
\begin{array}{rcrc}
 \Mackey \Z\colon & \xymatrix{
		\Z \ar@/_/[d]_{1} \\
		\Z \ar@/_/[u]_{2} \ar@(dl,dr)[]_{1}
	}
&\qquad\qquad
 \Mackey \Z_{-}\colon & \xymatrix{
		0 \ar@/_/[d] \\
		\Z \ar@/_/[u] \ar@(dl,dr)[]_{-1}
	 }
\\ \\
 \Mackey \Z'\colon & \xymatrix{
		\Z \ar@/_/[d]_{2} \\
		\Z \ar@/_/[u]_{1} \ar@(dl,dr)[]_{1}
	}
&\qquad\qquad
 \Mackey \Z'_{-}\colon & \xymatrix{
		\Z/2 \ar@/_/[d]_{0} \\
		\Z \ar@/_/[u]_{\pi} \ar@(dl,dr)[]_{-1}
	 }
\\ \\
 \conc\Z \colon & \xymatrix{
		\Z \ar@/_/[d] \\
		0 \ar@/_/[u] \ar@(dl,dr)[]_{}
	   }
&\qquad\qquad
 \conc{\Z/2}\colon & \xymatrix{
			\Z/2 \ar@/_/[d] \\
			0 \ar@/_/[u] \ar@(dl,dr)[]
	     }

\end{array}
\]
The Mackey functor $\Mackey \Z$
is usually referred to as the constant $\Z$ functor.
We can now picture the $RO(C_2)$-graded cohomology of $S^0$ as in Figures~\ref{fig:cohomPoint} and \ref{fig:cohomPointGens}.
There are various choices made in the literature as to what axes to use, but throughout
this paper we will use axes that place the group with grading $a+b\sigma$
at $(a, b)$.
Figure~\ref{fig:cohomPoint} shows the Mackey functors while 
Figure~\ref{fig:cohomPointGens} gives names to generators of those functors.
\begin{figure}
\[\def\objectstyle{\scriptstyle}
 \xymatrix@!0@R=5ex@C=2.5em{
  & & & & &\phantom{a}\ b & & & & & \\
  & \conc{\Z/2} &\cdot& \conc{\Z/2} &\cdot& \conc\Z &\cdot& \cdot &\cdot& \cdot & \cdot \\
  & \Mackey\Z &\cdot& \conc{\Z/2} &\cdot& \conc\Z &\cdot& \cdot &\cdot& \cdot & \cdot \\
  & \cdot &\Mackey\Z_{\mathrlap{-}}& \conc{\Z/2} &\cdot& \conc\Z &\cdot& \cdot &\cdot& \cdot & \cdot \\
  & \cdot &\cdot& \Mackey\Z &\cdot& \conc\Z &\cdot& \cdot &\cdot& \cdot & \cdot \\
  & \cdot &\cdot& \cdot &\Z_{\mathrlap{-}}& \conc\Z &\cdot& \cdot &\cdot& \cdot & \cdot \\
  \ar@{-}'[rrrrr][rrrrrrrrrrr]
   &  &  &  &  & \Mackey A &  &  &  &  &  & a\\
  & \cdot &\cdot& \cdot &\cdot& \conc\Z &\Mackey \Z_{\mathrlap{-}}& \cdot & \cdot & \cdot & \cdot  \\
  & \cdot &\cdot& \cdot &\cdot& \conc\Z &\cdot& \Mackey \Z' & \cdot & \cdot & \cdot  \\
  & \cdot &\cdot& \cdot &\cdot& \conc\Z &\cdot& \cdot & \Mackey \Z'_{\mathrlap{-}} & \cdot & \cdot  \\
  & \cdot &\cdot& \cdot &\cdot& \conc\Z &\cdot& \cdot & \conc{\Z/2} & \Mackey \Z' &  \cdot  \\
  & \cdot &\cdot& \cdot &\cdot& \conc\Z &\cdot& \cdot & \conc{\Z/2} & \cdot &  \Mackey \Z'_{\mathrlap{-}}  \\
  & \cdot &\cdot& \cdot &\cdot& \conc\Z &\cdot& \cdot & \conc{\Z/2} & \cdot &  \conc{\Z/2}  \\
  & & & & & \ar@{-}'[u]'[uu]'[uuu]'[uuuu]'[uuuuu]'[uuuuuu]'[uuuuuuu]'[uuuuuuuu]'[uuuuuuuuu]'[uuuuuuuuuu]'[uuuuuuuuuuu]'[uuuuuuuuuuuu][uuuuuuuuuuuuu]
 }
\]
\caption{$\protect\Mackey H_{C_2}^{RO(C_2)}(S^0)$ at $a+b\sigma$}\label{fig:cohomPoint} 
\end{figure}

\begin{figure}
\[\def\objectstyle{\scriptstyle}
 \xymatrix@!0@R=5ex@C=3em{
  & & & & & \phantom{a}\  b & & & & & \\
  & e\xi^2 &\cdot& e^3\xi &\cdot& e^5 &\cdot& \cdot &\cdot& \cdot & \cdot \\
  & \xi^2 &\cdot& e^2\xi &\cdot& e^4 &\cdot& \cdot &\cdot& \cdot & \cdot \\
  & \cdot &(\iota^3) & e\xi &\cdot& e^3 &\cdot& \cdot &\cdot& \cdot & \cdot \\
  & \cdot &\cdot& \xi &\cdot& e^2 &\cdot& \cdot &\cdot& \cdot & \cdot \\
  & \cdot &\cdot& \cdot &(\iota)& e &\cdot& \cdot &\cdot& \cdot & \cdot \\
  \ar@{-}'[rrrrr][rrrrrrrrrrr]
   &  &  &  &  & 1 &  &  &  &  &  & a\\
  & \cdot &\cdot& \cdot &\cdot& e^{-1}\kappa &(\iota^{-1})& \cdot & \cdot & \cdot &  \cdot  \\
  & \cdot &\cdot& \cdot &\cdot& e^{-2}\kappa &\cdot& (\iota^{-2}) & \cdot & \cdot & \cdot  \\
  & \cdot &\cdot& \cdot &\cdot& e^{-3}\kappa &\cdot& \cdot & (\iota^{-3}) & \cdot &  \cdot  \\
  & \cdot &\cdot& \cdot &\cdot& e^{-4}\kappa &\cdot& \cdot & e^{-1}\tau(\iota^{-3}) & (\iota^{-4}) &  \cdot  \\
  & \cdot &\cdot& \cdot &\cdot& e^{-5}\kappa &\cdot& \cdot & e^{-2}\tau(\iota^{-3}) & \cdot &  (\iota^{-5})  \\
  & \cdot &\cdot& \cdot &\cdot& e^{-6}\kappa &\cdot& \cdot & e^{-3}\tau(\iota^{-3}) & \cdot &  e^{-1}\tau(\iota^{-5})  \\
  & & & & & \ar@{-}'[u]'[uu]'[uuu]'[uuuu]'[uuuuu]'[uuuuuu]'[uuuuuuu]'[uuuuuuuu]'[uuuuuuuuu]'[uuuuuuuuuu]'[uuuuuuuuuuu]'[uuuuuuuuuuuu]'[uuuuuuuuuuuuu]
 }
\]
\caption{The generators of $\protect\Mackey H_{C_2}^{RO(C_2)}(S^0)$ at $a+b\sigma$}\label{fig:cohomPointGens}
\end{figure}

The element $\iota \in \Mackey H_{C_2}^{-1+\sigma}(S^0)(C_2/e)$ 
(the parentheses in Figure~\ref{fig:cohomPointGens} are to indicate that $\iota$ lives at level $C_2/e$)
is an invertible element representing
the fact that, nonequivariantly, $a+b\sigma = a+b$ so that the 
only non-zero group in the nonequivariant cohomology of a point
appears duplicated along the line $a+b = 0$ at level $C_2/e$.

The element $e \in \Mackey H_{C_2}^{\sigma}(S^0)(C_2/C_2)$ is the Euler class of $\R^\sigma$, 
the image of $1$ under the map
\[
 H_{C_2}^0(S^0) \iso H_{C_2}^\sigma(S^\sigma) \to H_{C_2}^\sigma(S^0),
\]
where the last map is restriction along the inclusion $S^0\to S^\sigma$.
Note that $\rho(e) = 0$. (For simplicity of notation, we will from now on write $\rho$ for the structural map $\rho^*$ in a Mackey functor and $\tau$ for $\tau^*$.) We also have that $\kappa e = 2e$.

The element $\xi\in \Mackey H_{C_2}^{-2 + 2\sigma}(S^0)(C_2/C_2)$ satisfies $\rho(\xi) = \iota^2$.
This implies that 
\[
 \xi\cdot\tau(\iota^{k}) =  \tau(\rho\xi\cdot \iota^k) = \tau(\iota^{k+2}).
\]
We also have that $\kappa\xi = \xi$ and $2e\xi = 0$.

Of the elements with $a + b < 0$, the notation is to indicate the following identities:
\begin{align*}
 e^m\cdot e^{-n}\kappa &= e^{m-n}\kappa
\\
 e^m\cdot e^{-n}\tau(\iota^{-(2k+1)})
  &= \begin{cases}
  		e^{m-n}\tau(\iota^{-(2k+1)}) & \text{if $m \leq n$} \\
		0 & \text{if $m > n$}
	 \end{cases}
\\
 \xi^m \cdot e^{-n}\tau(\iota^{-(2k+1)})
  &= \begin{cases}
  		e^{-n}\tau(\iota^{-(2(k-m)+1)}) & \text{if $m < k$} \\
		0 & \text{if $m \geq k$.}
	 \end{cases}
\end{align*}
More explanation and a full derivation appears in \cite{Co:BGU1preprint}.
Although $RO(C_2)$-graded cohomology rings are in general anti-commutative,
$\Mackey H_{C_2}^{RO(C_2)}(S^0)$ is strictly commutative.

\begin{remark}[Relationship with other common notations]\label{rem:othernotations}

The related and somewhat simpler ring $\Mackey H_{C_2}^{RO(C_2)}(S^0;\Mackey \Z)$
has been used extensively in recent literature. There is a map
$\epsilon_*\colon \Mackey H_{C_2}^{RO(C_2)}(S^0;\Mackey A) \to \Mackey H_{C_2}^{RO(C_2)}(S^0;\Mackey \Z)$
induced by the map $\Mackey A\to \Mackey \Z$ that is $\epsilon$ at level $C_2/C_2$.
Elements of $\Mackey H_{C_2}^{RO(C_2)}(S^0;\Mackey \Z)$ have been given
a multitude of different names.
Under $\epsilon_*$, $e$ maps to the element called
$a$ in \cite{HuKr:adamsnovikov}, $a_\sigma$ in \cite{HHRKervaire}, 
and $\rho$ in the motivic literature, for example in \cite{DuggerGrass}
(in which the coefficients are further reduced to the constant $\Z/2$ functor).
The element $\xi$ maps to the element called $u_{2\sigma}$ in \cite{HHRKervaire} and
the element $\tau^2$ in \cite{DuggerGrass},
whereas $\tau(\iota^{-2})$ maps to the element $\theta$ in \cite{DuggerGrass}.
In \cite{Dug:AHSSforKR}, Dugger uses constant $\Z$ coefficients but a completely different
set of names. The element he calls $\theta$ there is the image of our $\tau(\iota^{-3})$,
while the image of $\tau(\iota^{-2})$ is called $\alpha$.
\end{remark}

\subsection{$RO(\Pi B)$-graded cohomology}\label{sub ROPi}

As discussed in the introduction, $RO(G)$-graded cohomology is inadequate
for the study of equivariant vector bundles and manifolds.
 A basic reason for this is the lack of a Thom isomorphism theorem for general vector bundles
and the lack of Poincar\'e duality for general manifolds.
The Thom isomorphism should restrict to the local suspension isomorphisms induced by the fibers of the bundle.
But, the suspension isomorphism induced by a fiber $V$ induces a shift in grading by $V$,
and this shift is going to vary as we move around the base space because the representation $V$ may vary.
So the global dimension shift cannot be captured by a single representation of the 
group---we need a grading that allows for these varying representations.

This was part of the impetus behind the first author's work with J. Peter May and
Stefan Waner that led to \cite{CMW:orientation}, which introduced the notion of
the representation ring of the equivariant fundamental groupoid of a space,
and we should think of the dimension of a bundle or a manifold as living in this ring. 
For the general definition we refer the reader to \cite{CMW:orientation}; until the end of this subsection we will restrict our attention to the special case we shall use throughout the remainder of this paper.
We shall consider only $C_2$-spaces parametrized by a particular base space, the $C_2$-equivariant
classifying space modeled by the infinite complex projective space
\[
 B = B_{C_2}U(1) = \PP(\C^\infty\dirsum \C^{\infty\sigma}).
\]
Here, $\C$ denotes the trivial complex representation of $C_2$ while $\C^\sigma$
is the nontrivial complex representation.
We shall grade all of our cohomology groups on the representation ring of the fundamental groupoid of $B$,
which we now describe.

The equivariant fundamental groupoid of $B$ is a category $\Pi_{C_2}B$
(which we shall usually write as $\Pi B$, the group $C_2$ being understood) fibered over
the (unstable) orbit category $\orb{C_2}$. The fiber over $C_2/e$ is the nonequivariant fundamental groupoid
of $B$ while the fiber over $C_2/C_2$ is the nonequivariant fundamental groupoid of
\[
 B^{C_2} = \Xp{\infty} \disjunion \Xq{\infty}.
\]
We can picture a skeleton of $\Pi B$ as follows:
\[
 \xymatrix@!C=.5em{
  b_0 && b_1 & \qquad\qquad & C_2/C_2 \\
  &b \ar[ul] \ar[ur] \ar@(dl,dr)[]_t &&& C_2/e \ar[u]_\rho \ar@(dl,dr)[]_t \\
  &\Pi  B &&& \orb{C_2}
 }
\]
Because $B$ is nonequivariantly simply connected, a skeleton of its nonequivariant
fundamental groupoid can be taken as a single point $b$. On the other hand, $B^{C_2}$ consists
of two simply connected components, so a skeleton of its fundamental groupoid consists of
two points, $b_0$ and $b_1$. There are then unique maps in the category from $b\to b_k$
covering the projection $\rho$, and a unique self-map $t\colon b\to b$ covering the
switch map $t\colon C_2/e \to C_2/e$, with $t^2 = 1$.

A {\em representation} of $\Pi  B$ is a functor $\gamma$ over $\orb{C_2}$ from $\Pi B$ to the category
of $C_2$-vector bundles over orbits and $C_2$-homotopy classes of equivariant bundle maps. That is, it assigns to $b$ a $C_2$-bundle over $C_2/e$ and to $b_0$ and $b_1$
$C_2$-bundles over $C_2/C_2$. These will have the form
\begin{align*}
 \gamma(b) &= C_2 \times \R^n \\
 \gamma(b_0) &= \R^{k_0 + \ell_0\sigma} \\
 \gamma(b_1) &= \R^{k_1 + \ell_1\sigma}.
\end{align*}
We will also have equivariant bundle maps $C_2\times \R^n \to \R^{k_k + \ell_k\sigma}$, which implies that
\[
 k_0 + \ell_0 = k_1 + \ell_1 = n.
\]
Further, to be compatible with the map $\gamma(t)$, we must have
that $\ell_0$ and $\ell_1$ have the same parity. These conditions completely determine the possible
representations. The ring $RO(\Pi B)$ is then the Grothendieck group of representations
of $\Pi B$ under the operation of direct sum.
The product in this ring is induced by the tensor product of bundles, but we will use primarily
the additive structure.
Explicitly, $RO(\Pi B)$ is a free abelian group generated by $1 = \R$, $\sigma = \R^\sigma$,
and the (virtual) representation $\Omega_1$ that takes value 0 on $b_0$ and $\R^{2\sigma}-\R^2$ on $b_1$.
It is convenient to also name a representation $\Omega_0$ that takes value $\R^{2\sigma}-\R^2$
on $b_0$ and 0 on $b_1$, and we have then that $\Omega_0 + \Omega_1 = 2\sigma - 2$.

Any vector bundle over $B$ has an associated representation in $RO(\Pi B)$ given by
considering its fibers over $b$, $b_0$, and $b_1$. For example, the tautological
complex line bundle $\omega$ over $B$ has fiber $\C = \R^2$ over $b_0$ and $\C^\sigma = \R^{2\sigma}$
over $b_1$, so its associated representation, which we write as $\omega^*$
or simply $\omega$ when the context is clear, is
\[
 \omega^* = 2 + \Omega_1.
\]
We will use the fact that $\{1, \sigma, \omega\}$ is also a basis for $RO(\Pi B)$.
We can also consider the line bundle $\chiw = \omega\tensor_{\C}\C^\sigma$.
Its associated representation is
\[
 \chiw ^* = 2 + \Omega_0.
\]
Note that $\omega^* + \chiw^* = 2 + 2\sigma$. 

For $\gamma\in RO(\Pi B)$ we shall write $|\gamma|\in\Z = RO(\Pi_e B)$ for the dimension of the underlying representation $\gamma(b)$,
so
\begin{align*}
    |1| &= 1 \\
    |\sigma| &= 1 \\
    |\Omega_0| &= 0 \\
    |\Omega_1| &= 0 \\
    |\omega| &= 2.
\end{align*}
We shall write $\gamma^{C_2} \in \Z\dirsum\Z = RO(\Pi_e B^{C_2})$ for the pair consisting of the dimensions of
$\gamma(b_0)^{C_2}$ and $\gamma(b_1)^{C_2}$, respectively, so
\begin{align*}
    1^{C_2} &= (1,1) \\
    \sigma^{C_2} &= (0,0) \\
    \Omega_0^{C_2} &= (-2,0) \\
    \Omega_1^{C_2} &= (0,-2) \\
    \omega^{C_2} &= (2,0).
\end{align*}

The $RO(\Pi B)$-graded ordinary cohomology is defined on $C_2$-ex-spaces over $B$,
that is, triples $(X,p,\sigma)$ where $X$ is a $C_2$-space, $p\colon X\to B$ is a $C_2$-map,
and $\sigma\colon B\to X$ is a section of $p$,
so that $p\sigma = 1$. If $(X,p)$ is simply a $C_2$-space over $B$,
we form the ex-space $(X,p)_+ = X_+ = (X\disjunion B,p,\sigma)$, with $p$ extended to be the identity
on $B$ and $\sigma$ the section that takes $B$ identically to the added copy of itself.
The groups $\Mackey H_{C_2}^\gamma(X)$ are defined in \cite{CostenobleWanerBook}
via chain complexes, but we can describe how they are represented fairly simply.
For each $\gamma\in RO(\Pi B)$ there is a unique $C_2$-spectrum $H\Mackey A^\gamma$ parametrized by $B$
with the property that the fiber over $b$ is $(C_2)_+\smsh \susp^{|\gamma(b)|} H\Z$ while
the fiber over $b_k$ is $\susp^{\gamma(b_k)}H\Mackey A$.
(The development in \cite{CostenobleWanerBook} used \cite{MaySig:parametrized}
for the foundations of parametrized stable homotopy theory. Because that foundation
does not support a good theory of CW parametrized spectra, there is the possibility
that $H\Mackey A^\gamma$ is determined only up to non-unique equivalence.)

We can take advantage of the fact that all elements of $RO(\Pi B)$ come from vector bundles
over $B$ (something that is not true of all $C_2$-spaces) to be more explicit: For each $n\in\Z$ we
can write
\[
 H\Mackey A^{n\omega} = \susp_B^{n\omega} r^*H\Mackey A
\]
where $r\colon B\to *$ is projection to a point, so $r^*$ is the induced pullback
functor taking non-parametrized $C_2$-spectra to spectra over $B$.
Here, for $n\geq 0$, $\susp_B^{n\omega}$ denotes the fiberwise smash product over $B$ with the
one-point compactifications of the fibers of $n\omega$; for $n < 0$ it is the corresponding desuspension.
With future work in mind in which we may use base spaces for which not all representations
of their fundamental groupoids come from bundles, we will avoid using this description
of $H\Mackey A^{n\omega}$, but the reader is free to make this substitution to
simplify some of the arguments.

For $X$ any $C_2$-ex-space over $B$ we then have
\begin{align*}
 \Mackey H_{C_2}^{a+b\sigma+n\omega}(X)(C_2/J) 
 &= [(C_2/J)_+\smsh\susp_B^\infty X, \susp^{a+b\sigma}H\Mackey A^{n\omega}]^{C_2}_B \\
 &= [\susp_B^\infty X, \susp^{a+b\sigma}H\Mackey A^{n\omega}]^J_B
\end{align*}
for $J = C_2$ or $e$.
The development in \cite{CostenobleWanerBook} makes it clear that these groups depend only
on the representations in $RO(\Pi B)$, not how they may (or may not) be represented by vector bundles.
We write
\[
 \Mackey H_{C_2}^{RO(\Pi B)}(X) = \Dirsum_{a+b\sigma+n\omega} \Mackey H_{C_2}^{a+b\sigma+n\omega}(X).
\]
This is a module over $H_{C_2}^{RO(C_2)}(S^0)$ in the obvious way, viewing $RO(C_2)$ as a subgroup
of $RO(\Pi B)$.

There is also an associated homology theory defined by
\[
 \Mackey H^{C_2}_{a+b\sigma+n\omega}(X)(C_2/J) 
  = \Big[(C_2/J)_+\smsh S^{a+b\sigma}, r_!(X\smsh_B H\Mackey A^{-n\omega})\Big]^{C_2}
\]
where $r_!$ is the left adjoint to $r^*$. On a $C_2$-ex-space $Y$, the effect of $r_!$
is to identify all points in the image of the section $B\to Y$ to a single point,
giving a based (non-parametrized) $C_2$ space.
It may help the reader to take advantage of the fact that $\omega$ is a bundle and think of
\[
 r_!(X\smsh_B H\Mackey A^{-n\omega})
  = r_!(X\smsh_B \susp_B^{-n\omega} r^*H\Mackey A)
  \hmtpc r_!(\susp_B^{-n\omega}X) \smsh H\Mackey A.
\]

\subsection{Change of base}

For the remainder of this section, $B$ will denote a general base space. 
The fundamental change of base result is this: Given a map of $C_2$-spaces $f\colon A\to B$
and an ex-space $(X,p,\sigma)$ over $A$, we can form $f_! X = (f_! X,p',\sigma')$, 
the ex-space over $B$ given by
taking the pushout in the top square of the following diagram:
\[
 \xymatrix{
	A \ar[r]^f \ar[d]_\sigma & B \ar[d]^{\sigma'} \\
	X \ar[r] \ar[d]_p & f_! X \ar[d]^{p'} \\
	A \ar[r]_f & B
 }
\]
We then have
\[
 \Mackey H_{C_2}^\gamma(f_! X) \iso \Mackey H_{C_2}^{f^*\gamma}(X)
\]
for $\gamma\in RO(\Pi B)$. (This is \cite[3.8.2]{CostenobleWanerBook}.) 
In particular, if $X$ is any space over $B$ and
we take $A = X$ and $p\colon X\to A$ the identity,
we see that the Mackey functor $\Mackey H_{C_2}^\gamma(f_! X_+)$ depends only on
$f^*\gamma\in RO(\Pi X)$, not $\gamma$ itself.
As a result, for every $\gamma\in \ker f^*$, 
we will have an isomorphism in $\Mackey H_{C_2}^{RO(\Pi B)}(X_+)$ of the form
\[
 \Mackey H_{C_2}^0(X_+) \iso \Mackey H_{C_2}^\gamma(X_+)
\]
induced by multiplication by a unit in $\Mackey H_{C_2}^\gamma(X_+)$ that is
the image of $1$ under the isomorphism above.

\subsection{The Thom isomorphism and Poincar\'e duality}

If $\eta$ is a vector bundle over a $C_2$-space $X$, considered as a space over itself,
we now have, for any $\gamma\in RO(\Pi X)$,
\begin{align*}
 \Mackey H_{C_2}^{\gamma}(\susp_X^\eta X_+)
  &= [\susp_X^\infty \susp_X^\eta X_+, H\Mackey A^\gamma]^{C_2}_X \\
  &\iso [\susp_X^\infty X_+, \susp_X^{-\eta}H\Mackey A^\gamma]^{C_2}_X \\
  &= \Mackey H_{C_2}^{\gamma - \eta}(X_+),
\end{align*}
which is the Thom isomorphism theorem.
This isomorphism is given by multiplication by the {\em Thom class}
$t(\eta)\in H_{C_2}^\eta(\susp_X^\eta X_+)$, which is the element corresponding to $1\in H_{C_2}^0(X_+)$.
The {\em Euler class} of $\eta$ is then the element $e(\eta) \in H_{C_2}^\eta(X_+)$ given by
restricting the Thom class to the zero section.
(See \cite[\S 3.11.1]{CostenobleWanerBook} for more details.)

If $X$ is a space over $B$ and the representation $\eta^*\in RO(\Pi X)$ is the image of an element
$\eta\in RO(\Pi B)$, the change of base isomorphism tells us that we have a similar
isomorphism
\[
 \Mackey H_{C_2}^\gamma(\susp_B^\eta X_+) \iso \Mackey H_{C_2}^{\gamma-\eta}(X_+)
\]
when grading on $RO(\Pi B)$.

If $M$ is a smooth closed $C_2$-manifold thought of as a space over itself, and $\tau$ is its tangent bundle, 
we have Poincar\'e duality:
There is a class $[M]\in  H^{C_2}_{\tau}(M_+)$ such that
\[
 -\cap [M] \colon \Mackey H_{C_2}^\gamma(M_+) \to \Mackey H^{C_2}_{\tau-\gamma}(M_+)
\]
is an isomorphism for all $\gamma\in RO(\Pi M)$. This is a direct consequence of the Thom isomorphism and
the calculation of the stable dual of $M_+$; the details are given in \cite[\S 3.11.2]{CostenobleWanerBook}.
If $M$ is a manifold over $B$ and $\tau^*\in RO(\Pi M)$ is the image of a representation
$\tau\in RO(\Pi B)$, we get a similar isomorphism when grading on $RO(\Pi B)$.

The manifolds $\Xpq{p}{q}:=\PP(\C^p\dirsum\C^{q\sigma})$ we are interested in here
are naturally manifolds over $B = \PP(\C^\infty\dirsum\C^{\infty\sigma})$
by inclusion. Moreover, their tangent representations do come from $RO(\Pi B)$ because
the restriction map 
$RO(\Pi B) \to RO(\Pi \Xpq{p}{q})$
is an epimorphism for all $p$ and $q$, and an isomorphism if both $p$ and $q$ are positive.

Poincar\'e duality allows us define to push-forward maps as follows. 

\begin{definition}\label{def:pushforward}
Suppose that $f\colon M\to N$ is a map of smooth $C_2$-manifolds with tangent bundles
$\tau_M$ and $\tau_N$, respectively.
The {\em push-forward} map
\[
 f_!\colon \Mackey H_{C_2}^{\gamma+\tau_M}(M_+) \to \Mackey H_{C_2}^{\gamma+\tau_N}(N_+).
\]
for $\gamma\in RO(\Pi N)$,
is defined by the following commutative diagram, in which the vertical maps are Poincar\'e duality
isomorphisms.
\[
 \xymatrix{
	\Mackey H_{C_2}^{\gamma+\tau_M}(M_+) \ar@{.>}[r]^-{f_!} \ar[d]_\iso
		& \Mackey H_{C_2}^{\gamma+\tau_N}(N_+) \ar[d]^\iso \\
	\Mackey H^{C_2}_{-\gamma}(M_+) \ar[r]_{f_*}
		& \Mackey H^{C_2}_{- \gamma}(N_+)
 }
\]
\end{definition}

\begin{scholium}
The push-forward is only defined for $\gamma\in RO(\Pi N)$.
We cannot get around requiring this condition---that is, 
we cannot define the push-forward on all gradings in $RO(\Pi M)$.
\end{scholium}

We will use the push-forward in the following form.
If $M$ and $N$ are manifolds over $B$, $f$ is a map over $B$,
and $\eta\in RO(\Pi B)$ is a representation such that $\eta = \tau_N-\tau_M$ in $RO(\Pi M)$,
then we get a push-forward map
\[
    f_!\colon \Mackey H_{C_2}^{\gamma}(M_+) \to \Mackey H_{C_2}^{\gamma+\eta}(N_+)
\]
for $\gamma\in RO(\Pi B)$.
To see this, write $\gamma = (\gamma+\eta-\tau_N) + \tau_M$ in $RO(\Pi M)$ and 
apply the push-forward as defined above.
In this case, the push-forward is a map of $\Mackey H_{C_2}^{RO(\Pi B)}(B_+)$-modules.

\begin{remark}\label{rem:ambiguity}
This form of the push-forward depends on the choice of $\eta$ which, in general, is determined
only up to addition of an element in the kernel of the restriction map $RO(\Pi B)\to RO(\Pi M)$.
In many, but not all of the cases where we shall use the push-forward, that restriction map is an isomorphism,
in which case $\eta$ and the push-forward are uniquely determined.
\end{remark}

In the case of a smooth inclusion $i\colon M\to N$ over $B$ we can describe the push-forward
in another way: Let $\nu$ be the normal bundle to $M$ in $N$ and
let $U\subset N$ be a normal neighborhood, 
a closed neighborhood of $M$ homeomorphic to the unit disc bundle of $\nu$,
with $\bndry U$ homeomorphic to the unit sphere bundle.
Let $U/_B \bndry U \to B$ denote the quotient ex-space over $B$ and let
$c\colon N_+\to U/_B\bndry U$ be the collapse map over $B$. 
Then $i_!$ can also be described as the composite
\[
 \Mackey H_{C_2}^{\gamma+\tau_M}(M_+) \xrightarrow{\iso}
   \Mackey H_{C_2}^{\gamma+\tau_N}(U/_B \bndry U) \xrightarrow{c^*}
   \Mackey H_{C_2}^{\gamma+\tau_N}(N_+)
\]
where the first map is the Thom isomorphism.

In a special case we can say more, generalizing a nonequivariant result.

\begin{proposition}\label{Prop fund}
Suppose that $i\colon M\to N$ is a smooth inclusion of manifolds over $B$.
Suppose also that the normal bundle $\nu$ to $M$ in $N$ extends to a bundle
$\eta$ on all of $N$ and that there is a $C_2$-equivariant section $s\colon N\to \eta$ such that
$s^{-1}(0) = M$. 
Then we have
\[
 i_! i^*(x) = e(\eta)x
\]
for all $x\in \Mackey H_{C_2}^{RO(\Pi B)}(N_+)$.
\end{proposition}

\begin{proof}
Let $U$ be a normal neighborhood of $M$ in $N$, as in the discussion above.
Consider the composite
\[
 N_+ \xrightarrow{c} U/_B\bndry U \hmtpc \susp_B^\nu M_+ \xrightarrow{i} \susp_B^\eta N_+.
\]
Given the section $s$, this composite is homotopic to the zero section of $\eta$.
The proposition follows from the fact that the Thom isomorphism followed by restriction
to the zero section is the same as multiplication by the Euler class. 

Note that we are using the following compatibility between the cup product and the $\Mackey H_{C_2}^{RO(\Pi B)}(N_+)$-module structure induced by $i^*$ on $\Mackey H_{C_2}^{RO(\Pi B)}(M_+)$, sometimes known as projection formula: 
$$i_!(y\cup i^*(x)) = i_!(y)\cup x.$$
In this particular case we take $y=1$.
\end{proof}

We shall use the following notation extensively in our calculations.

\begin{definition}\label{def:fundcohomologyclass}
Let $i:M\rightarrow N$ be a map of manifolds over $B$ together with a representation $\eta\in RO(\Pi B)$ such that $\eta = \tau_N - \tau_M$ in $RO(\Pi M)$. 
We set
$$ [M]_\eta^*:=i_!(1_M)\in \Mackey H_{C_2}^\eta(N_+),$$
which we refer to as the \textit{fundamental cohomology class} associated to $\eta\in RO(\Pi B)$.
\end{definition}
\begin{remark}
If we think of grading on $RO(\Pi N)$, then,
by definition, the class $[M]_\eta^*$ is Poincar\'e dual to the fundamental homology class $i_*([M]) \in \Mackey H^{C_2}_{\tau_N-\eta}(N_+)$. 
As in Remark~\ref{rem:ambiguity}, the push-forward, hence the element $[M]_\eta^*$, depends on the choice of $\eta$,
which is determined only up to an element in the kernel of the restriction $RO(\Pi B)\to RO(\Pi M)$. Even though the notation does not reflect it, it should be noted that the fundamental cohomology class also depends on the map $i$.
\end{remark}

\subsection{Restriction maps}\label{subsec:restrictions}

There are two interesting maps taking $C_2$-cohomology to nonequivariant
cohomology: restriction to the trivial group and the fixed-point functor.

Because we are viewing cohomology as Mackey functor--valued, our calculations have built in the
calculations of the restriction map from $C_2$-cohomology with coefficients in $\Mackey A$
to nonequivariant cohomology with coefficients in $\Z$. The restriction map is given by
\[
 \rho^*\colon H_{C_2}^\gamma(X;\Mackey A)
	= \Mackey H_{C_2}^\gamma(X;\Mackey A)(C_2/C_2)
 	\to \Mackey H_{C_2}^\gamma(X;\Mackey A)(C_2/e)
	= H^{|\gamma|}(X;\Z).
\]
(Here, as introduced earlier, we write $|\gamma|$ for the nonequivariant dimension of $\gamma$.)
On the represented level, this is given by forgetting the action of $C_2$.
In the case we are interested in, where $B = \PP(\C^{\infty+\infty\sigma})$, we
use that $H\Mackey A^{n\omega}$ is nonequivariantly equivalent to 
$H(\Mackey A(C_2/e))^{n\omega} = H\Z^{n\omega} \hmtpc \susp^n H\Z$
using the nonequivariant Thom isomorphism and the canonical orientation of the complex bundle $\omega$.
(Alternatively, use that $H\Z^{n\omega}$ represents twisted cohomology, but our base space
is simply connected so the twisting is trivial.)

The fixed-point functor is a map
\[
 (-)^{C_2}\colon H_{C_2}^\gamma(X;\Mackey A) \to H^{\gamma^{C_2}}(X^{C_2};\Z).
\]
Note carefully that it takes fixed points of $\gamma$ as well as $X$, where we use the notation $\gamma^{C_2}$ introduced earlier
and by the notation on the right we mean
\[
 H^{\gamma^{C_2}}(X^{C_2};\Z) = \Dirsum_k H^{|\gamma(b_k)^{C_2}|}(X^{C_2}_k;\Z),
\]
where $X^{C_2}_k$ is the part of $X^{C_2}$ over the $k$th component of $B^{C_2}$, $k = 0$ or $1$.
If $\gamma^{C_2}$ has different dimensions over the different components of $B^{C_2}$,
then the target will consist of non-homogeneous elements.
This functor may be defined in several ways.
In \cite[3.9.8]{CostenobleWanerBook} it is defined as the composite
\[
 H_{C_2}^\gamma(X;\Mackey A) \to H_{C_2}^\gamma(X;\conc\Z) \iso H^{\gamma^{C_2}}(X^{C_2};\Z).
\]
Here, we have used a computation that, in the notation of \cite{CostenobleWanerBook},
$\Inf_{C_2/C_2}^{C_2}\Mackey A^{C_2} = \Inf_{C_2/C_2}^{C_2}\Z = \conc\Z$.
The isomorphism above is a special case of a very general one, but can be seen very simply in this instance
by noting that both sides are $RO(\Pi B)$-graded equivariant cohomology theories 
on spaces over $B$ and agree on spheres
over $B$. This isomorphism also appears in \cite[21.1]{Co:BGU1preprint}.

On the represented level, the fixed-point functor is induced
by the geometric fixed-point functor $\Phi^{C_2}$:
\begin{align*}
 (-)^{C_2}\colon  [\susp_B^\infty X, H\Mackey A^\gamma]^{C_2}_B
  &\to [\Phi^{C_2}(\susp_B^\infty X), \Phi^{C_2}(H\Mackey A^\gamma)]_B \\
  &\to [\susp_B^\infty X^{C_2}, \Phi^{C_2}(H\conc\Z^\gamma)]_B \\
  &\iso [\susp_B^\infty X^{C_2}, H\Z^{\gamma^{C_2}}]_B.
\end{align*}
Here, we are using the equivalence $\Phi^{C_2}(H\conc\Z^\gamma) \hmtpc H\Z^{\gamma^{C_2}}$, 
proved in \cite[3.9.20]{CostenobleWanerBook}.

\begin{remark}\label{rem:restrictions}
Both restriction maps enjoy useful formal properties, which we will need in our computations. 
By \cite[3.10.10]{CostenobleWanerBook} they respect cup products;
although not shown in \cite{CostenobleWanerBook}, it is straightforward to show
from the definitions there that they also respect cap products.
In view of Propositions~3.11.2 and 3.11.5 of \cite{CostenobleWanerBook},
they preserve Thom classes and fundamental homology classes. 
This implies that they respect the Thom isomorphism and Poincar\'e duality, that Euler classes are taken to Euler classes, and that fundamental cohomology classes are taken to fundamental cohomology classes.
\end{remark}

\subsection{The representing spectra as module spectra}

In \cite{CostenobleWanerBook} it was shown that the representing spectra
$H\Mackey A^\gamma$ are module spectra over the nonparametrized spectrum $H\Mackey A$.
We insert here a result we need that was not mentioned in that book and holds for any 
compact Lie group $G$.

\begin{proposition}
If $\gamma$ and $\delta$ are two representations in $RO(\Pi B)$, then
\[
 H\Mackey A^{\gamma}\smsh_{H\Mackey A} H\Mackey A^{\delta} \hmtpc H\Mackey A^{\gamma+\delta}.
\]
\end{proposition}

\begin{proof}
Equivalences of spectra over $B$ are the same as fiberwise equivalences, so it suffices
to check on the fibers over a point $b\in B$. But that is just the calculation
\[
 \susp^V H\Mackey A \smsh_{H\Mackey A} \susp^W H\Mackey A \hmtpc \susp^{V+W}H\Mackey A
\]
for virtual representations $V$ and $W$ of the isotropy subgroup of $b$.
\end{proof}


\section{Statement of the main theorem and proof of the base case}
\label{sec: Statement}
The goal of this section is to provide a precise formulation of the main theorem and prove it in the special cases $\Xp{p}$ and $\Xq{q}$, which will act as the basis of the induction in the proof of the general case. As it constitutes the model for cohomology of the finite cases, we begin by recalling the structure of the cohomology of the base space $B:=\PP(\C^\infty\dirsum\C^{\infty\sigma})$, which was obtained in \cite{Co:BGU1preprint}, as well as some notation.

Recall that we let $\omega$ denote the tautological complex line bundle over $B$ as well as its restrictions to each $\Xpq{p}{q}$.
As pointed out in section \ref{sub ROPi}, every equivariant bundle over $B$ gives an associated representation of $\Pi B$ and we will also write $\omega$ for the associated element of $RO(\Pi B)$. As we have already seen, the representation ring $RO(\Pi B)$ is isomorphic to $\Z^3$, 
generated by $1 = \R$, $\sigma = \R^\sigma$, and $\omega$.

Of particular importance for the presentation of cohomology is the $C_2$-involution $\chi\colon B\to B$ which classifies the operation of tensoring with the sign representation: if $f\colon X\to B$ classifies a complex line bundle $\eta$, then $\chi f$ will classify $\eta\tensor\C^{\sigma}$, which we also denote by $\chi\eta$. We write $\cw = e(\omega)\in \Mackey H_{C_2}^\omega(B_+)$ for the Euler class of $\omega$ and $c_{\chiw} = e(\chiw)\in \Mackey H_{C_2}^{\chiw}(B_+)$ for the Euler class of $\chiw$.
Finally we will denote the dual bundle by $\omega\dual = \Hom(\omega,\C)$, and observe that $(\chiw)\dual = \chi(\omega\dual) = \Hom(\omega,\C^\sigma)$, so we can simply write $\chiw\dual$. Their respective Euler classes will be denoted $\cwd = e(\omega\dual)$ and $\cxwd = e(\chiw\dual)$.
For later use, note also that $\chi$ restricts to a $C_2$-homeomorphism $\chi\colon \Xpq pq \homeo \Xpq qp$ for each $p$ and $q$, as $C_2$-spaces but not as spaces over $B$.

%

The following is the main result of \cite{Co:BGU1preprint}.

\begin{theorem}
\label{thm: infinite case}
$\Mackey H_{C_2}^{RO(\Pi B)}(B_+)$ is an algebra over $\Mackey H_{C_2}^{RO(C_2)}(S^0)$ generated by the Euler classes $c_{\omega}$ and $c_{\chiw}$ together with classes $\cwt$ and $\cxwt$. 
These elements live in gradings
\begin{align*}
 \deg\cw &= \omega &  \deg\cxw &= \chiw = -\omega + 2 + 2\sigma \\
 \deg\cwt &= \omega - 2 & \deg\cxwt &= \chiw-2 = -\omega + 2\sigma.
\end{align*}
They satisfy the following relations:
\begin{align*}
	\cwt \cxw &= (1-\kappa)\cxwt \cw + e^2 \qquad\text{and} \\
	\cxwt \cwt &= \xi,
\end{align*}
and these relations completely determine the algebra.
Writing
\[
 \epsilon = e^{-2}\kappa \cxwt \cw \in \Mackey H_{C_2}^0(B_+),
\]
the units in $\Mackey H_{C_2}^0(B_+)$, which all square to $1$, are
\[
 \pm 1,\ \pm(1-\kappa),\ \pm(1-\epsilon), \text{ and } \pm(1-\kappa)(1-\epsilon) = \pm(1-\kappa+\epsilon).
\]
The Euler classes of the dual bundles are given by 
\begin{align*}
	\cwd &= -(1-\epsilon)c_{\omega} \\
\intertext{and}
	\cxwd &= -(1-\kappa)(1-\epsilon)c_{\chiw}.
\end{align*}
\qed
\end{theorem}

Here, $\kappa,e$ and $\xi$ are elements of $\Mackey H^{RO(C_2)}_{C_2}(S^0)$ mentioned in \S\,\ref{subsec:cohompoint}.

The classes $\cwt$ and $\cxwt$ are characterized by the fact that they map to Thom classes of $\omega-2$ and $\chi\omega-2$, respectively, in
the cohomology of $B\times EC_2$, where $\omega = 2 = \chi\omega \in RO(\Pi (B\times EC_2))$.
They actually come from the cohomology of $B\Pi B$, a nonequivariantly contractible $C_2$-space whose fixed set consists of
two contractible components.
The restrictions of $\zeta_1$ and $\zeta_0$ to nonequivariant cohomology are invertible,
while the fixed points are given by $\cwt^{C_2} = (1,0)$ and $\cxwt^{C_2} = (0,1)$; see Lemma~\ref{lem:restrictions}.

One of the key differences in dealing with the finite case is that the elements $\cw$, $\cxw$, $\cwt$ and $\cxwt$, which are inherited from $B$ via the pullback along the obvious inclusion, are not sufficient to generate the cohomology
of $\Xpq pq$. 
We can illustrate this with the special cases of $\Xp{p}$ and $\Xq{q}$.
These spaces are both 
copies of complex projective spaces with trivial $C_2$-action and map to $B^{C_2}$. However they do differ as $C_2$-spaces over $B$, as they map to different components of the fixed points.

\begin{theorem}
\label{thm:fixedcasesMultiplicative}
\begin{enumerate}\item[]
\item
For $0<p<\infty$ we have 
\[
 \Mackey H_{C_2}^{RO(\Pi B)}(\Xp{p}_+) \iso 
   \Mackey H_{C_2}^{RO(C_2)}(S^0)[c_{\omega}, \cwt, \cwt^{-1}]/\langle c_{\omega}^p \rangle.
\]

\vspace{0.2 cm}

\item
For  $0<q<\infty$ we have
\[
 \Mackey H_{C_2}^{RO(\Pi B)}(\Xq{q}_+) \iso 
   \Mackey H_{C_2}^{RO(C_2)}(S^0)[c_{\chiw}, \cxwt, \cxwt^{-1}]/\langle c_{\chiw}^q \rangle.
\]
\end{enumerate}
\end{theorem}

\begin{proof}
As the proofs of the two parts are essentially identical, we will consider only (1). 
Recall that the nonequivariant cohomology of $\Xp{p}$ with $\Z$ coefficients is
\[
 H^{\Z}(\Xp{p}_+) = \Z[c]/\langle c^p \rangle,
\]
where $c$ is the (nonequivariant) Euler class of $\omega$, so $\deg c = 2$.
Because the nonequivariant cohomology is free over $\Z$, \cite[Proposition 6.2]{Co:BGU1preprint} implies that
\[
 \Mackey H_{C_2}^{RO(C_2)}(\Xp{p}_+) \iso H^{\Z}(\Xp{p}_+) \tensor \Mackey H_{C_2}^{RO(C_2)}(S^0)
   \iso \Mackey H_{C_2}^{RO(C_2)}(S^0)[c]/\langle c^p \rangle.
\]
Because $RO(\Pi \Xp{p}) = RO(C_2)$, the above computation extends to $RO(\Pi B)$-grading by the adjunction of an invertible element in grading $ \omega-2$, which generates the kernel of the map $RO(\Pi B) \to RO(\Pi \Xp{p})$. 
The element $\cwt$, when restricted to $\Xp p$, is such an element (as shown in \cite{Co:BGU1preprint}), so we get
\[
 \Mackey H_{C_2}^{RO(\Pi B)}(\Xp{p}_+) 
   \iso \Mackey H_{C_2}^{RO(C_2)}(S^0)[c,\cwt,\cwt^{-1}]/\langle c^p \rangle.
\]
From \cite{Co:BGU1preprint} we also know that
$c_{\omega} = \cwt c$, so we may replace $c$ with $c_{\omega}$ as a polynomial
generator, with $c_{\omega}^p = 0$ still.
\end{proof}

Although they do not appear in the presentation,  the elements $c_{\chiw}$ and $\cxwt$ do exist in the cohomology of $\Xp{p}$ and the elements  $c_{\omega}$ and $\cwt$ exist in the cohomology of $\Xq{q}$. 
We identify these elements in the following result.

\begin{proposition}
\label{prop:fixedcaseschi}
In $\Mackey H_{C_2}^{RO(\Pi B)}(\Xp{p}_+)$, we have
\begin{align*}
 c_{\chiw} &= e^2\cwt^{-1} + \xi \cwt^{-2} c_{\omega} \qquad\text{and} \\
 \cxwt &= \xi \cwt^{-1}.
\end{align*}
In $\Mackey H_{C_2}^{RO(\Pi B)}(\Xq{q}_+)$, we have
\begin{align*}
 c_{\omega} &= e^2\cxwt^{-1} + \xi \cxwt^{-2}c_{\chiw} \qquad\text{and} \\
 \cwt &= \xi \cxwt^{-1}.
\end{align*}
\end{proposition}

\begin{proof}
These follow on restricting the relations that hold in
$\Mackey H_{C_2}^{RO(\Pi B)}(B_+)$.
The identifications of $\cxwt$ and $\cwt$ are straightforward.
For the identification of $c_{\chiw}$ in the cohomology of $\Xp{p}$, we start with
\[
 \cwt c_{\chiw} = e^2 + (1-\kappa) \cxwt c_{\omega}.
\]
From this we get
\begin{align*}
 c_{\chiw} &= e^2\cwt^{-1} + \cwt^{-1}(1-\kappa)\xi \cwt^{-1} c_{\omega} \\
  &= e^2\cwt^{-1} + \xi \cwt^{-2} c_{\omega},
\end{align*}
where we use that $\kappa\xi = 0$, so $(1-\kappa)\xi = \xi$.
The identification of $c_{\omega}$ in the cohomology of $\Xq{q}$ is similar.
\end{proof}

The invertibility of the elements $\zeta_0$ and $\zeta_1$ in Theorem~\ref{thm:fixedcasesMultiplicative}
will, on pushing forward, imply the the divisibility of fundamental classes of the form $[\Xp{p'}]^*$ and $[\Xq{q'}]^*$
in the cohomology of $\Xpq{p}{q}$.
To explain that, we now take a closer look at these fundamental classes.
For future use, we begin with a more general case.
Though we have not yet stated the full structure of the cohomology groups in question, we do know
that they contain the elements $\cw$, $\cxw$, $\cwt$, and $\cxwt$ (inherited from the infinite case), which is all that we use.

\begin{lemma}
\label{lem:pushforwardnormal}
Consider the embedding $i\colon \Xpq{p'}{q'} \to \Xpq{p}{q}$ with $0<p' \leq p$ and $0<q' \leq q$.
Then the normal bundle to $i$ is $\nu=(p-p')\omega\dual \dirsum (q-q')\chiw\dual$ and we have
\begin{align*}
 \bigl[\Xpq{p'}{q'}\bigr]^*_\nu &=(1-\kappa)^{q-q'} [-(1-\epsilon)]^{p-p' + q-q'} \cw^{p-p'} \cxw^{q-q'} \\
 &=\cwd^{\;p-p'} \cxwd^{\;q-q'}\in \Mackey H_{C_2}^\nu(\Xpq{p}{q}_+).
\end{align*}
\end{lemma}

\begin{proof}
This is a straightforward application of Proposition \ref{Prop fund}. Specifically, the identification of the normal bundle is a standard observation nonequivariantly, which extends to the equivariant case:
The normal bundle of $i$ is the restriction of
\[
 \nu = \Hom(\omega, \Cpq{p-p'}{(q-q')}),
\]
thinking of $\Cpq{p-p'}{(q-q')}$ as the summands added in passing
from $\Xpq{p'}{q'}$ to $\Xpq{p}{q}$. We obtain the necessary section by considering the composition of the inclusion of the total space of $\omega$ into the trivial bundle $\Xpq{p}{q}\times (\Cpq{p}{q})$ with the projection to the quotient $\Cpq{p}{q}/\Cpq{p'}{q'}\cong\Cpq{p-p'}{(q-q')}$.

The computation of the Euler class of
$$\nu\cong \omega^\vee\otimes \Cpq{p-p'}{(q-q')}\cong (\omega^\vee)^{\oplus (p-p')}\oplus(\chiw^\vee)^{\oplus (q-q')}$$
follows from the multiplicativity of the Euler class on direct sums and the formulas given in Theorem \ref{thm: infinite case}.  
\end{proof}

We now look at projective spaces arising from a single representation, where there is a nontrivial choice involved in determining the grading.

\begin{lemma}
\label{lem:pushforwardfixed}
\begin{enumerate}\item[]
\item
Let $0<p'\leq p$. Consider the embeddings $i\colon \Xp{p'} \to \Xpq{p}{q}$ and 
$j\colon \Xp{p'} \to B$.
Fix an $\eta\in \ker j^*$, where $j^*$ is the restriction map on the representation rings of the fundamental groupoids,
so we can write $\eta = m(\omega-2)$ for some $m\in\Z$.
Then the normal bundle to $i$ is the restriction of $\nu=(p-p')\,\omega\dual+q\chi\omega\dual$ and, in the cohomology of $\Xpq{p}{q}$, we have
\begin{align*}
  \bigl[\Xp{p'}\bigr]_{\nu+\eta}^*&= [-(1-\epsilon)]^{p-p'}\cdot[-(1-\kappa)(1-\epsilon)]^q\cwt^m\cw^{p-p'}\cxw^q\\
  &=\cwt^m\cwd^{\;p-p'}\cxwd^{\;q}
\end{align*}
if $m \geq 0$
and
\begin{align*}
  \cwt^{|m|} \cdot\bigl[\Xp{p'}\bigr]_{\nu+\eta}^* &= [-(1-\epsilon)]^{p-p'}\cdot[-(1-\kappa)(1-\epsilon)]^q\cw^{p-p'}\cxw^q\\
  &=\cwd^{\;p-p'}\cxwd^{\;q}
\end{align*}
if $m < 0$.

\vspace{0.1 cm}

\item
Let $0<q'\leq q$. Consider the embeddings $i\colon \Xq{q'} \to \Xpq{p}{q}$ and 
$j\colon \Xq{q'} \to B$.
Fix an $\eta\in \ker j^*$, where $j^*$ is the restriction map on the representation rings of the fundamental groupoids,
so we can write $\eta = m(\chi\omega-2)$ for some $m\in\Z$.
Then the normal bundle to $i$ is the restriction of $\nu=p\omega\dual+(q-q')\chi\omega\dual$ and, in the cohomology of $\Xpq{p}{q}$, we have
\begin{align*}
  \bigl[\Xq{q'}\bigr]_{\nu+\eta}^*&= [-(1-\epsilon)]^{p}\cdot[-(1-\kappa)(1-\epsilon)]^{q-q'}\cxwt^m\cw^{p}\cxw^{q-q'}\\
  &=\cxwt^m\cwd^{\;p}\,\cxwd^{\;q-q'}
\end{align*}
if $m \geq 0$
and
\begin{align*}
  \cxwt^{|m|} \cdot\bigl[\Xq{q'}\bigr]_{\nu+\eta}^* &= [-(1-\epsilon)]^{p}\cdot[-(1-\kappa)(1-\epsilon)]^{q-q'}\cw^{p}\cxw^{q-q'}\\
  &=\cwd^{\;p}\,\cxwd^{\;q-q'}
\end{align*}
if $m < 0$.
\end{enumerate}
\end{lemma}

\begin{proof}
The proof is essentially the same as that of Lemma~\ref{lem:pushforwardnormal}, except for the presence of the factors of
$\cwt$ or $\cxwt$.
In part (1), for example, we are to take the push-forward $i_!(1)$. The element $1$ has grading $0$
when we grade on $RO(\Pi \Xp{p'})$, but we need to expand the grading to $RO(\Pi B)$ and take a corresponding element in grading $\eta$.
As in Theorem~\ref{thm:fixedcasesMultiplicative}, we expand the grading by introducing the invertible element $\cwt$ and the element in grading $\eta$ we take is then $\cwt^m$.
The result now follows by computing $i_!(\cwt^m) = i_!i^*(\cwt^m)$ when $m\geq 0$ or by writing $\cwt^{|m|}i_!(\cwt^m) = i_!(1)$ when $m<0$.
\end{proof}

We include one more calculation of a push-forward, which we will need later for an application.
In the following, we write $\zeta = \rho(\zeta_1) \in \Mackey H_{C_2}^{\omega-2}(\Xpq pq_+)(C_2/e)$ and we let
$c\in \Mackey H_{C_2}^2(\Xpq pq_+)(C_2/e) \iso H^2(\Xp{p+q}_+;\Z)$ be the nonequivariant first Chern class
of the dual of the tautological bundle. We then have that $\zeta$ is invertible (since $\zeta\cdot\rho(\zeta_0) = \iota^2$),
$\rho(\cwd) = \zeta c$, and $\rho(\cxwd) = \iota^{2}\zeta^{-1} c$.

\begin{lemma}
\label{lem:pushforwardfree}
Let $0 < r \leq p+q$, consider a nonequivariant embedding $\Xp r \to \Xpq pq$ induced by a nonequivariant linear inclusion
$\C^r\includesin\C^{p+q\sigma}$, and extend to a $C_2$-map
$k\colon C_2\times\Xp r \to \Xpq pq$. Fix an $\eta\in\ker k^*$ where $k^*\colon RO(\Pi\Xpq pq) \to RO(\Pi(C_2\times\Xp r)) \iso \Z$,
so $\eta = n(\omega-2) + m(2\sigma-2)$ for some $n, m\in\Z$. Then
\[
 [C_2\times\Xp r]^*_{\eta+2(p+q-r)} 
    = \tau(\iota^{m}\zeta^n c^{p+q-r})
\]
where $[C_2\times\Xp r]^*$ is defined to be $k_!(\tau(1)) = k_!([C_2])$.
\end{lemma}

\begin{proof}
If we write $k = C_2\times k'$, where $k'\colon \Xp r\to \Xpq pq$ is the nonequivariant embedding, then
$k_!(\tau(1)) = \tau(k'_!(1))$. The normal bundle to the nonequivariant embedding $k'$ is $(p+q-r)\omega\dual$,
where we use $\omega$ to also denote the nonequivariant tautological bundle, so $k'_!(1) = c^{p+q-r}$.
However, as in the preceding lemma, we need to adjust the grading when we look at the expanded grading $RO(\Pi\Xpq pq)$,
and the corresponding element in  grading $\eta+2(p+q-r)$ is $\iota^m\zeta^n c^{p+q-r}$.
The lemma follows.
\end{proof}

With these calculations in hand we can write an additive basis of the cohomology of $\Xp p$ or $\Xq q$ as follows.

\pagebreak
\begin{corollary}\label{cor:fixedcasesBases}
\begin{enumerate}\item[]
\item
Let $0<p<\infty$.
Additively, $\Mackey H_{C_2}^{RO(\Pi B)}(\Xp{p}_+)$ is a free module over $\Mackey H_{C_2}^{RO(C_2)}(S^0)$ where, for each $m\in\Z$,
 the submodule $\Mackey H_{C_2}^{m\omega+RO(C_2)}(\Xp{p}_+)$  has a basis given by the set
\[
 \Big\{ \big[\Xp{p}\big]_\eta^*, \big[\Xp{p-1}\big]_{\eta+2}^*, \ldots, \big[\Xp{}\big]_{\eta+2(p-1)}^* \Big\}
\]
with $\eta=m(\omega-2)$.

\vspace{0.1 cm}

\item
Let  $0<q<\infty$.
Additively, $\Mackey H_{C_2}^{RO(\Pi B)}(\Xq{q}_+)$ is a free module over $\Mackey H_{C_2}^{RO(C_2)}(S^0)$ where, for each $m\in\Z$,
 the submodule $\Mackey H_{C_2}^{m\omega+RO(C_2)}(\Xq{q}_+)$ has a basis given by the set
\[
 \Big\{\big[\Xq{q}\big]_\eta^*, \big[\Xq{(q-1)}\big]_{\eta+2}^*, \ldots, \big[\Xq{}\big]_{\eta+2(q-1)}^* \Big\}
\]
with  $\eta=-m(\chiw-2) = m(\omega-2\sigma)$.
\end{enumerate}
\end{corollary}

\begin{proof}
For part (1), Lemma~\ref{lem:pushforwardfixed} gives us that
\[
    \big[\Xp{p-j}\big]_{\eta+2j}^* = [-(1-\epsilon)]^j \cwt^{m-j} \cw^j,
\]
because $\eta+2j = j\omega + (m-j)(\omega-2)$.
Note that we do not need to distinguish the cases $m\geq 0$ and $m<0$ because $\cwt$ is invertible
in $\Mackey H_{C_2}^{RO(\Pi B)}(\Xp{p}_+)$.
From Theorem~\ref{thm:fixedcasesMultiplicative}, we can see that a basis for $\Mackey H_{C_2}^{m\omega+RO(C_2)}(\Xp{p}_+)$
is given by
\[
    \{ \cwt^m, \cwt^{m-1}\cw, \ldots, \cwt^{m-p+1}\cw^{p-1} \}.
\]
Part (1) follows and part (2) is similar.
\end{proof}

One consequence of Theorem~\ref{thm:fixedcasesMultiplicative} and Corollary~\ref{cor:fixedcasesBases} is that 
the fundamental cohomology classes $\big[\Xp{m}\big]_\eta^*$ and  $\big[\Xq{m}\big]_\eta^*$ are infinitely divisible
by $\zeta_1$ and $\zeta_0$, respectively.
This phenomenon propagates itself into the cohomology of $\Xpq{p}{q}$ with $p$ and $q$ both strictly positive, even though it will no longer be true that $\cwt$ and $\cxwt$ are actually invertible. 
In the general case we introduce the following elements, as suggested by Lemma~\ref{lem:pushforwardfixed}.

\begin{definition}\label{def:elements}  
Let $m$, $p$, $q\in\NN$ such that $m>0$ and $p+q>0$. We define the following elements in $\Mackey H_{C_2}^{RO(\Pi B)}(\Xpq{p}{q}_+)$:
\[
 \cwt^{-m}\cxw^q := [-(1-\kappa)(1-\epsilon)]^q\bigl[\Xp{p}\bigr]_{q\chiw - m(\omega-2)}^*
\]
and
\[
 \cxwt^{-m}\cw^p := [-(1-\epsilon)]^p\bigl[\Xq{q}\bigr]_{p\omega - m(\chiw-2)}^* \, .
\]
\end{definition}



We are finally in the position to state our main theorem. As we have been doing, we will use the same names for the restrictions to $\Xpq{p}{q}$ of $\omega$ and $\chiw$ and their Euler classes.

\begin{thm2}
\label{thm:finitespaces}
Let $0 \leq p < \infty$ and $0 \leq q < \infty$ with $p+q > 0$.
As a module over $\Mackey H_{C_2}^{RO(C_2)}(S^0)$, $\Mackey H_{C_2}^{RO(\Pi B)}(\Xpq{p}{q}_+)$ is free.
As a (graded) commutative algebra over $\Mackey H_{C_2}^{RO(C_2)}(S^0)$, 
$\Mackey H_{C_2}^{RO(\Pi B)}(\Xpq{p}{q}_+)$ is generated by $c_{\omega}$, $c_{\chiw}$, $\cwt$, and $\cxwt$,
together with the following classes:
$c_{\omega}^p$ is infinitely divisible by $\cxwt$, meaning that, for $k\geq 1$,
there are unique elements $\cxwt^{-k}c_{\omega}^p$ such that
\[
 \cxwt^k \cdot \cxwt^{-k} c_{\omega}^p = c_{\omega}^p.
\]
Similarly, $c_{\chiw}^q$ is infinitely divisible by $\cwt$, meaning that, for $k\geq 1$,
there are unique elements $\cwt^{-k}c_{\chiw}^q$ such that
\[
 \cwt^k \cdot \cwt^{-k} c_{\chiw}^q = c_{\chiw}^q.
\]
The generators satisfy the following relations:
\begin{align}
	c_{\omega}^p c_{\chiw}^q &= 0, \\
	\cwt c_{\chiw} &= (1-\kappa)\cxwt c_{\omega} + e^2 \qquad\text{and} \label{rel:tensor product relation}\\
	\cxwt \cwt &= \xi.
\end{align}
\end{thm2}

The reader should compare this to the classical computation of the cohomology of $\Xp{n}$.
Both $\cw$ and $\cxw$ are sent to $c_1(\omega)\in H^2(\C\PP^{p+q-1};\Z)$ under the forgetful map, and
 the relation $c_{\omega}^p c_{\chiw}^q = 0$ is the appropriate analog of $c_1(\omega)^{n+1}=0$ in the cohomology of $\C\PP^n$. 
On the other hand, both $\cwt$ and $\cxwt$ are sent to a unit in $H^0(\C\PP^{p+q-1};\Z)$.

For a more detailed description of the algebraic structure 
we refer the reader to Propositions \ref{prop: cancellation} and \ref{prop:relations}.


Theorem~A will be proved by induction on $p$ and $q$.
We have essentially already proved the base case, which is when one of $p$ or $q$ is 0. 

\begin{lemma}\label{lem:basis}
Theorem~A is true if $p=0$ or $q=0$.
\end{lemma}

\begin{proof}
We discuss the case $q=0$, the case $p=0$ being similar.
Theorem~\ref{thm:fixedcasesMultiplicative} showed that
\[
 \Mackey H_{C_2}^{RO(\Pi B)}(\Xp{p}_+) \iso 
   \Mackey H_{C_2}^{RO(C_2)}(S^0)[c_{\omega}, \cwt, \cwt^{-1}]/\langle c_{\omega}^p \rangle,
\]
and Proposition~\ref{prop:fixedcaseschi} showed that
\begin{align*}
 c_{\chiw} &= e^2\cwt^{-1} + \xi \cwt^{-2} c_{\omega} \qquad\text{and} \\
  \cxwt &= \xi \cwt^{-1}.
\end{align*}
The relations listed in Theorem A can now be verified.
\end{proof}


\section{Proof of the main result} 
\label{sec: Proof}
The goal of this section is the proof of the main result. 
Nonequivariantly, we can calculate the cohomology of the finite projective space $\Xp{p}$
from that of the infinite projective space $\Xp{\infty}$
by considering the long exact sequence of the pair $(\Xp{\infty}, \Xp{p})$,
which is equivalent to a Gysin sequence and which splits.
Equivariantly, we can examine the long exact sequence of the pair
$(\Xpq{\infty}{\infty}, \Xpq{p}{q})$, but the resulting Gysin sequence does not split,
making it unsuitable as a means of computing the cohomology of
$\Xpq{p}{q}$. Hence we use another approach.

We begin by showing that, in order to prove that a set of elements forms a basis for equivariant cohomology, it suffices to prove that it forms
a basis under restriction to nonequivariant cohomology and also on restriction to fixed points.
This method will then be applied to a family of elements constructed inductively via push-forwards to show that they form a basis, establishing the additive structure of the equivariant cohomology of $\Xpq{p}{q}$. 
We use this additive basis to derive the multiplicative structure. Finally, we will provide a way to visualize the additive generators of the cohomology.

    We would like to thank the referee for suggesting to us Proposition~\ref{prop:proofofbasis} and the overall strategy of the proof we employ here,
which is simpler than our original argument.

\subsection{How to recognize a basis} \label{susbsec: Lemma Basis}
We now work toward Proposition~\ref{prop:proofofbasis}, which gives us a method for showing that a set forms a basis for equivariant cohomology.
We use the restriction maps discussed in section~\ref{subsec:restrictions} and their represented equivalents $\rho^*$, the restriction to the nonequivariant case, and $\Phi^{C_2}$, the geometric fixed-point functor.

\begin{lemma}\label{lem:equivalence}
If $f\colon X\to Y$ is a map of (nonparametrized) $C_2$-spectra such that
$\rho^*f$ and $\Phi^{C_2}f$ are both nonequivariant equivalences, then $f$ is a $C_2$-equivalence.
\end{lemma}

\begin{proof}
The statement follows from \cite[V.7.5(iii)]{LMS}, but the proof is straightforward
in this special case:
Let $EC_2$ denote the universal free $C_2$-space, a free $C_2$-CW complex that is nonequivariantly
contractible. Let $\tE C_2$ be the cofiber in the sequence $(EC_2)_+ \to S^0 \to \tE C_2$.
The fact that $\rho^*f$ is a nonequivariant equivalence implies that
\[
 f\smsh (EC_2)_+\colon X\smsh (EC_2)_+ \to Y\smsh (EC_2)_+
\]
is a $C_2$-equivalence \cite[II.2.12]{LMS}.
On the other hand,
\[
 f\smsh \tE C_2 \colon X\smsh \tE C_2  \to Y\smsh \tE C_2 
\]
is a nonequivariant equivalence, because both spectra are nonequivariantly trivial, and on
$C_2$-fixed sets we have $(f\smsh \tE C_2)^{C_2} = \Phi^{C_2}f$, which is an equivalence by assumption.
Therefore, $f\smsh \tE C_2$ is also an equivariant equivalence. That $f$ is an equivalence now follows.
\end{proof}

We now want to generalize a nonparametrized, nonequivariant result which allows one to verify whether or not a collection of cohomology classes form a basis by checking whether an associated map of spectra is an equivalence.
Throughout we will use $\smsh$ to denote the parametrized fiberwise smash product of spectra over $B$.

  Consider a family of cohomology classes  $\mathcal{X}=\{ x_\alpha \in H_{C_2}^{\gamma_\alpha}(X;\Mackey A) \}$, where for each $\alpha$ one has
  $\gamma_\alpha = {\delta} + v_\alpha$ for some $v_\alpha\in RO(C_2)$ and a fixed ${\delta}\in RO(\Pi B)$.
  Each of its elements can be represented by a map $x_\alpha\colon X \to H\Mackey A^{{\delta}+v_\alpha}$, which can be extended to a map $H\Mackey A\smsh X\to H\Mackey A^{{\delta}+v_\alpha}$ of $H\Mackey A$-modules.
By further smashing this map with $H\Mackey A^{-{\delta}}$  and then applying the product map $H\Mackey A^{-{\delta}}\smsh_{H\Mackey A} H\Mackey A^{{\delta}+r_{\alpha}} \to H\Mackey A^{r_{\alpha}}$,
one obtains $H\Mackey A^{-{\delta}}\smsh X\to H\Mackey A^{v_\alpha}$. Here $H \Mackey A^{-{\delta}}$ is the Eilenberg-Mac Lane spectrum representing cohomology in gradings $-{\delta} + RO(C_2)$ (see \cite[Section 3.7]{CostenobleWanerBook}) and 
\[
 H\Mackey A^{v_\alpha} = r^*H\Mackey A^{v_\alpha} = r^*(\susp^{v_\alpha}H\Mackey A)
\]
is the pullback of the nonparametrized spectrum representing cohomology shifted by $v_\alpha$, where $r$ is the projection of $B$ to a point.
In view of the $r_!$-$r^*$ adjunction (see \cite[Section 2.4 and Proposition 2.8.2]{CostenobleWanerBook}), we obtain 
\[
 \hat x_{\alpha}\colon r_!(H\Mackey A^{-{\delta}}\smsh X) \to \susp^{v_\alpha}H\Mackey A,
\]
a map of nonparametrized spectra.
Finally, the maps $\hat x_{\alpha}$ can be assembled together by taking the wedge sum of the targets and we will denote by $\hat{\mathcal{X}}$ the resulting map
\begin{eqnarray}
\label{eqn family}
   \hat{\mathcal{X}} = (\hat x_{\alpha})\colon r_!(H\Mackey A^{-{\delta}}\smsh X) 
  \to \Wedge_{\alpha} \susp^{v_\alpha}H\Mackey A.  
\end{eqnarray}
Notice that this same construction can be applied to any group and in the special case of the trivial group $\hat{\mathcal{X}}$ becomes 
\[
 (\hat x_\alpha)\colon r_!(H\Z^{-{\delta}}\smsh X) \to \Wedge_\alpha \susp^{n_\alpha}H\Z
\]
with $n_\alpha\in \ZZ$.



We need the following result on how to recognize a nonequivariant basis.

\begin{lemma}\label{lem:nonequivariantbasis}
Let $X$ be a nonequivariant ex-space over $B$ of finite type, let ${\delta}\in RO(\Pi B)$ and let
\[
 \mathcal{X}:=\big\{ x_\alpha \in H^{{\delta}+n_\alpha}(X;\Z) \big\}
\]
be a collection of cohomology elements in $H^{{\delta}+\Z}(X;\Z)$.
Consider the associated map
\[
 \hat{\mathcal{X}} = (\hat x_\alpha)\colon r_!(H\Z^{-{\delta}}\smsh X) \to \Wedge_\alpha \susp^{n_\alpha}H\Z.
\]
Then $\{ x_\alpha \}$
is a basis for $H^{{\delta}+\Z}(X;\Z)$ over $\Z$ if and only if $\hat{\mathcal{X}}$ is an equivalence.
\end{lemma}

Note that the use of $\delta$ here indicates that we are considering cohomology with coefficients twisted by $\delta$.

\begin{proof}
Suppose first that $\hat{\mathcal{X}}$ is an equivalence. Then we have the following for any $n$,
where $[-,-]_{H\Z}$ denotes homotopy classes of maps of $H\Z$-modules:
\begin{align*}
 [X, H\Z^{{\delta}+n}]^B
  &\iso [H\Z\smsh X, H\Z^{{\delta}+n}]^B_{H\Z} \\
  &\iso [H\Z^{-{\delta}}\smsh X, r^*(\susp^n H\Z)]^B_{H\Z} \\
  &\iso [r_!(H\Z^{-{\delta}}\smsh X), \susp^n H\Z]_{H\Z} \\
  &\iso \bigl[\Wedge_\alpha \susp^{n_\alpha}H\Z, \susp^n H\Z \bigr]_{H\Z} \quad\text{via $\hat{\mathcal X}$} \\
  &\iso \Dirsum_{\alpha} [\susp^{n_\alpha}H\Z, \susp^n H\Z]_{H\Z} \\
  &\iso \Dirsum_{\alpha} [S^{n_\alpha}, \susp^n H\Z] \\
  &\iso \Dirsum_{\alpha} H^{n-n_\alpha}(S^0;\Z)
\end{align*}
This shows that $H^{{\delta}+\Z}(X;\Z)$ is a free $H^{\Z}(S^0;\Z) = \Z$-module with basis given
by the $\{x_\alpha\}$.

Conversely, suppose that $\{x_\alpha\}$ is a basis.
Then the homology $H_{{\delta}+\Z}(X;\Z)$ is also free with a dual basis $\{y_\alpha\}$.
Represent $y_\alpha$ by a map $S^{n_\alpha} \to r_!(H\Z^{-{\delta}}\smsh X)$.
The composite
\[
 S^{n_\beta} \xrightarrow{y_\beta} r_!(H\Z^{-{\delta}}\smsh X)
  \xrightarrow{\hat x_{\alpha}} \susp^{n_{\alpha}}H\Z
\]
represents $\langle x_\alpha, y_\beta \rangle = \delta_{\alpha,\beta}$.
The fact that $\{y_\alpha\}$ is a basis now implies that
\[
 (\hat x_{\alpha})_*\colon \pi_\Z r_!(H\Z^{-{\delta}}\smsh X) 
  \to \pi_\Z\bigl(\Wedge_\alpha \susp^{n_\alpha}H\Z\bigr)
\]
is an isomorphism, hence that $(\hat x_\alpha)$ is an equivalence.
\end{proof}

We can now prove our criterion for recognizing an equivariant basis.

\begin{proposition}\label{prop:proofofbasis}
Let $X$ be a $C_2$-ex-space over $B$ of finite type and let
\[
 \mathcal{X}:= \big\{ x_\alpha \in H_{C_2}^{\gamma_\alpha}(X) \big\}
\]
be a collection of cohomology elements.
If $ \rho^*\mathcal{X}$ is a basis 
for $H^{RO(\Pi B)}(X;\Z)$
and $\mathcal{X}^{C_2}$ is a basis for $H^{RO(\Pi B)}(X^{C_2};\Z)$
(both as $\Z$-modules),
then $\{ x_\alpha \}$ is a basis for $H_{C_2}^{RO(\Pi B)}(X)$ as a module over
$H_{C_2}^{RO(C_2)}(S^0)$.
\end{proposition}

\begin{proof}
Because we are considering $H_{C_2}^{RO(\Pi B)}(X)$ as a module over the $RO(C_2)$-graded ring $H_{C_2}^{RO(C_2)}(S^0)$,
it suffices to restrict to parts of the cohomology of the form $H_{C_2}^{\delta+RO(C_2)}(X)$ for a fixed $\delta\in RO(\Pi B)$.
Hence we will restrict to just those indices $\alpha$ such that the corresponding grading $\gamma_\alpha$ is of the form
$\gamma_\alpha = {\delta} + v_\alpha$ with $v_\alpha\in RO(C_2)$. This restriction will be in place for the
remainder of the proof.

We first wish to show that 
\[
 \hat{\mathcal{X}}\colon r_!(H\Mackey A^{-{\delta}}\smsh X) \to \Wedge_{\alpha} \susp^{v_\alpha}H\Mackey A
\]
is an equivalence. By Lemma~\ref{lem:equivalence} it suffices
to prove that it is an equivalence nonequivariantly and on applying $\Phi^{C_2}$.
The first is easy: That it is a nonequivariant equivalence follows directly from Lemma~\ref{lem:nonequivariantbasis}
and our assumption that ${\rho^*x_\alpha}$ is a basis.

To show that $\Phi^{C_2}\hat{\mathcal X}$ is an equivalence, on applying $\Phi^{C_2}$ we get the map
\[
 \Phi^{C_2}(\hat x_\alpha)\colon 
   r_!(\Phi^{C_2}(H\Mackey A^{-{\delta}})\smsh X^{C_2}) \to
   \Wedge_\alpha \susp^{v_\alpha^{C_2}} \Phi^{C_2}(H\Mackey A)
\]
of $\Phi^{C_2}(H\Mackey A)$-modules.
$\Phi^{C_2}(H\Mackey A)$ is an algebra over $H\Z$: The map $H\Mackey A\to H\Mackey A\smsh \tE C_2$ makes the latter an algebra
over $H\Mackey A$, $(-)^{C_2}$ is lax monoidal, and $(H\Mackey A)^{C_2} \hmtpc H\ZZ$.
Moreover, $\Phi^{C_2}(H\Mackey A)$ has an augmentation
\[
 \Phi^{C_2}(H\Mackey A) \to \Phi^{C_2}(H\conc\Z) \hmtpc H\Z.
\]
Applying the adjunction between maps of $\Phi^{C_2}(H\Mackey A)$-modules and $H\Z$-modules
and applying the augmentation gives us
\[
 (\hat x_\alpha^{C_2})\colon 
   r_!(H\Z^{-{\delta}^{C_2}}\smsh X^{C_2}) \to
   \Wedge_\alpha \susp^{v_\alpha^{C_2}} H\Z
\]
which is an equivalence of $H\Z$-modules, again by Lemma~\ref{lem:nonequivariantbasis} 
and the assumption that $\{ x_\alpha^{C_2} \}$ is a basis.
We can recover $\Phi^{C_2}(\hat x_\alpha)$ by applying $\Phi^{C_2}(H\Mackey A)\smsh_{H\Z}-$,
showing that $\Phi^{C_2}(\hat x_\alpha)$ is also an equivalence.

So we now have that  $\hat{\mathcal{X}}$ is an equivalence. 
That $\{ x_\alpha \}$ is therefore a basis follows by the same argument as in the first half of the
proof of Lemma~\ref{lem:nonequivariantbasis}.
\end{proof}

\subsection{Additive structure}\label{susbsec: Additive}
We will now employ our recognition principle to determine the additive structure of the cohomology of finite projective spaces. We first define the proposed basis.

To every projective space $\Xpq{p}{q}$ we will now associate an auxiliary function 
$$F_{p,q}:\Z\rightarrow \mathcal{P}\left(\Mackey H_{C_2}^{RO(\Pi B)}\big(\Xpq{p}{q}_+\big) \right),$$
constructed inductively on $p$ and $q$,
and we shall show that $F_{p,q}(n)$ is a basis for $\Mackey H_{C_2}^{n\omega+RO(C_2)}(\Xpq{p}{q}_+)$ for each $n\in\ZZ$.

The base cases of the definition are provided by Corollary \ref{cor:fixedcasesBases}.  One sets
 $$F_{p,0}(n):=\left\{ \big[\Xp{p}\big]_\nu^*, \big[\Xp{p-1}\big]_{\nu+2}^*, \dots, \big[\Xp{}\big]_{\nu+2(p-1)}^* \right\},$$
where $\nu:=n(\omega-2)$, and       
$$F_{0,q}(n):=\left\{ \big[\Xq{q}\big]_{\nu}^*, \big[\Xq{(q-1)}\big]_{\nu+2}^*, \dots, \big[\Xq{}\big]_{\nu+2(q-1)}^* \right\},$$
where $\nu:=-n(\chiw-2)=n(\omega-2\sigma)$.
For $p$ and $q$ both strictly greater than 0 we set
$$F_{p,q}(n):=
\begin{cases}  
\big\{\cwt^n\big[\Xpq{p}{q}\big]_0^*\big\}\cup i_!F_{p-1,q}(n-1)   & \text{ for } n\geq 0\\
\big\{\cxwt^{|n|}\big[\Xpq{p}{q}\big]_0^*\big\}\cup j_!F_{p,q-1}(n+1) & \text{ for } n< 0 \ ,
\end{cases}
$$
where $i:\Xpq{p-1}{q}\rightarrow\Xpq{p}{q}$ and $j:\Xpq{p}{(q-1)}\rightarrow\Xpq{p}{q}$ are the canonical inclusions.
It now remains to check that these functions identify an additive basis for the cohomology as we have claimed.

We shall use Proposition~\ref{prop:proofofbasis}, so we now consider the restriction of $F_{p,q}(n)$ to non-equivariant cohomology and to fixed points.

\begin{lemma}\label{lem:restrictions}
In $H^*(\Xp{p+q}_+;\ZZ)$ we have
 $$\rho^*(\cwt)= 1\,, \quad
 \rho^*(\cxwt) = 1\,,\quad \text{and}\quad
 \rho^*\left(\big[\Xpq{p'}{q'}\big]_\nu^*\right) = \big[\Xp{p'+q'}\big]^*$$
for any admissible choice of $\nu$ with $0\leq p'\leq p$, $0\leq q'\leq q$ and $p'+q' > 0$.   

As elements of
\begin{multline*}
    H^\ZZ(\Xpq{p}{q}_+^{C_2};\ZZ)\iso H^\Z((\Xp{p}\disjunion \Xp{q})_+;\Z)  \\ \iso H^\Z(\Xp{p}_+;\Z)\dirsum H^\Z(\Xp{q}_+;\Z)
\end{multline*}
we have
$$\cwt^{C_2} = (1,0)\,,\quad \cxwt^{C_2} = (0, 1)\,,\quad \text{and}
\quad \Big(\big[\Xpq{p'}{q'}\big]_\nu^*\Big)^{C_2}=\Big(\big[\Xp{p'}\big]^*,\big[\Xp{q'}\big]^*\Big)
$$
for any admissible choice of $\nu$ with $0\leq p'\leq p$, $0\leq q'\leq q$ and $p'+q' > 0$. 
\end{lemma}

\begin{proof}
The statements about $\zeta_0$ and $\zeta_1$ follow from calculations in \cite{Co:BGU1preprint}.
The statements about the fundamental classes are straightforward consequences of Remark \ref{rem:restrictions}. 
As noted in the statement of the lemma, $(\Xpq pq)^{C_2} = \Xp p \disjunion \Xp q$, and it follows that $(-)^{C_2}$ maps the fundamental cohomology class of an embedded manifold $M$ to the pair
\[
 \Big(\big[M\cap \Xp{p}\big]^*,\big[M\cap \Xp{q}\big]^*\Big). \qedhere
\]
\end{proof}

\begin{proposition}
$\Mackey H_{C_2}^{RO(\Pi B)}(\Xpq{p}{q}_+)$ is a free module over
$\Mackey H_{C_2}^{RO(C_2)}(S^0)$ generated by $Im\,(F_{p,q})$.

More specifically, for every $n\in \ZZ$ the set $F_{p,q}(n)$ gives an additive basis for $\Mackey H_{C_2}^{n\omega+ RO(C_2)}(\Xpq{p}{q}_+)$ and its image under $\rho^*$ and $\Phi^{C_2}$ provides additive bases for $H^{|n\omega+RO(C_2)|}(\Xpq{p}{q}_+;\ZZ)$ and $H^{(n\omega+RO(C_2))^{C_2}}(\Xpq{p}{q}_+^{C_2};\ZZ)$, respectively.
\end{proposition}
\begin{proof}
The proof is by induction, with base cases represented by the projective spaces of the form $\Xp{p}$ and $\Xq{q}$. For these the first part of the statement is precisely Corollary~\ref{cor:fixedcasesBases}, while the second part follows easily from Lemma \ref{lem:restrictions} and the fact that
the nonequivariant cohomology $H^\Z(\Xp{n};\Z)$ is additively generated by the fundamental cohomology classes of (non-strictly) smaller projective spaces.

We now assume that $p>0$ and $q>0$ and that the result is true for every pair $p'\leq p$ and $q'\leq q$ with $p'+q' < p+q$.

In order to apply Proposition \ref{prop:proofofbasis} we need to verify that $\rho^*F_{p,q}$ and $F_{p,q}^{C_2}$ form bases 
for $H^{RO(\Pi B)}(\Xpq{p}{q};\Z)$
 and $H^{RO(\Pi B)}(\Xpq{p}{q}^{C_2};\Z)$, respectively. 
It suffices to restrict to gradings of the form $n\omega+RO(C_2)$, hence reducing the proof to the verification that the sets 
$\rho^*\big(F_{p,q}(n)\big)$ and $F_{p,q}(n)^{C_2}$ form bases for the groups $H^{n\omega+\ZZ}(\Xpq{p}{q};\ZZ)$ and $H^{(n\omega+\ZZ)^{C_2}}(\Xpq{p}{q}^{C_2};\ZZ)$, respectively.


As with the definition of $F_{p,q}$, the inductive step has to be subdivided into two cases, depending on the value of $n$. We will only consider the case $n\geq 0$, as the case $n<0$ is essentially the same.

In view of the functorial compatibility between push-forwards and restriction maps, we have the following commutative diagram.
$$\xymatrix{
H^{(n-1)\omega+RO(C_2)}_{C_2}(\Xpq{p-1}{q}_+)\ar[r]^{\phantom{a}\quad\rho^*}\ar[d]^{i_!}& H^{2(n-1)+\ZZ}(\Xp{p+q-1}_+;\ZZ) \ar[d]^{i_!}\\
H^{n\omega+RO(C_2)}_{C_2}(\Xpq{p}{q}_+)\ar[r]^{\phantom{a}\quad\rho^*}& H^{2n+\ZZ}(\Xp{p+q}_+;\ZZ)
}$$
On the left side of the diagram we have the push-forward map on equivariant cohomology, while on the right we have the corresponding map on non-equivariant cohomology.  

By definition, $F_{p,q}(n) = \{\cwt^n[\Xpq{p}{q}]^*_0\}\union i_! F_{p-1,q}(n-1)$.
The inductive hypothesis ensures that $F_{p-1,q}(n-1)$ and its image under $\rho^*$ both form a basis for the cohomologies of the upper part of the diagram. Moreover, we know that the right vertical map is a split injection which, on generators, is given by
$i_![\Xp k]^* = [\Xp k]^*$ for $1\leq k \leq p+q-1$,
and we know that its cokernel is generated by the identity element $[\Xp{p+q}]^*$.
By Lemma~\ref{lem:restrictions}, the class $\cwt^n$ restricts to the identity, so we can conclude that $\rho^*F_{p,q}(n)$ indeed provides a basis for the non-equivariant cohomology.


Next we turn to the fixed points.
We have the following commutative diagram.
$$\xymatrix@C+2em{
 H^{(n-1)\omega+RO(C_2)}_{C_2}(\Xpq{p-1}{q}_+)\ar[r]^{(-)^{C_2}}\ar[d]^{i_!}& H^{(n-1)\omega^{C_2}+\ZZ}(\Xpq{p-1}{q}_+^{C_2};\ZZ) \ar[d]^{i_!}\\
H^{n\omega+RO(C_2)}_{C_2}(\Xpq{p}{q}_+)\ar[r]^{(-)^{C_2}}& H^{n\omega^{C_2}+\ZZ}(\Xpq{p}{q}_+^{C_2};\ZZ)
}$$
Here, we have
\[
 H^{n\omega^{C_2}+\ZZ}(\Xpq{p}{q}_+^{C_2};\ZZ) = H^{2n+\ZZ}(\Xp p;\ZZ) \dirsum H^{\ZZ}(\Xp q;\ZZ)
\]
and similarly for $H^{(n-1)\omega^{C_2}+\ZZ}(\Xpq{p-1}{q}_+^{C_2};\ZZ)$.
The push-forward $i_!$ on the right of the diagram is thus the map
\begin{multline*}
 (i_!,1)\colon H^{2(n-1)+\ZZ}(\Xp {p-1};\ZZ) \dirsum H^{\ZZ}(\Xp q;\ZZ) \\
  \to H^{2n+\ZZ}(\Xp p;\ZZ) \dirsum H^{\ZZ}(\Xp q;\ZZ)
\end{multline*}
which, on the first summand, is the push-forward along $\Xp{p-1}\includesin \Xp p$ and on the second summand is the identity.
By our inductive assumption, $F_{p-1,q}(n)$ gives a basis for the top left group in the diagram that maps to the basis $F_{p-1,q}^{C_2}$
of $H^{2(n-1)+\ZZ}(\Xp {p-1};\ZZ) \dirsum H^{\ZZ}(\Xp q;\ZZ)$. We now argue as before: $i_!$ is a split injection whose cokernel is
generated by 
$([\Xp{p}]^*,0) = (1,0) = (\cwt^n)^{C_2}$ by Lemma~\ref{lem:restrictions}.
Hence $F_{p,q}(n)^{C_2}$ is a basis for $H^{n\omega^{C_2}+\ZZ}(\Xpq{p}{q}_+^{C_2};\ZZ)$.

We can now apply Proposition~\ref{prop:proofofbasis} to finish the inductive step and the proof.
\end{proof}




If we unwind the definition of $F_{p,q}(n)$ we get the following result, whose proof is left to the reader.

\begin{proposition}\label{prop:additivestructure}
If $p\geq q$, then we have the following bases for the free modules
$\Mackey H_{C_2}^{n\omega+RO(C_2)}(\Xpq{p}{q}_+)$ over $\Mackey H_{C_2}^{RO(C_2)}(S^0)$:
\begin{itemize}
\setlength\itemsep{0.2 cm}
\item If $n\geq p$, we can take as a basis the set
$$
	\left\{ \cwt^{n-l}\big[\Xpq{p-l}{q}\big]_{l\omega}^* \right\}_{0\leq l<p} 
	\union	\left\{\big[\Xq{(q-l)}\big]_{\nu+2l}^* \right\}_{0\leq l<q},
$$
where $\nu=p\,\omega-(n-p)(\chiw-2) = n\omega - 2(n-p)\sigma$.

\vspace{0.1 cm}

\item If $p-q < n < p$, we can take as a basis the set
\begin{multline*}
	\left\{ \cwt^{n-l}\big[\Xpq{p-l}{q}\big]_{l\omega}^*
	\right\}_{0\leq l < n}
		 \union \left\{\cxwt^\delta\big[\Xpq{p-(n+l+\delta)}{(q-l)}\big]_{(n+\delta)\omega+l(2+2\sigma)}^* \right\}_{\substack{\delta=0,1\\0 \leq l < p-n}} \\
		 \cup\left\{ \big[\Xq{(q-p+n-l)}\big]_{\nu+2l}^* \right\}_{0 \leq l < q+n-p}
\end{multline*}
where $\nu = p\omega + (p-n)\chiw = n\omega + (p-n)(2+2\sigma)$.

\item If $0 \leq n \leq p-q$, we can take as a basis the set
\begin{multline*}
	\left\{ \cwt^{n-l}\big[\Xpq{p-l}{q}\big]_{l\omega}^*
	\right\}_{0\leq l < n} \union
	\left\{\cxwt^\delta\big[\Xpq{p-(n+l+\delta)}{(q-l)}\big]_{(n+\delta)\omega+l(2+2\sigma)}^* \right\}_{\substack{\delta=0,1\\0 \leq l < q}}\\
	\union\left\{ \big[\Xp{p-(n+q+l)}\big]_{\nu+2l}^* \right\}_{0 \leq l <p-q-n}
\end{multline*}
where $\nu = (n+q)\omega + q\chi\omega = n\omega + q(2+2\sigma)$.

\item If $-q < n < 0$, we can take as a basis the set
\begin{multline*}
	\left\{ \cxwt^{|n|-l}\big[\Xpq{p}{(q-l)}\big]_{l\chiw}^*
	\right\}_{0\leq l < |n|} \\
	\union
    \left\{\cxwt^\delta\big[\Xpq{p-l}{(q-|n|-l-\delta)}\big]_{(n+\delta)\omega+(l+|n|)(2+2\sigma)}^* \right\}_{\substack{\delta=0,1\\0 \leq l < q-|n|}}\\
    \union\left\{ \big[\Xp{p-(q-|n|+l)}\big]_{\nu+2l}^* \right\}_{0 \leq l < p+|n|-q}
\end{multline*}
where $\nu = (q-|n|)\omega+q\chiw = n\omega + q(2+2\sigma)$.

\item If $n \leq -q$, we can take as a basis the set
$$
\left\{ \cxwt^{|n|-l}\big[\Xpq{p}{(q-l)}\big]_{l\chiw}^*
	\right\}_{0\leq l < q} \union
		\left\{\big[\Xp{p-l}\big]_{\nu+2l}^* \right\}_{0\leq l<p},
$$
where $\nu=q\chiw+(q-|n|)(\omega-2) = n\omega-2n + 2q\sigma$.
\end{itemize}

\vspace{0.3 cm}

If $p < q$ we can write down a similar collection of bases:

\vspace{0.2 cm}

\begin{itemize}\setlength\itemsep{0.2 cm}
\item If $n\geq p$, we can take as a basis the set
$$
	\left\{ \cwt^{n-l}\big[\Xpq{p-l}{q}\big]_{l\omega}^* \right\}_{0\leq l<p} 
	\union	\left\{\big[\Xq{(q-l)}\big]_{\nu+2l}^* \right\}_{0\leq l<q},
$$
where $\nu= n\omega - 2(n-p)\sigma$.

\item If $0 < n < p$, we can take as a basis the set
\begin{multline*}
	\left\{ \cwt^{n-l}\big[\Xpq{p-l}{q}\big]_{l\omega}^*
	\right\}_{0\leq l < n}
		 \union \left\{\cxwt^\delta\big[\Xpq{p-(n+l+\delta)}{(q-l)}\big]_{(n+\delta)\omega+l(2+2\sigma)}^* \right\}_{\substack{\delta=0,1\\0 \leq l < p-n}} \\
		 \cup\left\{ \big[\Xq{(q-p+n-l)}\big]_{\nu+2l}^* \right\}_{0 \leq l < q+n-p}
\end{multline*}
where $\nu = n\omega + (p-n)(2+2\sigma)$.
\item If $p-q \leq n \leq 0$, we can take as a basis the set
\begin{multline*}
	\left\{ \cxwt^{|n|-l}\big[\Xpq{p}{(q-l)}\big]_{l\chiw}^*
	\right\}_{0\leq l < |n|} \\
	\union
	\left\{\cxwt^\delta\big[\Xpq{p-l}{(q-|n|-l-\delta)}\big]_{(n+\delta)\omega+(l+|n|)(2+2\sigma)}^* \right\}_{\substack{\delta=0,1\\0 \leq l < p}}\\
	\union\left\{ \big[\Xq{(q-|n|-p-l))}\big]_{\nu+2l}^* \right\}_{0 \leq l < n-p+q}
\end{multline*}
where $\nu = n\omega + (p-n)(2+2\sigma)$

\vspace{0.1 cm}

\item If $-q < n < p-q$, we can take as a basis the set
\begin{multline*}
	\left\{ \cxwt^{|n|-l}\big[\Xpq{p}{(q-l)}\big]_{l\chiw}^*
	\right\}_{0\leq l < |n|} \\
	\union
	\left\{\cxwt^\delta\big[\Xpq{p-l}{(q-|n|-l-\delta)}\big]_{(n+\delta)\omega+(l+|n|)(2+2\sigma)}^* \right\}_{\substack{\delta=0,1\\0 \leq l < q-|n|}}\\
	\union\left\{ \big[\Xp{p-q+|n|-l}\big]_{\nu+2l}^* \right\}_{0 \leq l <p-q-n}
\end{multline*}
where $\nu = n\omega + q(2+2\sigma)$.

\item If $n \leq -q$, we can take as a basis the set
$$
\left\{ \cxwt^{|n|-l}\big[\Xpq{p}{(q-l)}\big]_{l\chiw}^*
	\right\}_{0\leq l < q} \union
		\left\{\big[\Xp{p-l}\big]_{\nu+2l}^* \right\}_{0\leq l<p},
$$
where $\nu=n\omega-2n + 2q\sigma$.\qed
\end{itemize}
\end{proposition}

\begin{remark}
We can pull back the bases above along the $C_2$-homeomorphism $\chi\colon \Xpq pq\to \Xpq qp$ to give similar, but
different bases. They can be obtained from the following equalities:
\begin{align*}
    \chi^*\cxwt &= \cwt \\
    \chi^*\cwt &= \cxwt \quad\text{and} \\
    \chi^*[\Xpq{p'}{q'}]_\nu^* &= [\Xpq{q'}{p'}]_{\chi\nu}^*.
\end{align*}

\end{remark}

\begin{example}
It might help to see the locations of the generators plotted in an example,
for which we will use $\Xpq{5}{3}$. 
Note that, in the cases $n\geq p$ and $n\leq -q$, the bases split naturally into two parts, one having $p$ elements and
the other having $q$ elements. This division can be maintained for the cases $-q < n < p$, though in several ways.

In the following diagrams we show one way to do that, with the basis elements from one part marked with
an open dot while the basis elements from the other part are marked with a closed dot.
The origin for the cohomology in gradings $n\omega+RO(C_2)$ is taken as $n\omega$
and the grid spacing is 2. Each diagram represents $\Mackey H_{C_2}^{n\omega + RO(C_2)}(\Xpq{5}{3}_+)$ where the value of $n$ is marked under the diagram itself.  



%
%
%

\vspace{0.3 cm}
$
\phantom{aa} \hspace{- 2 cm}
\begin{tikzpicture}[scale=0.6]
	\draw[step=1cm, gray, very thin] (-1.8, -0.8) grid (9.8, 5.8);
	\draw[thick] (-2, 0) -- (10, 0);
	\draw[thick] (0, -1) -- (0, 6);

 	\node[below] at (4, -1) {$n=-4$};
    \node[below] at (0.8,-2) {$a=\big[\Xp{5}\big]_{8+6\sigma-4\omega}^*$};
    \node[below] at (6.7,-2) {$b=\big[\Xp{}\big]_{16+6\sigma-4\omega}^*$};
    \node[below] at (0.7,-3) {$y=\cxwt^4\big[\Xpq{5}{3}\big]_0^*\,$};
    \node[below] at (6.9,-3) {$z=\cxwt^2\big[\Xpq{5}{}\big]_{2\chiw}^*\,$};
	\node[below, fill=white] at (4, 3) {$a$}; 
	\node[below, fill=white] at (8, 3){$b$};
	\node[above left] at (0, 4) {$y$};
	\node[above, fill=white] at (2, 4) {$z$};

	\draw[fill=white] (4, 3) circle(3pt);
	\draw[fill=white] (5, 3) circle(3pt);
	\draw[fill=white] (6, 3) circle(3pt);
	\draw[fill=white] (7, 3) circle(3pt);
	\draw[fill=white] (8, 3) circle(3pt);

	\fill (0, 4) circle(3pt);
	\fill (1, 4) circle(3pt);
	\fill (2, 4) circle(3pt);

\end{tikzpicture}
\hspace{0.5 cm}
\begin{tikzpicture}[scale=0.6]
	\draw[step=1cm, gray, very thin] (-1.8, -0.8) grid (9.8, 5.8);x
	\draw[thick] (-2, 0) -- (10, 0);
	\draw[thick] (0, -1) -- (0, 6);

 	\node[below] at (4, -1) {$n=-3$};
\node[below] at (0.7,-3) {$y=\cxwt^3\big[\Xpq{5}{3}\big]_0^*\,\,$};
    \node[below] at (6.9,-3) {$z=\cxwt\big[\Xpq{5}{}\big]_{2\chiw}^*\,\,$};
    \node[below] at (0.8,-2) {$a=\big[\Xp{5}\big]_{6+6\sigma-3\omega}^*$};
    \node[below] at (6.7,-2) {$b=\big[\Xp{}\big]_{12+6\sigma-3\omega}^*$};

	\node[below, fill=white] at (3, 3) {$a$};
	\node[below, fill=white] at (7, 3) {$b$}; 
	\node[above left] at (0, 3) {$y$}; 
	\node[above, fill=white] at (2, 3) {$z$};

	\draw[fill=white] (3, 3) circle(3pt);
	\draw[fill=white] (4, 3) circle(3pt);
	\draw[fill=white] (5, 3) circle(3pt);
	\draw[fill=white] (6, 3) circle(3pt);
	\draw[fill=white] (7, 3) circle(3pt);

	\fill (0, 3) circle(3pt);
	\fill (1, 3) circle(3pt);
	\fill (2, 3) circle(3pt);

\end{tikzpicture}
$

\vspace{0.5 cm}

$
\phantom{aa} \hspace{- 2 cm}
\begin{tikzpicture}[scale=0.6]
	\draw[step=1cm, gray, very thin] (-1.8, -0.8) grid (9.8, 4.8);
	\draw[thick] (-2, 0) -- (10, 0);
	\draw[thick] (0, -1) -- (0, 5);

 	\node[below] at (4, -1) {$n=-2$};

	\node[below, fill=white] at (0.5,-3) {$y=\cxwt^2\big[\Xpq{5}{3}\big]_0^*\,$};
	\node[below, fill=white] at (6.8, -3) {$z=\cxwt\big[\Xpq{4}{}\big]_{4+4\sigma-\omega}^*\,$};
	\node[below, fill=white] at (0.5, -2) {$a=\big[\Xpq{5}{}\big]_{2\chiw}^*$};
	\node[below, fill=white] at (6.3,-2) {$b=\big[\Xp{}\big]_{12+6\sigma-2\omega}^*$};
	\node[left, fill=white] at (0, 2) {$y$};
	\node[above left, fill=white] at (2, 3) {$z$};
	\node[below, fill=white] at (2, 2) {$a$};
	\node[below right, fill=white] at (6,3) {$b$};

	\draw[fill=white] (2, 2) circle(3pt);
	\draw[fill=white] (3, 3) circle(3pt);
	\draw[fill=white] (4, 3) circle(3pt);
	\draw[fill=white] (5, 3) circle(3pt);
	\draw[fill=white] (6, 3) circle(3pt);

	\fill (0, 2) circle(3pt);
	\fill (1, 2) circle(3pt);
	\fill (2, 3) circle(3pt);

\end{tikzpicture}
\hspace{0.5 cm}
\begin{tikzpicture}[scale=0.6]
	\draw[step=1cm, gray, very thin] (-2.8, -0.8) grid (8.8, 4.8);
	\draw[thick] (-3, 0) -- (9, 0);
	\draw[thick] (0, -1) -- (0, 5);

 	\node[below] at (3, -1) {$n=-1$};
    \node[below] at (-0.3,-3) {$y=\cxwt\big[\Xpq{5}{3}\big]_0^*\,$};
    \node[below] at (5.9,-3) {$z=\cxwt\big[\Xpq{3}{}\big]_{4+4\sigma}^*\,$};
    \node[below] at (-0.4,-2) {$a=\big[\Xpq{5}{2}\big]_{\chiw}^*$};
    \node[below] at (5.5,-2) {$b=\big[\Xp{}\big]_{10+6\sigma-\omega}^*$};
	\node[below right, fill=white] at (1, 1) {$a$};
	\node[below right, fill=white] at (5,3) {$b$};
	\draw[fill=white] (1, 1) circle(3pt);
	\draw[fill=white] (2, 2) circle(3pt);
	\draw[fill=white] (3, 3) circle(3pt);
	\draw[fill=white] (4, 3) circle(3pt);
	\draw[fill=white] (5, 3) circle(3pt);

	\node[above left, fill=white] at (0, 1) {$y$};
	\node[above left, fill=white] at (2, 3) {$z$};
	\fill (0, 1) circle(3pt);
	\fill (1, 2) circle(3pt);
	\fill (2, 3) circle(3pt);

\end{tikzpicture}
$

\vspace{0.5 cm}

$\phantom{aa} \hspace{- 2 cm}
\begin{tikzpicture}[scale=0.6]
	\draw[step=1cm, gray, very thin] (-3.8, -1.8) grid (7.8, 4.8);
	\draw[thick] (-4, 0) -- (8, 0);
	\draw[thick] (0, -2) -- (0, 5);

 	\node[below] at (2, -2) {$n=0$};
    \node[below] at (-1.6,-4) {$y=\cxwt\big[\Xpq{4}{3}\big]_{\omega}^*\,$};
    \node[below] at (4.5,-4) {$z=\cxwt\big[\Xpq{2}{}\big]_{4+4\sigma-\omega}^*$};
    \node[below] at (-1.6,-3) {$a=\big[\Xpq{5}{3}\big]_0^*\,$};
    \node[below] at (3.4,-3) {$b=\big[\Xp{}\big]_{8+6\sigma}^*\,$};
	\node[below, fill=white] at (0, 0) {$a$};
	\node[below right, fill=white] at (4,3) {$b$};
	\draw[fill=white] (0, 0) circle(3pt);
	\draw[fill=white] (1, 1) circle(3pt);
	\draw[fill=white] (2, 2) circle(3pt);
	\draw[fill=white] (3, 3) circle(3pt);
	\draw[fill=white] (4, 3) circle(3pt);

	\node[above left, fill=white] at (0, 1) {$y$};
	\node[above left, fill=white] at (2, 3) {$z$};
	\fill (0, 1) circle(3pt);
	\fill (1, 2) circle(3pt);
	\fill (2, 3) circle(3pt);

\end{tikzpicture}
\hspace{0.5 cm}
\begin{tikzpicture}[scale=0.6]
	\draw[step=1cm, gray, very thin] (-4.8, -1.8) grid (6.8, 4.8);
	\draw[thick] (-5, 0) -- (7, 0);
	\draw[thick] (0, -2) -- (0, 5);

 	\node[below] at (1, -2) {$n=1$};
    \node[below] at (-2.2, -3) {$a=\cwt\big[\Xpq{5}{3}\big]_0^*\,$};
    \node[below] at (3.4,-3) {$b=\big[\Xp{}\big]_{6+6\sigma+\omega}^*$};
    \node[below] at (-2.1,-4) {$y=\cxwt\big[\Xpq{3}{3}\big]_{2\omega}^*$};
    \node[below] at (4.2,-4) {$z=\cxwt\big[\Xpq{1}{}\big]_{4+4\sigma+2\omega}^*$};
	\node[below left, fill=white] at (-1, 0) {$a$};
	\node[below right, fill=white] at (3, 3) {$b$};
	\draw[fill=white] (-1, 0) circle(3pt);
	\draw[fill=white] (0, 0) circle(3pt);
	\draw[fill=white] (1, 1) circle(3pt);
	\draw[fill=white] (2, 2) circle(3pt);
	\draw[fill=white] (3, 3) circle(3pt);

	\node[above left, fill=white] at (0, 1) {$y$};
	\node[above left, fill=white] at (2, 3) {$z$};
	\fill (0, 1) circle(3pt);
	\fill (1, 2) circle(3pt);
	\fill (2, 3) circle(3pt);

\end{tikzpicture}
$

$\phantom{aa} \hspace{- 2 cm}
\begin{tikzpicture}[scale=0.6]
	\draw[step=1cm, gray, very thin] (-5.8, -1.8) grid (5.8, 4.8);
	\draw[thick] (-6, 0) -- (6, 0);
	\draw[thick] (0, -2) -- (0, 5);

 	\node[below] at (0, -2) {$n=2$};
    \node[below] at (-3.2,-3) {$a=\cwt^2\big[\Xpq{5}{3}\big]_0^*\,$};
    \node[below] at (2.8,-3) {$b=\big[\Xpq{1}{}\big]_{4+4\sigma+2\omega}^*\,$};
    \node[below] at (-3.1,-4) {$y=\cxwt\big[\Xpq{2}{3}\big]_{3\omega}^*$};
    \node[below] at (2.5,-4) {$z=\big[\Xq{}\big]_{4+6\sigma+2\omega}^*\,$};
	\node[below, fill=white] at (-2, 0) {$a$};
	\node[below right, fill=white] at (2, 2) {$b$};
	\draw[fill=white] (-2, 0) circle(3pt);
	\draw[fill=white] (-1, 0) circle(3pt);
	\draw[fill=white] (0, 0) circle(3pt);
	\draw[fill=white] (1, 1) circle(3pt);
	\draw[fill=white] (2, 2) circle(3pt);

	\node[above left, fill=white] at (0, 1) {$y$};
	\node[above right, fill=white] at (2, 3) {$z$};
	\fill (0, 1) circle(3pt);
	\fill (1, 2) circle(3pt);
	\fill (2, 3) circle(3pt);

\end{tikzpicture}
\hspace{0.5 cm}
\begin{tikzpicture}[scale=0.6]
	\draw[step=1cm, gray, very thin] (-5.8, -1.8) grid (5.8, 3.8);
	\draw[thick] (-6, 0) -- (6, 0);
	\draw[thick] (0, -2) -- (0, 4);

 	\node[below] at (0, -2) {$n=3$};
    \node[below] at (-3.3,-3) {$a=\cwt^3\big[\Xpq{5}{3}\big]_0^*\,$};
    \node[below] at (2.8,-3) {$b=\big[\Xpq{1}{2}\big]_{2+2\sigma+3\omega}^*\,$};
    \node[below] at (-3.2,-4) {$y=\cxwt\big[\Xpq{1}{3}\big]_{4\omega}^*$};
    \node[below] at (2.3,-4) {$z=\big[\Xq{}\big]_{4+4\sigma+3\omega}^*$};
	\node[below, fill=white] at (-3, 0) {$a$};
	\node[below right, fill=white] at (1, 1) {$b$};
	\draw[fill=white] (-3, 0) circle(3pt);
	\draw[fill=white] (-2, 0) circle(3pt);
	\draw[fill=white] (-1, 0) circle(3pt);
	\draw[fill=white] (0, 0) circle(3pt);
	\draw[fill=white] (1, 1) circle(3pt);

	\node[above left, fill=white] at (0, 1) {$y$};
	\node[above right, fill=white] at (2, 2) {$z$};
	\fill (0, 1) circle(3pt);
	\fill (1, 2) circle(3pt);
	\fill (2, 2) circle(3pt);

\end{tikzpicture}
$ 

\vspace{1 cm}

$\phantom{aa} \hspace{- 2 cm}
\begin{tikzpicture}[scale=0.6]
	\draw[step=1cm, gray, very thin] (-5.8, -1.8) grid (5.8, 2.8);
	\draw[thick] (-6, 0) -- (6, 0);
	\draw[thick] (0, -2) -- (0, 3);

 	\node[below] at (0, -2) {$n=4$};
    \node[below] at (-3,-3) {$a=\cwt^4\big[\Xpq{5}{3}\big]_0^*\,$};
    \node[below] at (2.8,-3) {$b=\big[\Xpq{1}{3}\big]_{4\omega}^*$};
    \node[below] at (-3,-4) {$y=\big[\Xq{3}\big]_{2\sigma+4\omega}^*\,$};
    \node[below] at (3,-4) {$z=\big[\Xq{}\big]_{4+2\sigma-4\omega}^*$};
	\node[below, fill=white] at (-4, 0) {$a$};
	\node[below, fill=white] at (0, 0) {$b$};
	\draw[fill=white] (-4, 0) circle(3pt);
	\draw[fill=white] (-3, 0) circle(3pt);
	\draw[fill=white] (-2, 0) circle(3pt);
	\draw[fill=white] (-1, 0) circle(3pt);
	\draw[fill=white] (0, 0) circle(3pt);

	\node[above, fill=white] at (0, 1) {$y$};
	\node[above right, fill=white] at (2, 1) {$z$};
	\fill (0, 1) circle(3pt);
	\fill (1, 1) circle(3pt);
	\fill (2, 1) circle(3pt);

\end{tikzpicture}
\hspace{0.5 cm}
\begin{tikzpicture}[scale=0.6]
	\draw[step=1cm, gray, very thin] (-6.8, -1.8) grid (4.8, 2.8);
	\draw[thick] (-7, 0) -- (5, 0);
	\draw[thick] (0, -2) -- (0, 3);

 	\node[below] at (-1, -2) {$n=5$};
    \node[below] at (-4,-3) {$a=\cwt^5\big[\Xpq{5}{3}\big]_0^*\,$};
    \node[below] at (1.7,-3) {$b=\cwt\big[\Xpq{1}{3}\big]_{4\omega}^*\,$};
    \node[below] at (-4.5,-4) {$y=\big[\Xq{3}\big]_{5\omega}^*$};
    \node[below] at (1.3,-4) {$z=\big[\Xq{}\big]_{4+5\omega}^*\,$};
	\node[above, fill=white] at (-5, 0) {$a$};
	\node[above, fill=white] at (-1, 0) {$b$};
	\draw[fill=white] (-5, 0) circle(3pt);
	\draw[fill=white] (-4, 0) circle(3pt);
	\draw[fill=white] (-3, 0) circle(3pt);
	\draw[fill=white] (-2, 0) circle(3pt);
	\draw[fill=white] (-1, 0) circle(3pt);

	\node[below, fill=white] at (0, 0) {$y$};
	\node[below right, fill=white] at (2, 0) {$z$};
	\fill (0, 0) circle(3pt);
	\fill (1, 0) circle(3pt);
	\fill (2, 0) circle(3pt);

\end{tikzpicture}
$

\vspace{1 cm}

$\phantom{aa} \hspace{- 2 cm}
\begin{tikzpicture}[scale=0.6]
	\draw[step=1cm, gray, very thin] (-7.8, -2.8) grid (3.8, 2.8);
	\draw[thick] (-8, 0) -- (4, 0);
	\draw[thick] (0, -3) -- (0, 3);

 	\node[below] at (-2, -3) {$n=6$};
    \node[below] at (-5.1,-4) {$a=\cwt^6\big[\Xpq{5}{3}\big]_0^*$};
    \node[below] at (1,-4) {$b=\cwt^2\big[\Xpq{1}{3}\big]_{4\omega}^*\,$};
    \node[below] at (-4.9,-5) {$y=\big[\Xq{3}\big]_{-2\sigma+6\omega}^*$};
    \node[below] at (1,-5) {$z=\big[\Xq{}\big]_{4-2\sigma+6\omega}^*\,$};
	\node[above, fill=white] at (-6, 0) {$a$};
	\node[above, fill=white] at (-2, 0) {$b$};
	\draw[fill=white] (-6, 0) circle(3pt);
	\draw[fill=white] (-5, 0) circle(3pt);
	\draw[fill=white] (-4, 0) circle(3pt);
	\draw[fill=white] (-3, 0) circle(3pt);
	\draw[fill=white] (-2, 0) circle(3pt);

	\node[below, fill=white] at (0, -1) {$y$};
	\node[below right, fill=white] at (2, -1) {$z$};
	\fill (0, -1) circle(3pt);
	\fill (1, -1) circle(3pt);
	\fill (2, -1) circle(3pt);

\end{tikzpicture}
\hspace{0.5 cm}
\begin{tikzpicture}[scale=0.6]
	\draw[step=1cm, gray, very thin] (-7.8, -2.8) grid (3.8, 2.8);
	\draw[thick] (-8, 0) -- (4, 0);
	\draw[thick] (0, -3) -- (0, 3);

  	\node[below] at (-2, -3) {$n=7$};
    \node[below] at (-5.1,-4) {$a=\cwt^7\big[\Xpq{5}{3}\big]_0^*$};
    \node[below] at (1,-4) {$b=\cwt^3\big[\Xpq{1}{3}\big]_{4\omega}^*\,$};
    \node[below] at (-4.9,-5) {$y=\big[\Xq{3}\big]_{-4\sigma+7\omega}^*$};
    \node[below] at (1,-5) {$z=\big[\Xq{}\big]_{4-2\sigma+7\omega}^*$};
	\node[above, fill=white] at (-7, 0) {$a$};
	\node[above, fill=white] at (-3, 0) {$b$};
	\draw[fill=white] (-7, 0) circle(3pt);
	\draw[fill=white] (-6, 0) circle(3pt);
	\draw[fill=white] (-5, 0) circle(3pt);
	\draw[fill=white] (-4, 0) circle(3pt);
	\draw[fill=white] (-3, 0) circle(3pt);

	\node[below, fill=white] at (0, -2) {$y$};
	\node[below right, fill=white] at (2, -2) {$z$};
	\fill (0, -2) circle(3pt);
	\fill (1, -2) circle(3pt);
	\fill (2, -2) circle(3pt);

\end{tikzpicture}
$
\end{example}

\newpage
Before moving on to discuss the multiplicative structure of the cohomology, we need the following
result, which follows easily from the definition of the basis $F_{p,q}(n)$
and the structure of $\Mackey H_{C_2}^{RO(C_2)}(S^0)$.

\begin{corollary}
\label{cor:pushforward}
The push-forward
\[
 i_!\colon \Mackey H_{C_2}^{(n-1)\omega+RO(C_2)}(\Xpq{p-1}{q}_+)
 	\to \Mackey H_{C_2}^{n\omega+RO(C_2)}(\Xpq{p}{q}_+).
\]
is a split monomorphism if $n\geq 0$.
The cokernel is a free $\Mackey H_{C_2}^{RO(C_2)}(S^0)$-module on a single generator in grading
$n(\omega-2)$.
\[
 i_!\colon \Mackey H_{C_2}^{(n-1)\omega+a+b\sigma+2}(\Xpq{p-1}{q}_+) 
 	\to \Mackey H_{C_2}^{n\omega+a+b\sigma}(\Xpq{p}{q}_+)
\]
is an isomorphism if, in addition, $a > -2n$ and $a+b > -2n$.

The push-forward
\[
 j_!\colon \Mackey H_{C_2}^{(n+1)\omega+RO(C_2)}(\Xpq{p}{(q-1)}_+) 
  \to \Mackey H_{C_2}^{n\omega+RO(C_2)}(\Xpq{p}{q}_+).
\]
is a split monomorphism if $n\leq 0$.
The cokernel is a free $\Mackey H_{C_2}^{RO(C_2)}(S^0)$-module on a single generator in grading $n(\omega-2\sigma)$.
\[
 j_!\colon \Mackey H_{C_2}^{(n+1)\omega+a+b\sigma-2-2\sigma}(\Xpq{p}{(q-1)}_+) 
 	\to \Mackey H_{C_2}^{n\omega+a+b\sigma}(\Xpq{p}{q}_+)
\]
is an isomorphism if, in addition, $a > 0$ and $a+b > -2n$.
\qed
\end{corollary}

\subsection{The multiplicative structure}\label{susbsec: Multiplicative}

We can now show our main result, giving the multiplicative structure of the cohomology of $\Xpq pq$.

\begin{theorem}\label{thm:multiplicative}
Let $0 \leq p < \infty$ and $0 \leq q < \infty$ with $p+q > 0$.
As an algebra over $\Mackey H_{C_2}^{RO(C_2)}(S^0)$,
$\Mackey H_{C_2}^{RO(\Pi B)}(\Xpq{p}{q}_+)$ is generated by the elements
$\cw$, $\cxw$, $\cwt$, and $\cxwt$,
together with the following elements:
$\cw^p$ is infinitely divisible by $\cxwt$, meaning that, for $k\geq 1$,
there are unique elements $\cxwt^{-k}\cw^p$ such that
\[
 \cxwt^k \cdot \cxwt^{-k} \cw^p = \cw^p.
\]
Similarly, $\cxw^q$ is infinitely divisible by $\cwt$, meaning that, for $k\geq 1$,
there are unique elements $\cwt^{-k}\cxw^q$ such that
\[
 \cwt^k \cdot \cwt^{-k} \cxw^q = \cxw^q.
\]
The generators satisfy the following relations:
\begin{align*}
	\cw^p \cxw^q &= 0 \\
	\cwt \cxw &= (1-\kappa)\cxwt \cw + e^2 \qquad\text{and} \\
	\cxwt \cwt &= \xi.
\end{align*}
\end{theorem}

\begin{proof}
We know that the theorem is true if $p=0$ or $q=0$, and we proceed by induction on $p+q$. 
So assume that $p\geq 1$, $q\geq 1$, and the theorem is true for all
$p'$ and $q'$ with $p'+q' < p+q$.

We have already constructed the elements that we claim are multiplicative generators,
including the elements $\cxwt^{-k}\cw^p$ and $\cxwt^{-k}\cw^p$.
Proposition~\ref{prop:additivestructure} and Lemma~\ref{lem:pushforwardfixed} show that they generate, so it remains only to show
that they satisfy the relations given and that these completely determine the structure.

The cohomology of $\Xpq{p}{q}$ inherits from the cohomology of $\Xpq{\infty}{\infty}$
the relations
\begin{align*}
	\cwt \cxw &= (1-\kappa)\cxwt \cw + e^2 \qquad\text{and} \\
	\cxwt \cwt &= \xi.
\end{align*}
For the relation $\cw^p \cxw^q = 0$, there are two arguments:
One is to examine the bases calculated in Proposition~\ref{prop:additivestructure}
to see that the cohomology is 0 in the grading in which this element would live. A quicker
argument is inductive: If $j\colon \Xpq{p}{(q-1)} \to \Xpq{p}{q}$
is the inclusion, then
\[
 \cw^p \cxw^q = -(1-\kappa)(1-\epsilon) j_!(\cw^{p}\cxw^{q-1}) = 0
\]
by induction.

The uniqueness of the elements $\cxwt^{-k}\cw^p$ and $\cwt^{-k}\cxw^q$ will follow
from the cancellation properties we prove next.
The completeness of the set of relations will follow from Proposition~\ref{prop:relations}.
\end{proof}

The following two results complete the description of the multiplicative structure.

\begin{proposition}\label{prop: cancellation}
The elements $\cwt$ and $\cxwt$ have the following cancellation properties.
For an element $x\in \Mackey H_{C_2}^{n\omega+a+b\sigma}(\Xpq{p}{q}_+)$
with $n$, $a$, $b \in \Z$, we have
\begin{align*}
 \cxwt x &= 0 \implies x = 0
 	&&\text{if }  n\geq p,\ a \geq 2(p-n), \text{ and } a+b \geq 2(p-n) \\
	&&&\text{or if }  p-q \leq n < p,\ a \geq 2(p-n), \text{ and } a+b \geq 4(p-n); \\
 \cwt x &= 0 \implies x = 0
 	&&\text{if }  n \leq -q,\ a \geq 2q, \text{ and } a+b \geq 2(q-n) \\
	&&&\text{or if }  -q < n \leq p-q,\ a \geq 2q, \text{ and } a+b \geq 4q.
\end{align*}
\end{proposition}

\begin{proof}
We assume that $p\geq q$, the case $p < q$ following by symmetry.
We will proceed by induction, so we first establish the base cases.

When $q = 0$, the element $\cwt$ is invertible, so its cancellation property is clear.
For $\cxwt$, the only case we need to consider is when $n \geq p$. The generators in
gradings $n\omega+RO(C_2)$ are then
\[
 \{ \cwt^n, \cwt^{n-1}\cw, \ldots, \cwt^{n-p+1}\cw^{p-1} \},
\]
with gradings
\[
 n\omega - 2n, n\omega - 2(n-1), \ldots, n\omega - 2(n-p+1) = n\omega + 2(p-n-1).
\]
So if $x$ has grading $n\omega+a+b\sigma$ with $n\geq p$, $a\geq 2(p-n)$, and $a+b\geq 2(p-n)$,
then $x$ lives in a 0 group, hence $x=0$ and the cancellation property of $\cxwt$ is trivially true.

Because we are assuming $p \geq q$, the case $p=0$ is covered by the case $q=0$.
We now proceed by induction on $p+q$, so we assume that $p\geq 1$, $q\geq 1$,
and the cancellation properties are true
for all $\Xpq{p'}{q'}$ with $p'+q' < p+q$.

Consider first the cancellation property of $\cwt$.
If $j\colon \Xpq{p}{(q-1)}\to \Xpq{p}{q}$ is the inclusion, we have
the following commutative diagram:
\[
 \xymatrix{
 	\Mackey H_{C_2}^{(n+1)\omega+(a-2)+(b-2)\sigma}(\Xpq{p}{(q-1)}_+)
	 \ar[r]^-{j_!} \ar[d]_{\cwt\cdot}
	 & \Mackey H_{C_2}^{n\omega+a+b\sigma}(\Xpq{p}{q}_+) \ar[d]^{\cwt\cdot} \\
	\Mackey H_{C_2}^{(n+2)\omega+(a-4)+(b-2)\sigma}(\Xpq{p}{(q-1)}_+) \ar[r]_-{j_!}
	 & \Mackey H_{C_2}^{(n+1)\omega+(a-2)+b\sigma}(\Xpq{p}{q}_+)\,.
 }
\]
We know from Corollary~\ref{cor:pushforward}
that the top map is an isomorphism if $n\leq 0$, $a > 0$, and $a+b > -2n$,
while the bottom map is a monomorphism for $n\leq -1$.
Further, we are assuming that the map on the left is a monomorphism if either
$n\leq -q$, $a \geq 2q$, and $a+b\geq 2(q-n)$, or
$-q < n\leq p-q$, $a\geq 2q$, and $a+b\geq 4q$.
Because we are assuming that $q \geq 1$,
it follows that $j_!(\cwt\cdot-)$, hence $\cwt j_!(-)$, is a monomorphism if
$n \leq -q$, $a\geq 2q$, and $a+b\geq 2(q-n)$.
Because the $j_!$ along the top is an isomorphism in this range, it follows
that multiplication by $\cwt$ on the right of the diagram is a monomorphism in this range.

That multiplication by $\cwt$ is a monomorphism in the range
$-q < n < 0$, $a\geq 2q$, and $a+b \geq 4q$ follows similarly.

Now consider the following diagram:
\[
 \xymatrix{
	\Mackey H_{C_2}^{(n-1)\omega+(a+2)+b\sigma}(\Xpq{(p-1)}{q}_+)
		\ar[r]^-{i_!} \ar[d]_{\cwt\cdot}
	 & \Mackey H_{C_2}^{n\omega+a+b\sigma}(\Xpq{p}{q}_+) \ar[d]^{\cwt\cdot} \\
	\Mackey H_{C_2}^{n\omega+a+b\sigma}(\Xpq{(p-1)}{q}_+) \ar[r]_{i_!}
	 & \Mackey H_{C_2}^{(n+1)\omega+(a-2)+b\sigma}(\Xpq{p}{q}_+)\,.
 }
\]
We know from Corollary~\ref{cor:pushforward}
that the top map is an isomorphism if $n\geq 0$, $a > -2n$, and $a+b > -2n$,
and that the bottom map is a monomorphism for $n\geq -1$.
We are assuming that the map on the left is a monomorphism for
$-q+1 < n \leq p-q$, $a \geq 2q-2$, and $a+b \geq 4q-2$.
It follows that $i_!(\cwt\cdot-)$ is a monomorphism for
$0 < n \leq p-q$, $a \geq 2q-2$, and $a+b\geq 4q-2$.
Because the $i_!$ along the top is an isomorphism in this range, it follows that
multiplication by $\cwt$ on the right of the diagram is a monomorphism in this range,
which completes the inductive argument that multiplication by $\cwt$ is
a monomorphism in the ranges stated in the proposition.

For the case of multiplication by $\cxwt$, we look at the following two diagrams:
\[
 \xymatrix{
 	\Mackey H_{C_2}^{(n+1)\omega+(a-2)+(b-2)\sigma}(\Xpq{p}{(q-1)}_+)
	 \ar[r]^-{j_!} \ar[d]_{\cxwt\cdot}
	 & \Mackey H_{C_2}^{n\omega+a+b\sigma}(\Xpq{p}{q}_+) \ar[d]^{\cxwt\cdot} \\
	\Mackey H_{C_2}^{n\omega+a+b\sigma}(\Xpq{p}{(q-1)}_+) \ar[r]_-{j_!}
	 & \Mackey H_{C_2}^{(n-1)\omega+(a+2)+(b+2)\sigma}(\Xpq{p}{q}_+)
 }
\]
and
\[
 \xymatrix{
	\Mackey H_{C_2}^{(n-1)\omega+(a+2)+b\sigma}(\Xpq{(p-1)}{q}_+)
		\ar[r]^-{i_!} \ar[d]_{\cxwt\cdot}
	 & \Mackey H_{C_2}^{n\omega+a+b\sigma}(\Xpq{p}{q}_+) \ar[d]^{\cxwt\cdot} \\
	\Mackey H_{C_2}^{(n-2)\omega+(a+4)+(b+2)\sigma}(\Xpq{(p-1)}{q}_+) \ar[r]_{i_!}
	 & \Mackey H_{C_2}^{(n-1)\omega+(a+2)+(b+2)\sigma}(\Xpq{p}{q}_+)\,.
 }
\]
In the first diagram, the top map is an isomorphism for
$n\leq 0$, $a > 0$, and $a+b > -2n$, the bottom map is a monomorphism for $n\leq 1$,
and we assume that the left map is a monomorphism for 
$p-q \leq n < p-1$, $a\geq 2(p-n-1)$, and $a+b\geq 4(p-n-1)$.
This all implies that multiplication by $\cxwt$ on the right is a monomorphism in the range
$p-q \leq n < 0$, $a \geq 2(p-n)$, and $a+b \geq 4(p-n)$,

In the second diagram, the top map is an isomorphism for
$n\geq 0$, $a > -2n$, and $a+b > -2n$, the bottom map is a monomorphism for $n\geq -1$,
and we assume that the left map is a monomorphism for either
$n \geq p$, $a\geq 2(p-n)$, and $a+b \geq 2(p-n)$, or
$-q + 1 \leq n < p$, $a \geq 2(p-n)-2$, and $a+b \geq 4(p-n)-2$.
This all implies that multiplication by $\cxwt$ on the right is a monomorphism in the ranges
$n\geq p$, $a \geq 2(p-n)$, and $a+b\geq 2(p-n)$, or
$0 \leq n < p$, $a \geq 2(p-n)$, and $a+b \geq 4(p-n)$,
completing the inductive argument that multiplication by $\cxwt$ is a monomorphism
in the ranges stated in the proposition.
\end{proof}


The following result gives additional relations in the cohomology of $\Xpq pq$,
enough to conclude that the relations given in Theorem~\ref{thm:multiplicative} suffice to determine
the whole multiplicative structure.

\begin{proposition}
\label{prop:relations}
The
following relations in $\Mackey H_{C_2}^{RO(\Pi B)}(\Xpq{p}{q}_+)$ hold for all $k$ and $\ell$
and follow from those stated in Theorem~\ref{thm:multiplicative}.
\begin{align*}
	\cxwt^{-k}c_{\omega}^{p+1} 
		&= e^2\cxwt^{-(k+1)}c_{\omega}^p + \xi \cxwt^{-(k+2)}c_{\omega}^p c_{\chiw} \\
	\cwt^{-k}c_{\chiw}^{q+1} 
		&= e^2 \cwt^{-(k+1)} c_{\chiw}^q + \xi \cwt^{-(k+2)} c_{\omega} c_{\chiw}^q  \\
	\cxwt^{-k}c_{\omega}^m c_{\chiw}^n &= 0 & m &\geq p \text{ and } n \geq q \\
	\cwt^{-k}c_{\omega}^m c_{\chiw}^n &= 0 & m &\geq p \text{ and } n \geq q \\
	\cxwt\cdot \cxwt^{-k}c_{\omega}^m c_{\chiw}^n 
		&= \cxwt^{-(k-1)}c_{\omega}^m c_{\chiw}^n & m &\geq p \\
	\cwt\cdot \cxwt^{-k}c_{\omega}^m c_{\chiw}^n  
		&= \xi \cxwt^{-(k+1)}c_{\omega}^m c_{\chiw}^n & m &\geq p \\
	\cxwt^{-k}c_{\omega}^m c_{\chiw}^n \cdot \cxwt^{-\ell}c_{\omega}^s c_{\chiw}^t
		&= \cxwt^{-k-\ell}c_{\omega}^{m+s} c_{\chiw}^{n+t} & m, s &\geq p \\
	\cxwt\cdot \cwt^{-k}c_{\omega}^m c_{\chiw}^n 
		&= \xi \cwt^{-(k+1)}c_{\omega}^m c_{\chiw}^n & n &\geq q \\
	\cwt\cdot \cwt^{-k}c_{\omega}^m c_{\chiw}^n  
		&= \cwt^{-(k-1)}c_{\omega}^m c_{\chiw}^n & n &\geq q \\
	\cwt^{-k}c_{\omega}^m c_{\chiw}^n \cdot \cwt^{-\ell}c_{\omega}^s c_{\chiw}^t
		&= \cwt^{-k-\ell}c_{\omega}^{m+s} c_{\chiw}^{n+t} & n, t &\geq q \\
	\cxwt^{-k}c_{\omega}^m c_{\chiw}^n \cdot \cwt^{-\ell}c_{\omega}^s c_{\chiw}^t
	 	&= 0 & m &\geq p \text{ and } t\geq q
\end{align*}
where we write
\begin{align*}
 \cxwt^{-k} c_{\omega}^m c_{\chiw}^n 
 	&:= (\cxwt^{-k}c_{\omega}^p) c_{\omega}^{m-p}c_{\chiw}^n
	&& m \geq p \text{ and } k \geq 1 \\
\intertext{and}
 \cwt^{-k} c_{\omega}^m c_{\chiw}^n 
 	&:= (\cwt^{-k}c_{\chiw}^q) c_{\omega}^{m}c_{\chiw}^{n-q}
	&& n \geq q \text{ and } k \geq 1.
\end{align*}
\end{proposition}

As mentioned, the relations listed above determine the complete multiplicative structure of $\Mackey H_{C_2}^{RO(\Pi B)}(\Xpq{p}{q}_+)$,
as they can be used to rewrite any expression in the generators in terms of the additive basis given in
Proposition~\ref{prop:additivestructure}.
For example, the relation 
$\cxwt^{-k}c_{\omega}^{p+1} = e^2\cxwt^{-(k+1)}c_{\omega}^p + \xi \cxwt^{-(k+2)}c_{\omega}^p c_{\chiw}$
gives us the equality
\[
 \cxwt^{-k} c_{\omega}^{p+1} c_{\chiw}^n
  = e^2\cxwt^{-(k+1)}c_{\omega}^p c_{\chiw}^n + \xi \cxwt^{-(k+2)}c_{\omega}^p c_{\chiw}^{n+1}.
\]
By induction, we can then show that
\begin{eqnarray}
\label{eqn:relation}
 \cxwt^{-k} c_{\omega}^{p+m} c_{\chiw}^n
  = \sum_{\ell=0}^m \binom{m}{\ell} e^{2(m-\ell)}\xi^\ell \cxwt^{-(k+m+\ell)}
  	c_{\omega}^p c_{\chiw}^{n+\ell},
\end{eqnarray}
with the understanding that any terms with $n+\ell \geq q$ vanish.
With relations like these, we can rewrite any expression in terms of monomials in $\cw$ and $\cxw$ with exponents not exceeding
$p$ and $q$, respectively.

\begin{proof}[Proof of Proposition \ref{prop:relations}]
We will frequently make use of Proposition~\ref{prop: cancellation} throughout the proof.
Start with the claim that 
$\cxwt\cdot \cxwt^{-k}c_{\omega}^m c_{\chiw}^n = \cxwt^{-(k-1)}c_{\omega}^m c_{\chiw}^n$
for $m\geq p$. Consider the special case
$\cxwt \cdot \cxwt^{-k}c_{\omega}^p = \cxwt^{-(k-1)}c_{\omega}^p$ first.
We have
\[
 \cxwt^{k-1}(\cxwt \cdot \cxwt^{-k}c_{\omega}^p) = \cxwt^k \cdot \cxwt^{-k}c_{\omega}^p
 = c_{\omega}^p
\]
by definition, so we must have $\cxwt \cdot \cxwt^{-k}c_{\omega}^p = \cxwt^{-(k-1)}c_{\omega}^p$
as claimed, by the uniqueness of $\cxwt^{-(k-1)}c_{\omega}^p$.
The general case follows: If $m\geq p$, then
\begin{align*}
 \cxwt\cdot \cxwt^{-k}c_{\omega}^m c_{\chiw}^n
 	&= (\cxwt \cdot \cxwt^{-k}c_{\omega}^p) c_{\omega}^{m-p} c_{\chiw}^n \\
	&= (\cxwt^{-(k-1)} c_{\omega}^p) c_{\omega}^{m-p} c_{\chiw}^n \\
	&= \cxwt^{-(k-1)}c_{\omega}^m c_{\chiw}^n.
\end{align*}
The calculation of $\cwt\cdot \cwt^{-k}c_{\omega}^m c_{\chiw}^n$ is the same.

This gives us the easy calculation
\[
 \cwt\cdot \cxwt^{-k}c_{\omega}^m c_{\chiw}^n
  = \cwt\cxwt \cdot \cxwt^{-(k+1)} c_{\omega}^m c_{\chiw}^m
  = \xi \cxwt^{-(k+1)} c_{\omega}^m c_{\chiw}^m
\]
and similarly for $\cxwt\cdot \cwt^{-k}c_{\omega}^m c_{\chiw}^n$.

For the calculation of $\cxwt^{-k}c_{\omega}^{p+1}$, we have
\begin{align*}
 \cxwt^{-k}c_{\omega}^{p+1}
 	&= (\cxwt c_{\omega})\cxwt^{-(k+1)}c_{\omega}^p \\
	&= (e^2 + (1-\kappa)\cwt c_{\chiw})\cxwt^{-(k+1)}c_{\omega}^p \\
	&= e^2 \cxwt^{-(k+1)}c_{\omega}^p + (1-\kappa)\xi \cxwt^{-(k+2)} c_{\omega}^p c_{\chiw} \\
	&= e^2 \cxwt^{-(k+1)}c_{\omega}^p + \xi \cxwt^{-(k+2)} c_{\omega}^p c_{\chiw}.
\end{align*}
The calculation of $\cwt^{-k}c_{\chiw}^{q+1}$ is the same, {\it mutatis mutandis}.

Now consider the claim that $\cxwt^{-k}c_{\omega}^m c_{\chiw}^n = 0$ for
$m\geq p$ and $n \geq q$. It follows easily if $k \leq 0$, so the interesting case is $k > 0$.
Consider first $\cxwt^{-k} c_{\omega}^p c_{\chiw}^q$
and proceed by induction on $k$, the case $k=0$ being already known.
Inductively, we have that
\[
 \cxwt \cdot \cxwt^{-k} c_{\omega}^p c_{\chiw}^q
  = \cxwt^{-(k-1)}c_{\omega}^p c_{\chiw}^q = 0,
\]
so we need to check if we are in a grading where
we can cancel $\cxwt$ and conclude that $\cxwt^{-k} c_{\omega}^p c_{\chiw}^q = 0$.
This element lives in grading
$(p-q+k)\omega + 2q + 2(q-k)\sigma$.
With $n = p-q+k$ and $\alpha = 2q + 2(q-k)\sigma$, we see that
\begin{align*}
	|\alpha| &= 4q - 2k = 4(p-n) + 2k \\
\intertext{and}
	\alpha^{C_2} &= 2q = 2(p-n) + 2k.
\end{align*}
If $0 < k \leq q$, so that $p-q < n \leq p$, we have $|\alpha| > 4(p-n)$ and $\alpha^{C_2} > 2(p-n)$.
If $k > q$, then $n > p$ and we have
\[
 |\alpha| = 4(p-n) + 2k = 2(p-n) + 2(p-n+k) = 2(p-n) + 2q > 2(p-n)
\]
and again
\[
 \alpha^{C_2} = 2(p-n) + 2k > 2(p-n).
\]
Thus, for any $k > 0$, we are in a grading where $\cxwt$ can be canceled and we conclude that
$\cxwt^{-k} c_{\omega}^p c_{\chiw}^q = 0$. The general statement then follows for 
$m\geq p$ and $n\geq q$:
\[
 \cxwt^{-k}c_{\omega}^m c_{\chiw}^n
 	= (\cxwt^{-k} c_{\omega}^p c_{\chiw}^q)c_{\omega}^{m-p} c_{\chiw}^{n-q}
	= 0.
\]
The proof that $\cwt^{-k}c_{\omega}^m c_{\chiw}^n = 0$ for $m\geq p$ and $n\geq q$ is similar.

To calculate $\cxwt^{-k}c_{\omega}^m c_{\chiw}^n \cdot \cxwt^{-\ell}c_{\omega}^s c_{\chiw}^t$,
$m\geq p$ and $s\geq p$, it suffices to consider the special case
$\cxwt^{-k}c_{\omega}^p \cdot \cxwt^{-\ell}c_{\omega}^p$.
For this we proceed by induction on $k$, the case of $k=0$ being known by definition.
We have
\begin{align*}
 \cxwt \cdot (\cxwt^{-k}c_{\omega}^p \cdot \cxwt^{-\ell}c_{\omega}^p)
	&= \cxwt^{-(k-1)}c_{\omega}^p \cdot \cxwt^{-\ell}c_{\omega}^p \\
	&= \cxwt^{-(k+\ell-1)} c_{\omega}^{2p} \\
	&= \cxwt \cdot \cxwt^{-(k+\ell)} c_{\omega}^{2p}
\end{align*}
by induction, so we need to check that we are in a grading where we can cancel $\cxwt$.
The grading is $(2p+k+\ell)\omega - 2(k+\ell)\sigma$. With $n = 2p + k + \ell$ and
$\alpha = -2(k+\ell)\sigma$, we have $n>p$, 
\[
	|\alpha| = -2(k+\ell) = 2(2p-n) \geq 2(p-n),
\]
and $\alpha^{C_2} = 0 > 2(p-n)$. Thus we are in a grading where we can cancel $\cxwt$, so we conclude that
$\cxwt^{-k}c_{\omega}^p \cdot \cxwt^{-\ell}c_{\omega}^p = \cxwt^{-(k+\ell)}c_{\omega}^{2p}$,
and the general case follows.
The calculation of 
$\cwt^{-k}c_{\omega}^m c_{\chiw}^n \cdot \cwt^{-\ell}c_{\omega}^s c_{\chiw}^t$
for $n\geq q$ and $t\geq q$ is similar.

Finally, consider
$\cxwt^{-k}c_{\omega}^m c_{\chiw}^n \cdot \cwt^{-\ell}c_{\omega}^s c_{\chiw}^t$
for $m\geq p$ and $t\geq q$, for which it suffices to consider the case
$\cxwt^{-k}c_{\omega}^p \cdot \cwt^{-\ell}c_{\chiw}^q$ with $k\geq 0$ and $\ell \geq 0$.
We proceed by induction on $k+\ell$, the base case $k+\ell = 0$ being known.
We want to take advantage of either
\begin{align*}
	\cxwt (\cxwt^{-k}c_{\omega}^p \cdot \cwt^{-\ell}c_{\chiw}^q)
		&= \cxwt^{-(k-1)}c_{\omega}^p \cdot \cwt^{-\ell}c_{\chiw}^q
		= 0 \\
\intertext{or}
	\cwt (\cxwt^{-k}c_{\omega}^p \cdot \cwt^{-\ell}c_{\chiw}^q)
		&= \cxwt^{-k}c_{\omega}^p \cdot \cwt^{-(\ell-1)}c_{\chiw}^q
		= 0
\end{align*}
to conclude that $\cxwt^{-k}c_{\omega}^p \cdot \cwt^{-\ell}c_{\chiw}^q = 0$.
We first note that
$\cxwt^{-k}c_{\omega}^p \cdot \cwt^{-\ell}c_{\chiw}^q$
lives in grading $(p-q+k-\ell)\omega + 2(q+\ell) + 2(q-k)\sigma$,
so let $n = p-q+k-\ell$ and $\alpha = 2(q+\ell) + 2(q-k)\sigma$.

There are several cases to consider. First, suppose that $k-\ell \geq q$ so that $n\geq p$.
We have
\begin{align*}
	|\alpha| &= 2(2q-k+\ell) = 2(p-n+q) \geq 2(p-n) \\
\intertext{and}
	\alpha^{C_2} &= 2(q+\ell) = 2(p-n+k) \geq 2(p-n),
\end{align*}
so we are in a grading where 
$\cxwt (\cxwt^{-k}c_{\omega}^p \cdot \cwt^{-\ell}c_{\chiw}^q) = 0$ implies that the second factor is trivial.

If $0 \leq k-\ell \leq q$, so $p-q \leq n \leq p$ and $q \geq p-n$, then
\begin{align*}
	|\alpha| &= 2(p-n+q) \geq 4(p-n) \\
\intertext{and}
	\alpha^{C_2} &= 2(q+\ell) = 2(p-n+k) \geq 2(p-n),
\end{align*}
so we can again conclude that $\cxwt^{-k}c_{\omega}^p \cdot \cwt^{-\ell}c_{\chiw}^q = 0$.

If $-p \leq k - \ell \leq 0$, so $-q \leq n \leq p-q$ and $p-n \geq q$, then
\begin{align*}
	|\alpha| &= 2(p-n+q) \geq 4q \\
\intertext{and}
	\alpha^{C_2} &= 2(q+\ell) \geq 2q,
\end{align*}
so we can conclude from
$\cwt (\cxwt^{-k}c_{\omega}^p \cdot \cwt^{-\ell}c_{\chiw}^q) = 0$
that $\cxwt^{-k}c_{\omega}^p \cdot \cwt^{-\ell}c_{\chiw}^q = 0$.

If $k - \ell \leq -p$, so $n \leq -q$, then
\begin{align*}
	|\alpha| &= 2(p-n+q) \geq 2(q-n) \\
\intertext{and}
	\alpha^{C_2} &= 2(q+\ell)  \geq 2q,
\end{align*}
so we can again conclude that $\cxwt^{-k}c_{\omega}^p \cdot \cwt^{-\ell}c_{\chiw}^q = 0$.
\end{proof}


\section{Related results}\label{sec: leftovers}

In this section, we collect remarks of interest.
The first subsection includes the computations of the cohomology of $\Xpq{p}{q}$ when one of $p$ or $q$ is finite and the other infinite.
Then we discuss the $RO(C_2)$-graded subring of the cohomology and cohomology with coefficients in other Mackey functors.

\subsection{The case $p=\infty$ or $q=\infty$}
\label{subsec infinite cases}
Consider $\Xpq{\infty}{q}$ with $q < \infty$. This space can be written as
\[
 \Xpq{\infty}{q} = \colim_p \Xpq{p}{q}
\]
and we can consider the resulting limit of cohomology groups.
It follows from Proposition \ref{prop:additivestructure} that the restriction
$\Mackey H_{C_2}^{n\omega+RO(C_2)}(\Xpq{p}{q}_+) \to \Mackey H_{C_2}^{n\omega+RO(C_2)}(\Xpq{p-1}{q}_+)$
is an isomorphism in a range of degrees that expands as $p$ increases. 
From this, and similar considerations for $\Xpq{p}{\infty}$, 
we can conclude the following result.

\begin{theorem}
Let $0 \leq q < \infty$. Then $\Mackey H_{C_2}^{RO(\Pi B)}(\Xpq{\infty}{q}_+)$ is generated, 
as an algebra over $\Mackey H_{C_2}^{RO(C_2)}(S^0)$,
 by the elements
$c_{\omega}$, $c_{\chiw}$, $\cwt$, and $\cxwt$,
together with the following elements: 
$c_{\chiw}^q$ is infinitely divisible by $\cwt$, meaning that, for $k\geq 1$,
there are unique elements $\cwt^{-k}c_{\chiw}^q$ such that
\[
	\cwt^k \cdot \cwt^{-k} c_{\chiw}^q = c_{\chiw}^q.
\]
The generators satisfy the following relations:
\begin{align*}
	\cwt c_{\chiw} - (1-\kappa)\cxwt c_{\omega} &= e^2 \qquad\text{and} \\
	\cxwt \cwt &= \xi.
\end{align*}
The element $\cwt$ has the same cancellation property
as given in Proposition \ref{prop: cancellation}:
For an element $x \in \Mackey H_{C_2}^{n\omega + \alpha}(\Xpq{p}{q}_+)$ with 
$n \in \Z$ and $\alpha\in RO(C_2)$,
we have
\begin{align*}
 \cwt x &= 0 \implies x = 0
 	&&\text{if }  n \leq -q,\ |\alpha| \geq 2(q-n), \text{ and } \alpha^{C_2} \geq 2q \\
	&&&\text{or if }  -q \leq n,\ |\alpha| \geq 4q, \text{ and } \alpha^{C_2} \geq 2q.
\end{align*}
We define
\begin{align*}
 \cwt^{-k} c_{\omega}^m c_{\chiw}^n 
 	&:= (\cwt^{-k}c_{\chiw}^q) c_{\omega}^{m}c_{\chiw}^{n-q}
	&& n \geq q \text{ and } k \geq 1.
\end{align*}
Additively, $\Mackey H_{C_2}^{RO(\Pi B)}((\Xpq{\infty}{q})_+)$ is a free
$\Mackey H_{C_2}^{RO(C_2)}(S^0)$-module on the following set of generators:

\begin{alignat*}{2}
	&\{ \cwt^k c_{\omega}^m &&\mid  0<k,\  0\leq m \} \\
	\union &\{ \cxwt^k c_{\chiw}^n &&\mid 0<k,\ 0 \leq n < q,\ 0<k+n \} \\
	\union &\{ c_{\omega}^m c_{\chiw}^n &&\mid  0\leq m,\ 0 \leq n \leq q \} \\
	\union &\{ \cxwt c_{\omega}^m c_{\chiw}^n &&\mid  1\leq m,\ 0 \leq n < q,\ n+1 \leq m \} \\
	\union &\{ \cwt c_{\omega}^m c_{\chiw}^n &&\mid 0 \leq m,\ 2 \leq n \leq q,\ m+1 < n \} \\
	\union &\{ \cwt^{-k} c_{\omega}^m c_{\chiw}^q &&\mid 1\leq k,\ 0 \leq m \}.
\end{alignat*}

Similarly, if $p < \infty$, then $\Mackey H_{C_2}^{RO(\Pi B)}(\Xpq{p}{\infty}_+)$ is generated, 
as an algebra over $\Mackey H_{C_2}^{RO(C_2)}(S^0)$,
 by the elements
$c_{\omega}$, $c_{\chiw}$, $\cwt$, and $\cxwt$,
together with the following elements: 
$c_{\omega}^p$ is infinitely divisible by $\cxwt$, meaning that, for $k\geq 1$,
there are unique elements $\cxwt^{-k}c_{\omega}^p$ such that
\[
	\cxwt^k \cdot \cxwt^{-k} c_{\omega}^p = c_{\omega}^p.
\]
The generators satisfy the following relations:
\begin{align*}
	\cwt c_{\chiw} - (1-\kappa)\cxwt c_{\omega} &= e^2 \qquad\text{and} \\
	\cxwt \cwt &= \xi.
\end{align*}
The element $\cxwt$ has the same cancellation property as given in Proposition \ref{prop: cancellation}:
\begin{align*}
 \cxwt x &= 0 \implies x = 0
 	&&\text{if }  n\geq p,\ |\alpha| \geq 2(p-n), \text{ and } \alpha^{C_2} \geq 2(p-n) \\
	&&&\text{or if }  n \leq p,\ |\alpha| \geq 4(p-n), \text{ and } \alpha^{C_2} \geq 2(p-n).
\end{align*}

We define
\begin{align*}
 \cxwt^{-k} c_{\omega}^m c_{\chiw}^n 
 	&:= (\cxwt^{-k}c_{\omega}^p) c_{\omega}^{m-p}c_{\chiw}^{n}
	&& m \geq p \text{ and } k \geq 1.
\end{align*}

Additively, $\Mackey H_{C_2}^{RO(\Pi B)}(\Xpq{p}{\infty}_+)$ is a free
$\Mackey H_{C_2}^{RO(C_2)}(S^0)$-module on the following set of generators:
\begin{alignat*}{2}
	&\{ \cwt^k c_{\omega}^m &&\mid  0<k,\ 0 \leq m < p \} \\
	\union &\{ \cxwt^k c_{\chiw}^n &&\mid 0<k,\ 0 \leq n,\  0<k+n \} \\
	\union &\{ c_{\omega}^m c_{\chiw}^n &&\mid 0 \leq m \leq p,\ 0\leq n \} \\
	\union &\{ \cxwt c_{\omega}^m c_{\chiw}^n &&\mid 1 \leq m \leq p,\ 0 \leq n,\ n+1 \leq m \} \\
	\union &\{ \cwt c_{\omega}^m c_{\chiw}^n &&\mid 0 \leq m < p,\ 2 \leq n,\ m+1 < n \} \\
	\union &\{ \cxwt^{-k} c_{\omega}^p c_{\chiw}^n &&\mid 1\leq k,\ 0 \leq n \}.
\end{alignat*}
\qed
\end{theorem}

Note that, in these cases, we do not have a vanishing result along the lines
of $c_{\omega}^p c_{\chiw}^q = 0$. However, we do still have relations like
\[
	c_{\omega}^{p+1} = e^2\cxwt^{-1}c_{\omega}^p + \xi \cxwt^{-2}c_{\omega}^p c_{\chiw}
\]
in $\Mackey H_{C_2}^{RO(\Pi B)}(\Xpq{p}{\infty}_+)$ that limit the powers we need to work with of one of the Euler classes.

\subsection{The $RO(C_2)$-graded subring}
\label{subsec: RO(G) part}
In \cite[Theorem~5.1]{LewisCP}, Lewis calculated the $RO(C_2)$-graded cohomology of
both the infinite and finite complex projective spaces. We here relate our results to his in the finite case.

Changing Lewis's notation slightly so as not to clash with ours, his result 
for finite projective spaces can be stated as follows.

\begin{theorem}[Lewis]
Assume that $q \leq p < \infty$. Then $\Mackey H_{C_2}^{RO(C_2)}(\Xpq{p}{q}_+)$ is generated
by an element $\gamma$ in grading $2\sigma$ and elements $\Gamma(k)$, $1\leq k < p$,
in gradings $2k + 2\min(k,q)\sigma$. These generators satisfy the relations
\begin{align*}
	\gamma^2 &= e^2\gamma + \xi\Gamma(1) \\
	\gamma \Gamma(k) &= \xi\Gamma(k+1) \qquad \text{for } k\geq q \\
	\Gamma(j)\Gamma(k) &=
		\begin{cases}
			\Gamma(j+k) & \text{for } j + k \leq q \\
			\sum_{i=0}^{\bar\jmath+\bar k-q} \binom{\bar\jmath+\bar k-q}{i} e^{2(\bar\jmath+\bar k-q-i)}\xi^i \Gamma(j+k+i)
				& \text{for } j + k > q
		\end{cases}
\end{align*}
where $\bar\jmath = \min(j,q)$, $\bar k = \min(k,q)$, and we interpret $\Gamma(n) = 0$ if $n \geq p$.
\qed
\end{theorem}

Proposition \ref{prop:additivestructure} gives us the following basis for $\Mackey H_{C_2}^{RO(C_2)}(\Xpq{p}{q}_+)$ when $p\geq q$:
If $p = q$, we have
\[
F_{p,q}(0)=\left\{\cxwt^\delta\big[\Xpq{p-(l+\delta)}{(q-l)}\big]_{\delta\omega+l(2+2\sigma)}^* \right\}_{\substack{\delta=0,1\\0 \leq l < q}} 
\]
which, using multiplicative generators, can be rewritten as
\[
 \{ 1, \cxwt c_{\omega}, c_{\omega}c_{\chiw}, \cxwt c_{\omega}^2 c_{\chiw},
 	c_{\omega}^2c_{\chiw}^2, \ldots, \cxwt c_{\omega}^{q} c_{\chiw}^{q-1} \}.
\]
If $p > q$, we have
\begin{multline*}
F_{p,q}(0)=\left\{\cxwt^\delta\big[\Xpq{p-(l+\delta)}{(q-l)}\big]_{\delta\omega+l(2+2\sigma)}^* \right\}_{\substack{\delta=0,1\\0 \leq l < q}} \\
		 \cup\left\{ \big[\Xq{(p-q-l)}\big]_{q(2+2\sigma)+2l}^* \right\}_{0 \leq l < p-q}
\end{multline*}
which, using multiplicative generators, can be expressed as
\begin{multline*}
 \{ 1, \cxwt c_{\omega}, c_{\omega}c_{\chiw}, \cxwt c_{\omega}^2 c_{\chiw},
 	c_{\omega}^2c_{\chiw}^2, \ldots, \cxwt c_{\omega}^{q} c_{\chiw}^{q-1}, \\
	c_{\omega}^q c_{\chiw}^q, \cwt^{-1} c_{\omega}^{q+1} c_{\chiw}^q,
	\cwt^{-2} c_{\omega}^{q+2} c_{\chiw}^q, \ldots,
	\cwt^{-(p-q-1)} c_{\omega}^{p-1}c_{\chiw}^q \}.
\end{multline*}
We can match up the two results by identifying
\begin{align*}
  \gamma &= \cxwt c_{\omega} \\
  \Gamma(k) &= 
  	\begin{cases}
		c_{\omega}^k c_{\chiw}^k & \text{for } 1\leq k < q \\
		\cwt^{-(k-q)} c_{\omega}^k c_{\chiw}^q & \text{for } q\leq k < p.
	\end{cases} \\
	&= \cwt^{-(k-\bar k)} c_{\omega}^k c_{\chiw}^{\bar k}.
\end{align*}
Lewis's relations can then be derived from ours:
\begin{align*}
 \gamma^2
 	&= (\cxwt c_{\omega})^2 \\
	&= \cxwt c_{\omega}(e^2 + (1-\kappa)\cwt c_{\chiw}) \\
	&= e^2\cxwt c_{\omega} + \xi c_{\omega} c_{\chiw} \\
	&= e^2\gamma + \xi\Gamma(1)
\end{align*}
using that $(1-\kappa)\xi = \xi$.
For $k\geq q$,
\begin{align*}
 \gamma\Gamma(k)
 	&= \cxwt c_{\omega} \cdot \cwt^{-(k-q)} c_{\omega}^k c_{\chiw}^q \\
	&= \cxwt \cwt \cwt^{-(k-q+1)} c_{\omega}^{k+1} c_{\chiw}^q \\
	&= \xi\Gamma(k+1).
\end{align*}
For $j + k \leq q$,
\begin{align*}
 \Gamma(j)\Gamma(k)
 	&= (c_{\omega}^j c_{\chiw}^j)(c_{\omega}^k c_{\chiw}^k) \\
	&= c_{\omega}^{j+k} c_{\chiw}^{j+k} \\
	&= \Gamma(j+k),
\end{align*}
but, if $j+k \geq q$,
\begin{align*}
 \Gamma(j)\Gamma(k)
 	&= (\cwt^{-(j-\bar\jmath)} c_{\omega}^j c_{\chiw}^{\bar\jmath})
		(\cwt^{-(k-\bar k)} c_{\omega}^k c_{\chiw}^{\bar k}) \\
	&= \cwt^{-(j+k-\bar\jmath-\bar k)} c_{\omega}^{j+k} c_{\chiw}^{\bar\jmath+\bar k} \\
	&= \sum_{i=0}^{\bar\jmath + \bar k - q} \binom{\bar\jmath + \bar k - q}{i}
		e^{2(\bar\jmath+\bar k-q-i)}\xi^i 
		\cwt^{-(j+k+i-q)} c_{\omega}^{j+k+i} c_{\chiw}^q \\
	&= \sum_{i=0}^{\bar\jmath + \bar k - q} \binom{\bar\jmath + \bar k - q}{i}
		e^{2(\bar\jmath+\bar k-q-i)}\xi^i \Gamma(j+k+i).
\end{align*}

Finally, in \cite[Remarks~5.3]{LewisCP}, Lewis pointed out that we can replace
$\gamma$ with $\tilde\gamma = (1-\kappa)\gamma + e^2$.
In our notation, we have
\[
 \tilde\gamma = (1-\kappa)\gamma + e^2 = (1-\kappa)\cxwt c_{\omega} + e^2 = \cwt c_{\chiw}.
\]
This reflects an alternate basis we could have chosen for the $RO(C_2)$-graded part:
If $p = q$ we could have used
\[
 \{ 1, \cwt c_{\chiw}, c_{\omega}c_{\chiw}, \cwt c_{\omega} c_{\chiw}^2,
 	c_{\omega}^2c_{\chiw}^2, \ldots, \cwt c_{\omega}^{q-1} c_{\chiw}^{q} \}
\]
while, if $p > q$, we could have used
\begin{multline*}
 \{ 1, \cwt c_{\chiw}, c_{\omega}c_{\chiw}, \cwt c_{\omega} c_{\chiw}^2,
 	c_{\omega}^2c_{\chiw}^2, \ldots, \cwt c_{\omega}^{q-1} c_{\chiw}^{q}, \\
	c_{\omega}^q c_{\chiw}^q, \cwt^{-1} c_{\omega}^{q+1} c_{\chiw}^q,
	\cwt^{-2} c_{\omega}^{q+2} c_{\chiw}^q, \ldots,
	\cwt^{-(p-q-1)} c_{\omega}^{p-1}c_{\chiw}^q \}.
\end{multline*}

\subsection{Other coefficient systems}
\label{subsec:other coefficients}

All of the above computations were done with coefficients in $\Mackey A$. 
The general result follows from the freeness of the cohomology we computed.

\begin{theorem}
\label{thm:allcoeffs}
If $\Mackey T$ is any Mackey functor, then multiplication induces isomorphisms
\begin{multline*}
  \Mackey H^{C_2}_{RO(\Pi B)}(\Xpq{p}{q}_+;\Mackey A) \tensorS \Mackey H^{C_2}_{RO(C_2)}(S^0;\Mackey T) \\
  \iso \Mackey H^{C_2}_{RO(\Pi B)}(\Xpq{p}{q}_+;\Mackey T)
\end{multline*}
and
\begin{multline*}
  \Mackey H_{C_2}^{RO(\Pi B)}(\Xpq{p}{q}_+;\Mackey A) \tensorS \Mackey H_{C_2}^{RO(C_2)}(S^0;\Mackey T) \\
  \iso \Mackey H_{C_2}^{RO(\Pi B)}(\Xpq{p}{q}_+;\Mackey T).
\end{multline*}
\qed
\end{theorem}
This is a direct corollary of \cite[Proposition~12.2]{Co:BGU1preprint}.
A similar result follows for $\Xpq{\infty}{q}$ and $\Xpq{p}{\infty}$.

Recall that $\Mackey \Z$ is the constant $\Z$ Mackey functor, which is also a ring
(Green functor).
There is an epimorphism $\Mackey A\to \Mackey \Z$ that induces an epimorphic ring map
$\Mackey H_{C_2}^{RO(C_2)}(S^0;\Mackey A) \onto \Mackey H_{C_2}^{RO(C_2)}(S^0;\Mackey \Z)$.
The kernel of this map is $\{ e^n\kappa \mid n\in\Z \}$, where we recall that
$e^n\kappa = 2e^n$ if $n>0$. From this we get the following.

\begin{corollary}
Theorem A is true after replacing $\Mackey A$ coefficients with
$\Mackey \Z$ coefficients throughout,
with the only change being that we have the relation
\[
 \cwt c_{\chiw} = \cxwt c_{\omega} + e^2
\]
in place of the relation
$\cwt c_{\chiw} = (1-\kappa)\cxwt c_{\omega} + e^2$.
In particular, we have that $\Mackey H_{C_2}^{RO(\Pi B)}(\Xpq{p}{q}_+;\Mackey \Z)$ is a free
module over $\Mackey H_{C_2}^{RO(C_2)}(S^0;\Mackey \Z)$.
\end{corollary}

\begin{proof}
Freeness is clear from Theorem~\ref{thm:allcoeffs}.
The map $\Mackey H_{C_2}^{RO(\Pi B)}(\Xpq{p}{q}_+;\Mackey A) \to \Mackey H_{C_2}^{RO(\Pi B)}(\Xpq{p}{q}_+;\Mackey \Z)$
is an epimorphism with kernel generated by the multiples of $e^n\kappa$, $n\in\Z$.
The only change that makes in the statement of the theorem is that we may replace $1-\kappa$ with $1$
because $\kappa = 0$ in $\Mackey H_{C_2}^{RO(C_2)}(S^0;\Mackey \Z)$.
\end{proof}


\section{Line bundles on $C_2$-equivariant finite projective spaces} \label{sec: Line bundles}
The aim of this section is to provide formulas for the Euler classes of equivariant line bundles over $\Xpq{p}{q}$ as well as for their direct sums. To achieve this goal we will first have to identify the $C_2$-equivariant Picard groups $\Pic^{C_2}(\Xpq{p}{q})$.

\subsection{The $C_2$-equivariant Picard group of finite projective spaces}

We begin with a technical lemma that will enable us to describe equivariantly trivial line bundles, followed by a result specific to $\Xpq{1}{}$.

\begin{lemma}\label{lem:onerepn}
Suppose that $\lambda$ is a complex $C_2$-line bundle over a $C_2$-space $X$ of the
$C_2$-homotopy type of a $C_2$-CW complex. Suppose further that $\lambda$ is nonequivariantly
trivial and that each fiber of $\lambda$ over $X^{C_2}$ is $C_2$-isomorphic to $\C$.
Then $\lambda \iso X\times\C$ as a complex $C_2$-vector bundle.\end{lemma}

\begin{proof}
We can assume that $X$ is a $C_2$-CW complex and we construct a trivialization by induction on
its skeleta $\{X^n\}$. There is a unique (up to homotopy) trivialization on the set of vertices $X^0$, so assume now
that $n\geq 1$ and that we have a trivialization of $\lambda|X^{n-1}$.
Consider an $n$-cell $C_2/H\times D^n$ attached along a map $\alpha\colon C_2/H\times S^{n-1}\to X^{n-1}$,
where $H = C_2$ or the trivial subgroup.
Using, in the case $H = C_2$, that the fibers over the fixed sets are isomorphic to $\C$, we have that
$\lambda|C_2/H\times D^n \iso C_2/H\times D^n\times\C$. By induction, we have that
$\alpha^*(\lambda|X^{n-1}) \iso C_2/H\times S^{n-1}\times\C$.
The bundle $\lambda$ then determines a $C_2$-clutching function $\bar\lambda\colon C_2/H\times S^{n-1} \to S^1$,
where $S^1$ is the space of orthogonal maps from $\C$ to itself, with $C_2$ acting on said maps by conjugation,
that is, trivially.
By our assumption that $\lambda$ is nonequivariantly trivial, we have that
$\bar\lambda|eH\times S^{n-1}$ is nonequivariantly homotopically trivial. It follows that $\bar\lambda$
is $C_2$-homotopically trivial, hence we may extend the trivialization of $\lambda|X^{n-1}$ over this $n$-cell.
By induction, we may construct a $C_2$-trivialization of $\lambda$ on all of $X$.
\end{proof}

\begin{lemma}\label{lem:mixed}
Suppose that $\lambda$ is a complex $C_2$-line bundle over $\Xpq{1}{}$ with
its restriction to one of the fixed points being $\C$ and its restriction to the other
being $\C^\sigma$. Then, nonequivariantly, $\lambda \iso \omega^{\tensor n}$ for some odd $n$.
\end{lemma}

\begin{proof}
Write
\[
 \Xpq{1}{} = D(\C^\sigma) \union_{S(\C^\sigma)} D(\C^\sigma) = D_0 \union D_1,
\]
and suppose that $\lambda|D_0 \iso D_0\times \C$ while $\lambda|D_1 \iso D_1\times\C^\sigma$.
Then $\lambda$ is determined by its clutching function, which will be a $C_2$-map from 
the equator $S(\C^\sigma)$
to the space of orthogonal maps $\C\to \C^\sigma$ with $C_2$ acting by conjugation.
This space of maps is another copy of $S(\C^\sigma)$, so the clutching function is a $C_2$-map
$S(\C^\sigma)\to S(\C^\sigma)$. Because $C_2$ acts by negation on the circle $S(\C^\sigma)$,
any such map must have odd degree as a nonequivariant map, and we have, nonequivariantly,
that $\lambda \iso \omega^{\tensor n}$ where $n$ is the degree of the clutching function.
\end{proof}

We are now in the position to determine the $C_2$-equivariant Picard group of projective spaces.

\begin{theorem}
If $p+q\geq 2$, then
$\Pic^{C_2}(\Xpq{p}{q}) \iso \Z\times \Z/2$.
Writing the operation as multiplication, the group is generated by $\omega$ and $\chiw$
with $(\chiw)^2 = \omega^2$.
Alternatively, it is generated by $\omega$ and $\C^\sigma$, with $(\C^\sigma)^2 = \C = 1$
and $\omega\cdot\C^\sigma = \chiw$.
\end{theorem}

\begin{proof}
The cases where $p=0$ or $q=0$ are simple, so we assume that $p\geq 1$ and $q\geq 1$.
We know that we have a short exact sequence
\[
 0 \to \ker u \to \Pic^{C_2}(\Xpq{p}{q}) \xrightarrow{u} \Pic(\Xpq{p}{q}) \to 0
\]
where $u$ is the forgetful map.
Suppose $\lambda\in \ker u$, so $\lambda$ is a $C_2$-line bundle over $\Xpq{p}{q}$
that is nonequivariantly trivial. Restrict to the subspace $\Xpq{1}{}$.
Then Lemma~\ref{lem:mixed} implies that the restrictions of $\lambda$ to the two fixed points must
give the same representations of $C_2$,
hence the fibers of $\lambda$ on $\Xpq{p}{q}^{C_2}$ are all the same representation.

If that common representation is $\C$, then it follows immediately from
Lemma~\ref{lem:onerepn} that $\lambda \iso \C$.
If the common representation is $\C^\sigma$, then apply Lemma~\ref{lem:onerepn}
to $\lambda\tensor\C^{\sigma}$ to get $\lambda\tensor\C^\sigma\iso \C$,
hence
\[
 \lambda \iso (\lambda\tensor\C^\sigma)\tensor\C^\sigma \iso \C^\sigma.
\]
Hence $\ker u \iso \Z/2 = \{\C,\C^\sigma\}$.
The theorem now follows from the nonequivariant calculation of $\Pic(\Xpq{p}{q}) \iso \Z$,
generated by $\omega$.
\end{proof}

Since every equivariant bundle gives rise to a representation of the fundamental groupoid, it is possible to subdivide the elements of $\Pic^{C_2}(\Xpq{p}{q})$  according to the element of $RO(\Pi B)$ they represent. It is easy to see that only four elements of $RO(\Pi B)$ can arise this way: $\omega^*$, $\C$, $\chiw^*$  and $\C^\sigma$. This suggests the following subdivision.

\begin{definition}
Let $L\in \Pic^{C_2}(\Xpq{p}{q})$. We will say that $L$ is of 
\begin{itemize}
    \item type $\I$ if $L\cong \omega^{\otimes n}$ for $n\in \ZZ$ odd;
    \item type $\II$ if $L\cong \omega^{\otimes n}$ for $n\in \ZZ$  even;
    \item type $\III$ if $L\cong \chiw^{\otimes n}$ for $n\in \ZZ$ odd;
    \item type $\IV$ if $L\cong \chiw^{\otimes n}$ for $n\in \ZZ$ even.
\end{itemize}
We will write $\Pic_\dagger^{C_2}(\Xpq{p}{q})$ for the collection of line bundles of type $\dagger$, for $\dagger\in \{\I,\II,\III,\IV\}$.
\end{definition}

\subsection{The  Euler class of $O(n)$}

We now adopt the convention from algebraic geometry of writing $O(n)$ for $(\omega\dual)^{\tensor n}$,
the $n$th tensor power of the dual of the tautological bundle $\omega$.
We write $\chi O(n) = O(n)\tensor\C^\sigma$.
We also switch to using the generators $\cwd$ and $\cxwd$, which will simplify some of the formulas.

In the following calculation, we use the fact, shown in \cite{Co:BGU1preprint}, that the map
\begin{multline*}
 \eta\colon \Mackey H_{C_2}^{RO(\Pi B)}(\Xpq{\infty}{\infty}_+)
 	\to \Mackey H_{C_2}^{RO(\Pi B)}(\Xpq{\infty}{\infty}_+^{C_2}) \\
 \iso \Mackey H_{C_2}^{RO(C_2)}(S^0)[c,\zeta_1^{\pm 1}]
    \dirsum \Mackey H_{C_2}^{RO(C_2)}(S^0)[c,\zeta_0^{\pm 1}]
\end{multline*}
induced by the inclusion of the fixed points
is injective in even gradings, that is, gradings $\alpha$ such that both $|\alpha|$ and $\alpha^{C_2}$ are even.
Here, $c\in H^2(\Xp\infty_+)$ is the Euler class of the dual bundle $\omega\dual$,
$\deg\zeta_1 = \omega-2$, and $\deg\zeta_0 = \chi\omega-2$.

\begin{proposition}\label{prop:EulerClasses}
In the cohomology of the infinite projective space $\Xpq\infty\infty$ we have, for $n\in\Z$,
\begin{align*}
 e(O(2n)) &= n\cwd(\tau(\iota^{-2})\cxwt + e^{-2}\kappa\cxwd) \\
 e(O(2n+1)) &= \cwd(1 + n\tau(1) + ne^{-2}\kappa\cwt\cxwd) \\
 e(\chi O(2n)) &= n\tau(1)\cwt\cxwd + e^2 \\
 e(\chi O(2n+1)) &= \cxwd(1 + n\tau(1) + ne^{-2}\kappa\cxwt\cwd).
\end{align*}
\end{proposition}

\begin{proof}
Recall the map $\eta$ mentioned before the proposition.
Because these elements live in gradings where $\eta$ is injective,
it suffices to show equality after applying $\eta$ to both sides.

It is shown in \cite{Co:BGU1preprint} that
\begin{align*}
 \eta(e(O(1))) = \eta(\cwd) &= (c\zeta_1, (\xi c + e^2)\zeta_0^{-1}) \\
 \eta(e(\chi O(1))) = \eta(\cxwd) &= ((\xi c + e^2)\zeta_1^{-1}, c\zeta_0) \\
 \eta(\cwt) &= (\zeta_1, \xi\zeta_0^{-1}) \\
 \eta(\cxwt) &= (\xi\zeta_1^{-1}, \zeta_0)
\end{align*}
An argument similar to the one that computed the Euler classes of $O(1)$ and $\chi O(1)$ above shows that
\begin{align*}
 \eta(e(O(2n))) &= (2nc, 2nc) \\
 \eta(e(O(2n+1))) &= ((2n+1)c\zeta_1, ((2n+1)\xi c + e^2)\zeta_0^{-1}) \\
 \eta(e(\chi O(2n))) &= (2n\xi c + e^2, 2n\xi c + e^2) \\
 \eta(e(\chi O(2n+1))) &= (((2n+1)\xi c + e^2)\zeta_1^{-1}, (2n+1)c\zeta_0).
\end{align*}
These facts and the fact that $\eta$ is a map of algebras over $\Mackey H_{C_2}^{RO(C_2)}(S^0)$
now allows us to apply $\eta$ to verify the equations in the proposition.
\end{proof}

\subsection{The Euler class of a sum of line bundles}

We will consider the following scenario in the next section.
Suppose that we have a list $\F$ of line bundles over $\Xpq pq$ which we subdivide into four families $\F_\I$, $\F_\II$, $\F_\III$, and $\F_\IV$,
where each $\F_\dagger$ consists of the bundles of type $\dagger$.
Let $n_\dagger$ be the number of bundles in $\F_\dagger$ and let
$d_\dagger$ denote the products of the degrees of the line bundles in $\F_\dagger$.
Finally, set $n = n_\I+n_\II+n_\III+n_\IV$ and $d = d_\I d_\II d_\III d_\IV$.
In this section we will compute the Euler class of the sum of all of these line bundles.

Let $F_\dagger$ be the sum of the line bundles in $\F_\dagger$. 

\begin{lemma}
$e(F_\II) = \frac12 d_\II \cwd^{\;n_\II}(\tau(\iota^{-2n_\II})\cxwt^{n_\II} + e^{-2n_\II}\kappa\cxwd^{\;n_\II})$.
\end{lemma}

\begin{proof}
If $\F_\II = \{ O(d_1), \ldots, O(d_{n_\II}) \}$, where the $d_i$ are all even, then
\[
 e(F_\II) = \prod_{i=1}^{n_\II} \frac{d_i}{2} \cwd(\tau(\iota^{-2})\cxwt + e^{-2}\kappa\cxwd)
\]
by Proposition~\ref{prop:EulerClasses}. We compute
\[
 (\tau(\iota^{-2})\cxwt + e^{-2}\kappa\cxwd)^{n_\II}
  = 2^{n_\II-1}(\tau(\iota^{-2n_\II})\cxwt^{n_\II} + e^{-2n_\II}\kappa\cxwd^{\;n_\II})
\]
using that $\tau(\iota^{m})\tau(\iota^{n}) = 2\tau(\iota^{m+n})$,
$e^m\kappa \cdot e^n\kappa = 2e^{m+n}\kappa$, and $\tau(\iota^m)e^{n}\kappa = 0$.
The lemma follows.
\end{proof}

\begin{lemma}
$e(F_\IV) = \frac12 (d_\IV\tau(1)\cwt^{n_\IV}\cxwd^{\;n_\IV} + e^{2n_\IV}\kappa)$.
\end{lemma}

\begin{proof}
If $\F_\IV =\{ \chi O(d_1), \ldots, \chi O(d_{n_\IV}) \}$, where the $d_i$ are all even, then
\[
 e(F_\IV) = \prod_{i=1}^{n_\IV} \left(\frac{d_i}{2}\tau(1)\cwt\cxwd + e^2 \right)
\]
by Proposition~\ref{prop:EulerClasses}. 
Using that
$\tau(1)^2 = 2\tau(1)$, $\tau(1) e = 0$, and $e=\frac{1}{2}e\kappa$ we get
\begin{align*}
 e(F_\IV) &= \frac12 d_\IV\tau(1)\cwt^{n_\IV}\cxwd^{\;n_\IV} + e^{2n_\IV} \\
  &= \frac12 (d_\IV\tau(1)\cwt^{n_\IV}\cxwd^{\;n_\IV} + e^{2n_\IV}\kappa)
\end{align*}
if $n_\IV > 0$, but this formula also works when $n_\IV = 0$ and $d_\IV = 1$, where it reduces to 1.
\end{proof}

In order to calculate the Euler classes of $F_\I$ and $F_\III$ we need the following algebraic observation.

\begin{lemma}\label{lemma:2x}
Consider the ring $\Z[x]/\langle x^2-2x \rangle$. If $a_1$, \dots, $a_\ell$ are integers, then
\[
 \prod_{j=1}^\ell (1 + a_j x) = 1 + \tfrac{1}{2}\Bigl[\prod_j (1+2a_j) - 1\Bigr]x.
\]
\end{lemma}

\begin{proof}
Expand $\prod_j (1+a_j x) = \sum_n b_n x^n$ in $\Z[x]$. When we reduce modulo the relation $x^2 = 2x$,
the coefficient of $x$ will be $b_1 + 2b_2 + 4b_3 + \cdots$. Substituting $2$ for $x$ in the expansion, we get
\[
 \prod_j (1 + 2a_j) = 1 + 2b_1 + 4b_2 + 8b_3 + \cdots,
\]
from which the computation now follows.
\end{proof}

We will use this computation in the following form.

\begin{corollary}\label{cor:2x}
If $a_1$, \dots, $a_\ell$ are odd integers, then, in $\Z[x]/\langle x^2-2x \rangle$, we have
\[
 \prod_{j=1}^\ell \left(1 + \frac{a_j - 1}{2}x \right) = 1 + \frac{\prod_j a_j - 1}{2} x.
\]
\qed
\end{corollary}

\begin{lemma}
$\displaystyle e(F_\I) = \cwd^{\;n_\I}\left(1 + \frac{d_\I-1}{2}\tau(1) + \frac{d_\I-1}{2}e^{-2}\kappa\cwt\cxwd\right)$
\end{lemma}

\begin{proof}
If $\F_\I =\{ O(d_1), \ldots,  O(d_{n_\I}) \}$, where the $d_i$ are all odd, then
\[
 e(F_\I) = \prod_{i=1}^{n_\I} \cwd\left(1 + \frac{d_i-1}{2}\tau(1) + \frac{d_i-1}{2}e^{-2}\kappa\cwt\cxwd\right)
\]
by Proposition~\ref{prop:EulerClasses}. We have that
\[
 (\tau(1) + e^{-2}\kappa\cwt\cxwd)^2 = 2(\tau(1) + e^{-2}\kappa\cwt\cxwd),
\]
so the lemma follows from Corollary~\ref{cor:2x}.
\end{proof}

\begin{lemma}
$\displaystyle e(F_\III) = \cxwd^{\;n_\III}\left(1 + \frac{d_\III-1}{2}\tau(1) + \frac{d_\III-1}{2}e^{-2}\kappa\cxwt\cwd\right)$
\end{lemma}

\begin{proof}
This is proved in the same way as the previous lemma.
\end{proof}

We can now put these pieces together to get the computation we want.
As we did in introducing Lemma~\ref{lem:pushforwardfree},
we write $\zeta = \rho(\zeta_1) \in H^{\omega-2}(\Xp\infty_+)$ for the restriction of $\zeta_1$
and $c\in H^2(\Xp\infty_+)$ for the Euler class of the dual of the tautological bundle. Remember that we have
$\rho(\zeta_0) = \iota^2\zeta^{-1}$, $\rho(\cwd) = \zeta c$, and $\rho(\cxwd) = \iota^2\zeta^{-1}c$.

\begin{proposition}\label{prop:totalEulerClass}
If $F$ is the sum of all the line bundles in the collections
$\F_\I$, $\F_\II$, $\F_\III$, and $\F_\IV$, then
\begin{align*}
 e(F) &= \frac{d}{2} \tau(\iota^{2(n_\III+n_\IV)}\zeta^{n_\I-n_\III} c^n) \\
  &\quad + \frac{ d_\I d_\II}{2} e^{-2(n_\II-n_\IV+1)}\kappa\cwt\cwd^{\;n_\I+n_\II}\cxwd^{\;n_\II+n_\III+1} \\
  &\quad + \frac{d_\II d_\III}{2} e^{-2(n_\II-n_\IV+1)}\kappa\cxwt\cwd^{\;n_\I+n_\II+1}\cxwd^{\;n_\II+n_\III}.
\end{align*}
\end{proposition}

\begin{proof}
We need to compute the product of the Euler classes $e(F_\dagger)$ that we computed above.
It is simplest to start with the following pairs.
\begin{align*}
 e(F_\II\dirsum F_\IV)
  &= \frac{d_\II}{2} \cwd^{\;n_\II}\left(\tau(\iota^{-2n_\II})\cxwt^{n_\II} + e^{-2n_\II}\kappa\cxwd^{\;n_\II}\right)
     \cdot \frac12 (d_\IV\tau(1)\cwt^{n_\IV}\cxwd^{\;n_\IV} + e^{2n_\IV}\kappa) \\
  &= \frac{d_\II}{2} \cwd^{\;n_\II} \left(d_\IV\tau(\iota^{-2n_\II})\cxwt^{n_\II}\cwt^{n_\IV}\cxwd^{\;n_\IV}
  					+ e^{-2(n_\II-n_\IV)}\kappa\cxwd^{\;n_\II}\right)
\end{align*}
using the fact that $\tau(\iota^m)\tau(\iota^n) = 2\tau(\iota^{m+n})$ and that
$\tau(\iota^m)\kappa = 0$.
\begin{align*}
 e(F_\I\dirsum F_\III)
  &= \cwd^{\;n_\I}\left(1 + \frac{d_\I-1}{2}\tau(1) + \frac{d_\I-1}{2}e^{-2}\kappa\cwt\cxwd\right) \\
  &\qquad \cdot \cxwd^{\;n_\III}\left(1 + \frac{d_\III-1}{2}\tau(1) + \frac{d_\III-1}{2}e^{-2}\kappa\cxwt\cwd\right) \\
  &= \cwd^{\;n_\I}\cxwd^{\;n_\III}\left(1 + \frac{d_\I d_\III-1}{2}\tau(1) \right. \\
  &\qquad\qquad\qquad \left. {}+ \frac{d_\I-1}{2}e^{-2}\kappa\cwt\cxwd + \frac{d_\III-1}{2}e^{-2}\kappa\cxwt\cwd \right)
\end{align*}
This uses previously noted vanishings as well as 
$e^{-2}\kappa\cwt\cdot e^{-2}\kappa\cxwt = 2\xi e^{-4}\kappa = 0$.
Putting these computations together we get
\begin{align*}
e(F) 
  &= \frac{d_\II}{2} \cwd^{\;n_\II}\left(d_\IV\tau(\iota^{-2n_\II})\cxwt^{n_\II}\cwt^{n_\IV}\cxwd^{\;n_\IV}
  					+ e^{-2(n_\II-n_\IV)}\kappa\cxwd^{\;n_\II}\right) \\
  & \quad \cdot \cwd^{\;n_\I}\cxwd^{\;n_\III}\left(1 + \frac{d_\I d_\III-1}{2}\tau(1) 
  		+ \frac{d_\I-1}{2}e^{-2}\kappa\cwt\cxwd + \frac{d_\III-1}{2}e^{-2}\kappa\cxwt\cwd \right) \\
  &= \frac{d_\II}{2}\cwd^{\;n_\I+n_\II}\cxwd^{\;n_\III}
  		\Bigl( d_\I d_\III d_\IV\tau(\iota^{-2n_\II})\cxwt^{n_\II}\cwt^{n_\IV}\cxwd^{\;n_\IV}
				+ e^{-2(n_\II-n_\IV)}\kappa\cxwd^{\;n_\II} \\
  &\qquad + (d_\III-1)e^{-2(n_\II-n_\IV+1)}\kappa\cwt\cxwd^{n_\II+1}
				+ (d_\III-1)e^{-2(n_\II-n_\IV+1)}\kappa\cxwt\cwd\cxwd^{\;n_\II} \Bigr) \\
  &= \frac{d_\II}{2}\cwd^{\;n_\I+n_\II}\cxwd^{\;n_\III}
  		\Bigl( d_\I d_\III d_\IV\tau(\iota^{-2n_\II})\cxwt^{n_\II}\cwt^{n_\IV}\cxwd^{\;n_\IV} \\
  &\qquad\qquad + d_\I e^{-2(n_\II-n_\IV+1)}\kappa\cwt\cxwd^{n_\II+1}
				+ d_\III e^{-2(n_\II-n_\IV+1)}\kappa\cxwt\cwd\cxwd^{\;n_\II} \\
  &\qquad\qquad + e^{-2(n_\II-n_\IV+1)}\kappa\cxwd^{\;n_\II}(e^2 - \cwt\cxwd - \cxwt\cwd) \Bigr) \\
  &= \frac{d_\II}{2}\cwd^{\;n_\I+n_\II}\cxwd^{\;n_\III}
  		\Bigl( d_\I d_\III d_\IV\tau(\iota^{-2n_\II})\cxwt^{n_\II}\cwt^{n_\IV}\cxwd^{\;n_\IV} \\
  &\qquad\qquad + d_\I e^{-2(n_\II-n_\IV+1)}\kappa\cwt\cxwd^{n_\II+1}
				+ d_\III e^{-2(n_\II-n_\IV+1)}\kappa\cxwt\cwd\cxwd^{\;n_\II} \\
  &\qquad\qquad + e^{-2(n_\II-n_\IV+1)}\kappa\cxwd^{\;n_\II}\cdot \tau(1)\cxwt\cwd \Bigr) \\
  &= \frac{d_\II}{2}\cwd^{\;n_\I+n_\II}\cxwd^{\;n_\III}
  		\Bigl( d_\I d_\III d_\IV\tau(\iota^{-2n_\II})\cxwt^{n_\II}\cwt^{n_\IV}\cxwd^{\;n_\IV} \\
  &\qquad\qquad + d_\I e^{-2(n_\II-n_\IV+1)}\kappa\cwt\cxwd^{n_\II+1}
				+ d_\III e^{-2(n_\II-n_\IV+1)}\kappa\cxwt\cwd\cxwd^{\;n_\II} \Bigr).
\end{align*}
Finally, we use that
\begin{multline*}
 \tau(\iota^{-2n_\II})\cxwt^{n_\II}\cwt^{n_\IV}\cwd^{\;n_\I+n_\II}\cxwd^{\;n_\III + n_\IV} \\
  = \tau(\iota^{-2n_\II}\iota^{2n_\II}\zeta^{-n_\II}\zeta^{n_\IV}\zeta^{n_\I+n_\II}c^{n_\I+n_\II}
           \iota^{2(n_\III+n_\IV)}\zeta^{-(n_\III+n_\IV)}c^{n_\III+n_\IV}) \\
  = \tau(\iota^{2(n_\III+n_\IV)}\zeta^{n_\I-n_\III}c^n)
\end{multline*}
to get the result of the proposition.
\end{proof}

\section{Equivariant Bezout's theorem}\label{sec: Application}

We now consider the following scenario, which we began to discuss in the preceding section. We consider a finite list $\F$ of line bundles over $\Xpq{p}{q}$, which we subdivide into four families $\F_\I$, $\F_\II$, $\F_\III$, and $\F_\IV$,
where each $\F_\dagger$ consists of the bundles of type $\dagger$.
Let $n_\dagger$ be the number of bundles in $\F_\dagger$ and let
$d_\dagger$ denote the products of the degrees of the line bundles in $\F_\dagger$.
Set  $n = n_\I+n_\II+n_\III+n_\IV$ and $d = d_\I d_\II d_\III d_\IV$; we shall assume that $n = p+q-1$.
Also set
\[
 \alpha = n_\I + n_\II - (p-1) \quad\text{and}\quad \beta = n_\II + n_\III - (q-1).
\]
Suppose that, for each bundle $L_i\in\F$ we have an algebraic equivariant section $s_i$ whose zero locus $H_i$ is a
$C_2$-invariant hypersurface in $\Xpq pq$.
We want to investigate the intersection of these hypersurfaces, which is the zero locus of the direct sum of the corresponding sections. The quantities $-\alpha$ and $-\beta$ can be thought of as the ``expected codimensions'' of the intersections 
\[
\Big(\bigcap_i H_i \Big)\cap \Xp{p} \text{ and } \Big(\bigcap_i H_i \Big) \cap \Xq{q},
\]
respectively.

We shall want the hypersurfaces to be at least nonequivariantly dimensionally transverse, so that their intersection is a discrete set of points
(given that $n = p+q-1$), hence a $C_2$-set. We first notice the following constraint.

\begin{proposition}
If the $p+q-1$ hypersurfaces $H_i$ are (nonequivariantly) dimensionally transverse, then $\alpha\geq 0$ and $\beta\geq 0$.
\end{proposition}

\begin{proof}
Suppose that $\alpha < 0$, so $n_\I + n_\II < p-1$. Consider the restrictions of the sections $s_i$ to the
fixed set $\Xp p \subset \Xpq pq$.
The fibers of any line bundle of type $\III$ or $\IV$ over $\Xp p$ have the form $\C^\sigma$, so any section of such a line
bundle must be zero on $\Xp p$. That leaves only $n_\I + n_\II < p-1$ sections that could possibly be nonzero on $\Xp p$,
but the intersection of the corresponding hypersurfaces must then have codimension less than $p-1$, hence cannot consist of
a discrete set of points. Therefore we must have $\alpha\geq 0$. The argument that $\beta\geq 0$ is similar.
\end{proof}

So from now on we shall assume $\alpha\geq 0$ and $\beta\geq 0$.

Having used the term ``dimensionally transverse,'' a word about the terminology we shall use.
When we write simply {\em transverse} we will mean an equivariant map whose underlying nonequivariant map is transverse in the sense of
differential topology, which is precisely the notion of equivariant transversality. 
There are the weaker notions of (equivariantly) {\em dimensionally transverse}, to be defined below, and {\em nonequivariantly dimensionally transverse},
the latter meaning that the intersection has the expected (co)dimension.

Write $F$ for the sum of the line bundles in $\F$. We start with some purely topological results.

\begin{proposition}\label{prop:transversality}
If $\alpha\geq 0$ and $\beta\geq 0$, then we have the following.
\begin{enumerate}
    \item There exists an equivariant topological section of $F$ that is transverse to the zero section.
    \item If $s$ is an equivariant section of $F$ whose zero locus consists of a discrete set of points, then there exists an equivariant section $s'$ transverse to the zero section and agreeing with $s$ outside of an arbitrarily small neighborhood of the zeros of $s$.
\end{enumerate}
\end{proposition}

\begin{proof}
For part (1) we cannot directly apply the equivariant transversality results of Wasserman \cite{Was:difftopology}, but we can adapt
the line of argument given there or in \cite{Haus:Transversality}, as follows.
By nonequivariant transversality,
there exists a section $s\colon \Xpq pq^{C_2}\to F^{C_2}$
transverse to the zero section.
Because of our assumptions, $s^{-1}(0)$  consists of isolated points (or is empty).
Let $x\in \Xp p$ be a point such that $s(x) = 0$; for dimensional reasons such
a point can exist only if $\alpha = 0$, that is, $n_\I + n_\II = p-1$. In that case, we have
\[
 \tau_x \iso \C^{p-1}\dirsum \C^{q\sigma} \iso F_x,
\]
where $\tau_x$ is the tangent space to $x$ in $\Xpq pq$.
By transversality, $s$ induces an isomorphism of $\tau_x^{C_2}$ with $F_x^{C_2}$.
Writing $V_{C_2}$ for the orthogonal complement of $V^{C_2}$ in any $C_2$-representation $V$,
take an isomorphism of $(\tau_x)_{C_2}$ with $(F_x)_{C_2}$ and use that to extend $s$ equivariantly
and transversely
to a small open neighborhood of $x$ in $\Xpq pq$. 
Do the same for all such $x$, and similarly for any zeros of $s$ that lie in $\Xq q$.
We can then extend further to a finite collection of small balls around other points in $\Xpq pq^{C_2}$
(staying away from the zero section) and then shrink to obtain a tubular open neighborhood 
$U$ of $\Xpq pq^{C_2}$ with a transverse equivariant section $s$ defined on a neighborhood of $\bar U$.

Now consider $M = \Xpq pq \setminus U$, a free $C_2$-manifold with boundary
$\bndry M = \bndry\bar U$. We have $s$ defined on a neighborhood of $\bndry M$ and seek to extend it over all of $M$. 
Let $z$ be the zero section of $F|M$ and choose a homotopy of $z$ with a section $z'$ agreeing with $s$ on a neighborhood of $\bndry M$.
Because $C_2$ acts freely, $M/C_2$ is again a manifold and $s$ induces a section
$\bar s = s/C_2$ of the bundle $(F|M)/C_2$ on a neighborhood of $\bndry M/C_2$. By nonequivariant transversality, we may
extend $\bar s$ to a section $\bar s'$ on all of $M/C_2$ that is transverse to the zero section
of $(F|M)/C_2$ and agrees with $\bar s$ in a neighborhood of $\bndry M/C_2$. 
If $\pi\colon M\to M/C_2$ is the projection, consider $\bar z'\circ\pi$ and
$\bar s'\circ\pi$, where $\bar z' = z'/C_2$. Take a $C_2$-homotopy
from $\bar z'\circ\pi$ to $\bar s'\circ\pi$ and lift it along the fibration $F|M \to (F|M)/C_2$,
starting at $z'$,
to obtain a homotopy from $z'$ to an equivariant section $s'$ agreeing with $s$ on a neighborhood of $\bndry M$.
Gluing this to the section $s$ on $\bar U$, we obtain the desired equivariant section
transverse to the zero section.

The argument for part (2) is simpler. If $x$ is a $C_2$-fixed isolated zero of $s$, take an arbitrarily small
disc neighborhood $U$ of $x$, not containing any other zeros of $s$ and over which $F$ is trivial, As argued above, we must have
$\tau_x \iso F_x$ as $C_2$-representations. Wasserman's transversality results then allow us to find a homotopy of $s$ to
a section $s'$ agreeing with $s$ outside of $U$ and transverse to the zero section inside $U$.
At any free orbit in the zero section of $s$ we can make a similar argument using nonequivariant transversality.
Doing this at all points in the zero section of $s$ gives us a transverse section $s'$ as in the proposition.
\end{proof}

Proposition~\ref{prop:transversality}(1) reassures us that the following result is not vacuous.

\begin{theorem}\label{thm:genericallyTransverse}
Let $F$ be as above and suppose that $s$ is an equivariant topological section of $F$ transverse to the zero section.
Let $Z(s)$ be its zero locus.
Let $\eta$ be the representation of $\Pi B$ induced by $F$, so
\[
 \eta = (n_\I-n_\III)\omega + 2(n_\II + n_\III) + 2(n_\III+n_\IV)\sigma.
\]
Then
\[
 [Z(s)]^*_\eta = \frac{d-A-B}{2}[C_2]^*_\eta + A[\Xp{}]^*_\eta + B[\Xq{}]^*_\eta
\]
where the pair of integers $(A,B)$ is given by
\[
(A,B)=
\begin{cases}
 \left(d_{\I}d_{\II},d_{\II}d_{\III} \right)  \quad& \text{ for }\alpha=0,\ \beta=0 \\
 \ \ \left(d_{\I}d_{\II},0 \right)  \quad& \text{ for }\alpha=0,\ \beta>0 \\
 \ \ \left(0,d_{\II}d_{\III} \right)  \quad &\text{ for }\alpha>0,\ \beta=0\\
 \ \ \ \ (0,0)   \quad& \text{ for }\alpha>0,\ \beta>0.
\end{cases}
\]
\end{theorem}

\begin{proof}
Given that $s$ is transverse to the zero section, we know that $[Z(s)]^*_\eta = e(F)$. In Proposition~\ref{prop:totalEulerClass} we
calculated that
\begin{align*}
 e(F) &= \frac{d}{2} \tau(\iota^{2(n_\III+n_\IV)}\zeta^{n_\I-n_\III} c^n) \\
  &\qquad + \frac{d_\I d_\II}{2} e^{-2(n_\II-n_\IV+1)}\kappa\cwt\cwd^{\;n_\I+n_\II}\cxwd^{\;n_\II+n_\III+1} \\
  &\qquad + \frac{d_\II d_\III}{2} e^{-2(n_\II-n_\IV+1)}\kappa\cxwt\cwd^{\;n_\I+n_\II+1}\cxwd^{\;n_\II+n_\III} \\
  &= \frac{d}{2} [C_2]^*_\eta + \frac{d_\I d_\II}{2} e^{-2(\alpha+\beta)}\kappa\cwt\cwd^{\;p-1 + \alpha}\cxwd^{\;q+\beta} \\
  &\qquad + \frac{d_\II d_\III}{2} e^{-2(\alpha+\beta)}\kappa\cxwt\cwd^{\;p+\alpha}\cxwd^{\;q-1 + \beta}
\end{align*}
where we use Lemma~\ref{lem:pushforwardfree} to identify the first term.
Consider the second term. If $\alpha > 0$, then the term is a multiple of $\cwd^{\;p}\cxwd^{\;q} = 0$; these are the cases where $A = 0$
in the statement of the theorem. So suppose that $\alpha = 0$. Then the second term is
\[
    \frac{d_\I d_\II}{2} e^{-2\beta}\kappa\cwt\cwd^{\;p-1}\cxwd^{\;q+\beta}
    = \frac{d_\I d_\II}{2}\kappa\cwt^{-\beta+1}\cwd^{\;p-1}\cxwd^{\;q},
\]
which follows by induction, using the relation $\cwt\cxwd = (1-\kappa)\cxwt\cwd + e^2$, the fact that
$\kappa\cxwt\cwt = \kappa\xi = 0$, and the fact that $\cxwd^{\;q}$ is divisible by $\cwt$. Taking this a step further,
we can rewrite the term as
\begin{align*}
 \frac{d_\I d_\II}{2}\kappa\cwt^{-\beta+1}\cwd^{\;p-1}\cxwd^{\;q}
  &= \frac{d_\I d_\II}{2}(2-\tau(1))\cwt^{-\beta+1}\cwd^{\;p-1}\cxwd^{\;q} \\
  &= d_\I d_\II [\Xp{}]^*_\eta - \frac{d_\I d_\II}{2}[C_2]^*_\eta.
\end{align*}
A similar analysis shows that the third term is $0$ if $\beta > 0$ and, if $\beta = 0$, we have
\[
    \frac{d_\II d_\III}{2} e^{-2\alpha}\kappa\cxwt\cwd^{\;p+\alpha}\cxwd^{\;q-1}
    = d_\II d_\III [\Xq{}]^*_\eta - \frac{d_\II d_\III}{2}[C_2]^*_\eta.
\]
The theorem follows.
\end{proof}

Note that there is an abuse of notation in the statement of the theorem: If $\alpha>0$ then $\eta$ is not a valid
grading for the element $[\Xp p]^*$, and we really mean that the term does not appear in the answer. Similarly
for $\beta>0$ and $[\Xq{}]^*$. But abusing notation in this way lets us write the result in a simpler way.

Returning to the context we set up at the beginning of this section, we suppose again that each bundle
$L_i\in \F$ has an equivariant algebraic section $s_i$, with corresponding $C_2$-invariant hypersurface $H_i$.

\begin{corollary}\label{cor:genericallyTransverse}
If the hypersurfaces $H_i$ intersect transversely, then
\[
 \Bigl[\bigcap_{i=1}^n H_i\Bigr]^*_\eta = \frac{d-A-B}{2}[C_2]^*_\eta + A[\Xp{}]^*_\eta + B[\Xq{}]^*_\eta
\]
where the pair of integers $(A,B)$ is given by
\[
(A,B)=
\begin{cases}
 \left(d_{\I}d_{\II},d_{\II}d_{\III} \right)  \quad& \text{ for }\alpha=0,\ \beta=0 \\
 \ \ \left(d_{\I}d_{\II},0 \right)  \quad& \text{ for }\alpha=0,\ \beta>0 \\
 \ \ \left(0,d_{\II}d_{\III} \right)  \quad &\text{ for }\alpha>0,\ \beta=0\\
 \ \ \ \ (0,0)   \quad& \text{ for }\alpha>0,\ \beta>0.
\end{cases}
\]
\qed
\end{corollary}

Corollary~\ref{cor:genericallyTransverse} describes the intersection of the hypersurfaces as a $C_2$-set mapping into $\Xpq pq$.
This $C_2$-set matches what we should expect: Nonequivariantly, it contains $d$ points, which is the number of points of intersection
that Bezout's theorem tells us there should be.
Now look at what happens when we take $C_2$-fixed sets of the line bundles: If $L$ is a bundle of type $\III$ or $\IV$,
then $L^{C_2}$ is the zero bundle over $\Xp p$, while if $L$ is of type $\I$ or $\II$, then $L^{C_2}$ has dimension 1 over $\Xp p$.
If $\alpha = 0$, so that there are exactly $p-1$ line bundles of type $\I$ and $\II$, Bezout's theorem then tells us that
the intersection of the fixed points of the hypersurfaces should give us exactly $d_\I d_\II$ fixed points in $\Xp p$,
which is what the corollary says. If $\alpha>0$, then there should be no points in $\Xp p$, again agreeing with the corollary.
A similar analysis shows that, when $\beta = 0$, we expect to get $d_\II d_\III$ fixed points in $\Xq q$,
and when $\beta > 0$ we expect none, as in the corollary.

Bezout's theorem also applies nonequivariantly when the hyperplanes are not necessarily transverse but are only known
to be  dimensionally transverse, provided we count things with proper multiplicities.
We propose here a set of definitions to generalize these ideas to the equivariant context and show an analogue
of Corollary~\ref{cor:genericallyTransverse} in this context.

\begin{definition}
Suppose that $M$ is a smooth $C_2$-manifold and $N\subset M$ is a smooth $C_2$-submanifold. We define the {\em $C_2$-codimension}
of $N$ in $M$ to be the representation $\nu(N,M)\in RO(\Pi_{C_2}N)$ associated with the normal bundle of $N$ in $M$.
\end{definition}

Note that the equivariant codimension captures not just the difference between the nonequivariant dimensions, but on taking fixed points
also gives the nonequivariant codimensions of the components of $N^{C_2}$ in (the components of) $M^{C_2}$.

\begin{definition}
Let $M$ be a smooth complex $C_2$ manifold and let $E\to M$ be a complex $C_2$-vector bundle over $M$.
A section $s$ of $E$ is said to be {\em dimensionally transverse} to the zero section if $s^{-1}(0)$ is a smooth submanifold
of codimension $E|s^{-1}(0)$.
\end{definition}

\begin{examples}
In these examples, we label the homogeneous coordinates on $\Xpq pq$ as $[x_1:\ldots:x_p:y_1:\ldots:y_q]$.
\begin{enumerate}
    \item The section of $L = O(2)$ defined by the polynomial $x_1^2$ is {\em not} dimensionally transverse to the zero section.
    We have $s^{-1}(0) = \Xpq{p-1}{q}$, which has the correct nonequivariant codimension. However, its fixed sets are
    \[
        \Xpq{p-1}{q}^{C_2} = \Xp{p-1} \disjunion \Xq q.
    \]
    The first component $\Xp{p-1}$ has the correct codimension of 1 in $\Xp p$, but the second component $\Xq q$ has
    codimension 0 in itself, where it would need to have codimension 1 to match the dimension of $L^{C_2}|\Xq q$.
    
    \item The section of $L = O(2)$ defined by $\sum_{i=1}^p x_i^2 + \sum_{j=1}^q y_j^2$ is dimensionally transverse to the zero section.
    If $Z(s) = s^{-1}(0)$ is the zero locus, then $Z(s)^{C_2}$ is the union of the zero locus of $\sum_{i=1}^p x_i^2$ in $\Xp p$ and the zero locus
    of $\sum_{j=1}^q y_j^2$ in $\Xq q$, both of which have the correct codimension of~1.
    
    \item Suppose $F$ is the sum of $p+q-1$ line bundles over $\Xpq pq$ with a section $s$ such that $Z(s)$ is a discrete set of points.
    Then, in order for $s$ to be dimensionally transverse, it is necessary and sufficient that: (i)   
    $Z(s)$ has no fixed points in $\Xp p$ in the case that $\alpha>0$ and (ii) it has no fixed points in $\Xq q$ in the case that $\beta > 0$.
    
\end{enumerate}
\end{examples}

Consider again a smooth complex $C_2$-manifold $M$ with a complex $C_2$-vector bundle $E\to M$.
Suppose that we have a section $s$ of $E$ that is dimensionally transverse to the zero section, with zero locus $Z = Z(s)$.
We can then construct a map 
\[
Th_Z(s): Th_Z(\nu(Z,M)) \longrightarrow Th_Z(E|Z),
\]
as follows,
where $Th_Z$ denotes the fiberwise one-point compactification over $Z$, which is the
appropriate Thom space when working with parametrized spaces.
Take a tubular neighborhood $N$ of $Z$, so $\pi\colon N\to Z$ is equivalent to the disc bundle $D(\nu(Z,M))$,
and $s$ is nonzero on $\bndry N$. 
A technical but necessary point: $E|N \iso \pi^*(E|Z)$ because $\pi$ is a homotopy equivalence,
hence, on $N$, we can replace $s$ with the composite
\[
 s'\colon N \xrightarrow{s} E|N \xrightarrow{\iso} \pi^*(E|Z)
\]
which is a map over $Z$.
Now $s'(\bndry N)$ falls outside of some
small radius disc bundle in $\pi^*(E|Z)$, hence induces a map
\[
 Th_Z(s)\colon Th_Z(\nu) \iso N/_Z\bndry N \to Th_Z(E|Z).
\]
(Here, $/_Z$ indicates fiberwise quotient over $Z$.)

\begin{definition}\label{def:multFunction}
Let $M$ be a smooth complex $C_2$ manifold and let $E\to M$ be a complex $C_2$-vector bundle over $M$.
Suppose that $s$ is a section of $E$ that is dimensionally transverse to the zero section, with zero locus $Z = Z(s)$.
Let $\eta \in RO(\Pi Z)$ denote the representation associated to the normal bundle $\nu(Z,M)$, which is assumed
equal to that associated to the bundle $E|Z$.
We define the {\em multiplicity function} $\mu(s) \in H_{C_2}^0(Z_+)$ to be the image of 1 under the map
\[
    H_{C_2}^0(Z_+) \iso H_{C_2}^\eta(Th_Z E|Z) \xrightarrow{Th_Z(s)^*} H_{C_2}^\eta(Th_Z (\nu(Z,M))) \iso H_{C_2}^0(Z_+),
\]
where the first and last maps are Thom isomorphisms, as in \cite[Theorem~3.11.3]{CostenobleWanerBook}.
\end{definition}

We call the class $\mu(s)$ the ``multiplicity function'' thinking of the map
\[
 \widetilde{\mu(s)}= \langle \mu(s), - \rangle \colon H_0^{C_2}(Z_+) \to A(C_2),
\]
defined by evaluation,
which assigns an element of the Burnside ring to each component of $Z$ and to each component of the fixed set of $Z$.
In more detail, we can identify $H^{C_2}_0(Z_+)$ with the Grothendieck group of homotopy classes of finite $C_2$-sets mapping to $Z$
with addition given by disjoint union. Thus, $\widetilde{\mu(s)}$ is a function associating an element of the
Burnside ring to each $C_2$-orbit $z\colon C_2/K\to Z$. We think of the element $\widetilde{\mu(s)}(z) \in A(C_2)$
as the multiplicity of $s$ at $z$.

Note that, if $s$ is actually transverse to the zero section, then we can use its derivative to identify $\nu(Z,M)$
and $E|Z$, the map $Th_Z(s)^*$ is then an isomorphism, and $\mu(s) = 1$.

Now let $i\colon Z\to M$ denote the inclusion. Let $\eta\in RO(\Pi M)$ denote the representation associated to $E\to M$,
so that its restriction to $RO(\Pi Z)$ is the class we have also been calling $\eta$.

\begin{definition}\label{def:fundclasswithmult}
We call the class
\[
 \{Z(s)\}^*_\eta := i_!\mu(s) = i_!(\mu(s)[Z(s)]^*_0) \in H_{C_2}^{\eta}(M_+)
\]
the {\em fundamental cohomology class of $Z(s)$ counted with multiplicities}.
\end{definition}

We use the notation $\{Z(s)\}^*_\eta$ to distinguish this class from the fundamental cohomology class $[Z(s)]^*_\eta$ of 
$Z(s)$ just considered as a submanifold of $M$. 
If $s$ is actually transverse to the zero section, we have $\{Z(s)\}^*_\eta = [Z(s)]^*_\eta$.
We also have the following characterization of the Poincar\'e dual:
\[
 \{Z(s)\}^*_\eta \cap [M] = i_!\mu(s) \cap [M] = i_*(\mu(s)\cap [Z(s)])
\]

Although we propose this as the correct way of counting multiplicities in general, we shall use it here only in the case
where $Z(s)$ is discrete. In that case, we can be much more explicit.
So now assume that $Z = Z(s)$ is a discrete $C_2$-space, that is, a finite $C_2$-set
\[
 Z = \Disjunion_k Z_k= \Disjunion_k C_2/K_k
\]
where $K_k$ is either $C_2$ or $\{e\}$ for each $k$, hence each $Z_k$ is a single embedded orbit. Then we can identify
\[
 H^{C_2}_0(Z_+) \iso \Dirsum_k A(K_k).
\]
This is a module over $A(C_2)$, with the action on a summand of the form $A(C_2)$ being multiplication and
the action on $A(e) = \Z$ having $[C_2]$ acting as $2$.
The cohomology is
\[
 H_{C_2}^0(Z_+) \iso \Hom_{A(C_2)}(H^{C_2}_0(Z_+),A(C_2)),
\]
the isomorphism being the adjoint of the evaluation map.
Note that an $A(C_2)$-homomorphism $A(e) \to A(C_2)$ is determined by the image of $1$, which must be of the form $n[C_2]$ for some $n\in\Z$.
In this situation, the multiplicity function $\mu(s)$ is completely determined by its value on each orbit $Z_k$.
Returning to the main context of this section, we can describe the function as follows.

\begin{proposition}\label{prop:multiplicityCalc}
Let $F$ be a $(p+q-1)$-dimensional complex vector bundle over $\Xpq pq$. Assume that we have an
equivariant section $s$ that is dimensionally transverse to the zero section, so its zero locus $Z(s) = \Disjunion_k Z_k$ is a finite $C_2$-set.
Take a neighborhood $N_k$ of each orbit $Z_k$ consisting of a ball small enough that it does not contain any other zeros.
Then, for each $k$, the section $s$ determines a map of spheres $\bar s_k\colon N_k/\bndry N_k\to N_k/\bndry N_k$, which represents an element
$[\bar s_k] \in A(C_2)$, and $\widetilde{\mu(s)}([Z_k]) = [\bar s_k]$.
\end{proposition}

\begin{proof}
The maps $\bar s_k$ are the restrictions of the map $Th_Z(s)$ used in Definition~\ref{def:multFunction}.
The proposition then follows from that definition applied to the case where $Z(s)$ is discrete.
In the case where $Z_k$ is a free orbit, note that the $C_2$-set we get is also free.
\end{proof}

We could also use Proposition~\ref{prop:transversality}(2) to realize each $\widetilde{\mu(s)}[Z_k]$ as a virtual $C_2$-set mapping to $Z(s)$,
by deforming $s$ to a section $s'$ actually transverse to the zero section and equal to $s$ outside the union of the balls $N_k$.
The zeros of $s'$ in $N_k$, counted with signs, also then give $\widetilde{\mu(s)}[Z_k]$.

\begin{proposition}
Let $F$ be a sum of $p+q-1$ line bundles $L_i$ over $\Xpq pq$, each of which has an algebraic section $s_i$ defining a
hypersurface $H_i$. Assume further that the $H_i$ intersect dimensionally transversely, meaning here that
the sum of sections $s = \Dirsum s_i$ is dimensionally transverse to the zero section of $F$.
As before, write $Z(s) = \bigcap_i H_i = \Disjunion_k Z_k$ where each $Z_k$ is an orbit.
\begin{enumerate}
    \item If $Z_k$ is a fixed orbit, then $\widetilde{\mu(s)}[Z_k]$ is determined as follows. Its nonequivariant restriction $\rho(\widetilde{\mu(s)}[Z_k])\in\Z$
    is the nonequivariant intersection number of the $H_i$ at $Z_k$, while the number of fixed points $(\widetilde{\mu(s)}[Z_k])^{C_2}\in\Z$ is the 
    intersection number in $\Xpq pq^{C_2}$ of those $H_i$ for which $(L_i^{C_2})_{Z_k} \neq 0$.
    
    \item If $Z_k$ is a free orbit, then $\widetilde{\mu(s)}[Z_k] = n[C_2]$ where $n$ is the intersection number of the $H_i$ at
    either of the two points in the orbit.
\end{enumerate}
In the first case, the total number of points and the number of fixed points will both be nonnegative,
and there will be at least as many points total as there are fixed points.
In the second case, the number $n$ will be nonnegative. In either case this means that the element $\widetilde{\mu(s)}[Z_k] \in A(C_2)$
is represented by an actual $C_2$-set.
\end{proposition}

\begin{proof}
The descriptions of the total number of points and the number of fixed points follows from the previous proposition
together with the equality of the topological and algebraic methods of calculating multiplicities, as shown, for example,
in \cite[\S5.2]{GriffithsHarris}.
For the last statements, the intersection numbers are nonnegative, so what we need to verify is that, in case (1),
the total number of points is no less than the number of fixed points.
Assume that $Z_k\in \Xp p$ and, without loss of generality, that the line bundles are ordered so that
$L_1^{C_2}$ through $L_{p-1}^{C_2}$ have nonzero fibers at $Z_k$ while $L_p^{C_2}$ through $L_{p+q-1}^{C_2}$ have zero fibers at $Z_k$.
If the sections $s_i$ are defined by polynomials $f_i$, then, if $\Oscr:=\Oscr_{\Xpq pq,\,Z_k}$ is the local ring at $Z_k$ in $\Xpq pq$ and $\Oscr^{C_2}:=\Oscr_{\Xp{p},\,Z_k}$
is the local ring at $Z_k$ in $\Xp p$, the total number of points in
$\mu(s)[Z_k]$ is given by the height of $\Oscr/(f_1,\ldots,f_{p+q-1})$ while the number of fixed points
is given by the height of $\Oscr^{C_2}/(f_1,\ldots,f_{p-1})$.
However, the latter is equal to the height of $\Oscr/(f_1,\ldots,f_{p-1},y_1,\ldots,y_q)$ where, as earlier, we write the
homogeneous coordinates of $\Xpq pq$ as $[x_1:\ldots:x_p:y_1:\ldots:y_q]$. But the height
of $\Oscr/(f_1,\ldots,f_{p-1},y_1,\ldots,y_q)$ is less than or equal to that of $\Oscr/(f_1,\ldots,f_p,f_{p+1},\ldots,f_{p+q-1})$,
so the claim follows.
Finally, an element $a \in A(C_2)$ is represented by an actual $C_2$-set if and only if $0\leq a^{C_2}\leq \rho(a)$.
\end{proof}

The import of this result is that, under the assumptions of the proposition, the class $\{Z(s)\}^*_\eta\in H_{C_2}^\eta(\Xpq pq_+)$
is the push-forward of an actual $C_2$-set mapping to $M$.
We form a $C_2$-set $\mu(s)Z(s)$ by replacing each orbit $Z_k$ in $Z$ with the actual $C_2$-set representing $\widetilde{\mu(s)}[Z_k]$.
We then have a $C_2$-map
\[
 j\colon \mu(s)Z(s) \to Z(s) \includesin M,
\]
and $\{Z(s)\}^*_\nu = j_![\mu(s)Z(s)]^*_0)$.

Finally, we can state the following generalization of Corollary~\ref{cor:genericallyTransverse}.

\begin{theorem} \label{Theorem BezoutMult}
With the assumptions of Corollary~\ref{cor:genericallyTransverse}, but assuming now only that the hypersurfaces $H_i$
intersect dimensionally transversely, \textit{i.e.} the associated section is dimensionally transverse to the zero section, we have that
\[
 \Bigl\{\bigcap_{i=1}^n H_i\Bigr\}^*_\eta = \frac{d-A-B}{2}[C_2]^*_\eta + A[\Xp{}]^*_\eta + B[\Xq{}]^*_\eta
\]
where the pair of integers $(A,B)$ is given by
\[
(A,B)=
\begin{cases}
 \left(d_{\I}d_{\II},d_{\II}d_{\III} \right)  \quad& \text{ for }\alpha=0,\ \beta=0 \\
 \ \ \left(d_{\I}d_{\II},0 \right)  \quad& \text{ for }\alpha=0,\ \beta>0 \\
 \ \ \left(0,d_{\II}d_{\III} \right)  \quad &\text{ for }\alpha>0,\ \beta=0\\
 \ \ \ \ (0,0)   \quad& \text{ for }\alpha>0,\ \beta>0.
\end{cases}
\]
\end{theorem}

\begin{proof}
This follows from Theorem~\ref{thm:genericallyTransverse} on approximating $s$ by a section that is actually transverse
as in Proposition~\ref{prop:transversality}(2)
and appealing to Proposition~\ref{prop:multiplicityCalc} to identify the result
with the fundamental cohomology class counted with multiplicities.
\end{proof}

\section*{Appendix A: Glossary of notations} \label{App: A}

We provide here a brief glossary of notations that may be unfamiliar.

\begin{description}\setlength{\itemsep}{7pt}
    \item [$\R^{\sigma}$] The nontrivial irreducible real representation of $C_2$. The corresponding element of the representation ring $RO(C_2)$ is usually written simply as $\sigma$.
    
    \item [$\C^{\sigma}$] The nontrivial irreducible complex representation of $C_2$. As an element of the real representation ring $RO(C_2)$, this is $2\sigma$.

    \item [$A(G)$] The Burnside ring of the finite group $G$, the Grothendieck group of finite $G$-sets under disjoint union. See Section~\ref{subsec:cohompoint}.
    
    \item [$\kappa \in A(C_2)$] The element $\kappa = 2 - [C_2]$ generates the kernel of the dimension map $A(C_2)\to \Z$, and $A(C_2) \iso \Z[\kappa]/\langle \kappa^2 - 2\kappa \rangle$. See Section~\ref{subsec:cohompoint}.
    
    \item [$\Mackey A$] The Burnside ring Mackey functor, the universal coefficient system for $RO(C_2)$-graded cohomology, and the coefficient system used throughout the paper unless otherwise specified. See Section~\ref{subsec:cohompoint}.
    
    \item [$\Mackey H_{C_2}^{RO(C_2)}(X) = \Mackey H_{C_2}^{RO(C_2)}(X;\Mackey A)$] The $RO(C_2)$-graded Mackey functor-valued cohomology of the based $C_2$-space $X$ with coefficients in $\Mackey A$. See Section~\ref{subsec:cohompoint}.
    
    \item [$H\Mackey A$] The Eilenberg-Mac\,Lane $C_2$-spectrum representing $\Mackey H_{C_2}^{RO(C_2)}(X;\Mackey A)$. See Section~\ref{subsec:cohompoint}.

    \item [$\Mackey \Z$, $\Mackey \Z_-$, $\Mackey \Z'$, $\Mackey \Z'_-$, $\conc{\Z}$,  $\conc{\Z/2}$] Particular Mackey functors that appear in the $C_2$-cohomology of a point. See Section~\ref{subsec:cohompoint}.
    
    \item [$\iota$, $\xi$, $ e$, 
    $e^{-m}\kappa$,  $e^{-m}\tau(\iota^{-(2k+1)})$] Generators of the $C_2$-cohomology of a point. See Section~\ref{subsec:cohompoint}. See also Remark~\ref{rem:othernotations} for the relationship with other notations used in the literature for some of these elements.
    
    \item [$\Pi_{C_2}X = \Pi X$] The equivariant fundamental groupoid of the $C_2$-space $X$. See Section~\ref{sub ROPi}.
    
    \item [$RO(\Pi X)$] The real representation ring of the fundamental groupoid of the $C_2$-space $X$. See Section~\ref{sub ROPi}.

    \item [$\Mackey H_{C_2}^{RO(\Pi B)}(X) = \Mackey H_{C_2}^{RO(\Pi B)}(X;\Mackey A)$] The $RO(\Pi B)$-graded Mackey functor-valued cohomology of the $C_2$-ex-space $X$ over $B$. See Section~\ref{sub ROPi}.
    
    \item [$H\Mackey A^\gamma$] The $C_2$-spectrum parametrized by $B$ representing $\Mackey H_{C_2}^{\gamma+RO(C_2)}(X)$ for $C_2$-ex-spaces $X$ over $B$ and $\gamma\in RO(\Pi B)$. See Section~\ref{sub ROPi}.
    

    \item [$\rho^*$] The restriction map from $C_2$-cohomology to nonequivariant cohomology. See Section~\ref{subsec:restrictions}.
    
    \item [$(-)^{C_2}$] The fixed-point map from $C_2$-cohomology of $X$ to the nonequivariant cohomology of $X^{C_2}$. See Section~\ref{subsec:restrictions}.
    
    \item [$\Phi^{C_2}$] The geometric fixed-point functor on $C_2$-spectra. It represents the fixed-point map $(-)^{C_2}$. See Section~\ref{subsec:restrictions}.
    
            
    \item [$\Xpq{p}{q}$] The $C_2$-equivariant projective space of complex lines in $\C^p\dirsum(\C^\sigma)^q$, where $0\leq p\leq\infty$ and $0\leq q\leq\infty$.
    
    \item [$B$] Used in most of the paper as shorthand for $\Xpq\infty\infty$.
    
     \item [$\chi$] The $C_2$-involution of $\Xpq{\infty}{\infty}$ that classifies the operation of tensoring a complex line bundle with $\C^\sigma$.
    
    \item [$\omega$] The tautological bundle over $\Xpq pq$. Also, the corresponding representation in $RO(\Pi \Xpq pq)$.
    
    \item [$\chi\omega$] The bundle $\omega\tensor_\C \C^\sigma$ over $\Xpq pq$. Also, the corresponding representation in $RO(\Pi \Xpq pq)$.
    
    \item [$\cw$, $\cxw$] The Euler classes of the bundles $\omega$ and $\chiw$, respectively.
    
    \item [$\cwd$, $\cxwd$] The Euler classes of the dual bundles to $\omega$ and $\chiw$, respectively.
    
    \item [$\cwt$, $\cxwt$] Elements in $\Mackey H_{C_2}^{RO(\Pi B)}(B_+)$, $B = \Xpq\infty\infty$, needed as algebraic generators along with $\cw$ and $\cxw$.

    \item [$f_!$] The push-forward map associated to a map between $C_2$-manifolds. See Definition~\ref{def:pushforward}.
    
   \item [${[M]_\eta^*}$] The fundamental cohomology class of $M$ in $N$. See Definition~\ref{def:fundcohomologyclass}.
    
    \item [$\{Z(s)\}^*_\eta$] The fundamental cohomology class of $Z(s)$ counted with multiplicities, where $Z(s)$ is the zero locus of a dimensionally transverse section of a complex vector bundle over a smooth manifold. See Definition~\ref{def:fundclasswithmult}.

\end{description}

\section*{Appendix B: Table of restrictions} \label{App: B}

Because our calculations depend on knowing the restrictions to nonequivariant cohomology and the fixed points of various
elements, we collect those values here. The first table involves elements of $\Mackey H_{C_2}^{RO(C_2)}(S^0)$
while the second has elements of $\Mackey H_{C_2}^{RO(\Pi B)}(\Xpq pq_+)$.

\setlength{\tabcolsep}{1em}
\renewcommand{\arraystretch}{2}
 
\[
\begin{tabular}{|c||c|c|c|c|c|c|}
\hline
$x$ & $\kappa$ & $\xi$ & $e$ & $e^{-m}\kappa$ & $e^{-m}\tau(\iota^{-(2k+1)})$ & $1-\kappa$  \\
\hline
$\rho^*(x)$ & $0$ & $\iota^2$ & $0$ & $0$ & $0$ & $1$\\
\hline
$x^{C_2}$ & $2$ & $0$ & $1$ & $2$ & $0$ & $-1$\\
\hline
\end{tabular}
\]

\[
\begin{tabular}{|c||c|c|c|c|c|c|}
\hline
$x$ & $\cxwt$ & $\cwt$ & $\cwd$ & $\cxwd$ & $1-\epsilon$ \\
\hline
$\rho^*(x)$ & $\iota^2\zeta^{-1}$ & $\zeta$ & $\zeta c$ & $\iota^2\zeta^{-1}c$ & $1$ \\
\hline
$x^{C_2}$ & $(0,1)$ & $(1,0)$ & $(c,1)$ & $(1,c)$ & $(1,-1)$ \\
\hline
\end{tabular}
\]


\bibliography{Bibliography}{}
\bibliographystyle{amsplain} 
\end{document}